\documentclass[11pt,reqno]{amsart}
\usepackage{amssymb}
\usepackage{amsmath,amsthm}
\usepackage{color}
\usepackage{graphics, cite, bookmark, float}
\usepackage{epsfig}
\usepackage{hyperref}
\hypersetup{hidelinks}
\usepackage{epstopdf}
\usepackage[normalem]{ulem}
\newcommand\redout{\bgroup\markoverwith
{\textcolor{red}{\rule[.5ex]{2pt}{0.4pt}}}\ULon}
\usepackage{xcolor}
\usepackage[utf8]{inputenc}
\usepackage{setspace}
\usepackage{ulem}
\usepackage{cancel}
\usepackage[margin=1.05in]{geometry}

\newtheorem{theorem}{Theorem}[section]
\newtheorem*{theorem*}{Theorem}
\newtheorem{proposition}[theorem]{Proposition}
\newtheorem{lemma}[theorem]{Lemma}

\theoremstyle{definition}

\theoremstyle{remark}

\numberwithin{equation}{section}
\allowdisplaybreaks

\emergencystretch=\maxdimen
\newcommand{\dis}{\displaystyle}
\newcommand{\R}{\mathbb{R}}

\newcommand{\M}{\mathcal{M}}
\newcommand{\dt}{\partial_t}
\newcommand{\dx}{\partial_x}

\newcommand{\drho}{\partial_\rho}
\newcommand{\dthe}{\partial_\theta}

\newcommand{\va}{{\varphi}}
\newcommand{\eps}{\varepsilon}
\newcommand{\vsi}{\varsigma}

\title[The Isometric Immersion of  Surfaces]
{
The Isometric Immersion of Negatively Curved Surfaces with Finite Total Curvature}

\author[W. Cao, Q. Han, F. Huang and D. Wang]
{Wentao Cao \and Qing Han \and Feimin Huang \and Dehua Wang}

\address{Academy for Multidisciplinary Studies, Capital Normal University, Beijing, 100048, P.R. China.} 

\email{cwtmath@cnu.edu.cn}

\address{Department of Mathematics,
University of Notre Dame, Notre Dame, IN 46556, USA.}
\email{qing.han.7@nd.edu}

\address{Institute of Applied Mathematics,
   AMSS, Chinese Academy of Sciences, Beijing 100190, China.}
\email{fhuang@amt.ac.cn}

\address{Department of Mathematics, University of Pittsburgh, Pittsburgh, PA 15260, USA.}
\email{dwang@math.pitt.edu}

\begin{document}

\date{}

\begin{abstract}
	
In this paper, we study the smooth isometric immersion of a complete, simply connected surface with a negative Gauss curvature into the three-dimensional Euclidean space. 
A fundamental and longstanding problem  is to find a sufficient condition for a complete negatively curved surface to be isometrically embedded in $\R^3$ \cite{Yau}. 
It can be described as an initial and/or boundary value problem for a hyperbolic system of nonlinear partial differential equations derived from the Gauss-Codazzi equations. The mathematical theory associated with this system is largely incomplete. The global smooth isometric immersion has been proven in the literature when the Gauss curvature decays rapidly and monotonically. However, when the Gauss curvature oscillates or decays slowly, the problem becomes much more challenging and little is known. 

 In our paper, we find a sufficient condition, consisting of a finite total Gauss curvature and appropriate oscillations of the Gauss curvature. Under this condition we prove the global existence of a smooth solution to the Gauss-Codazzi system, achieving a global smooth isometric immersion of the surface into the three-dimensional Euclidean space. Furthermore, we show that the finite total Gauss curvature is necessary for the existence of a solution in a special case of the Gauss-Codazzi system.

New techniques are developed to overcome the difficulties posed by the slow decay and oscillations of the Gauss curvature. By observing that certain combinations of the Riemann invariants decay faster than others, we reformulate the Gauss-Codazzi equations as a symmetric hyperbolic system and uncover a crucial structure of partial dampings. These partial dampings, along with the finite total curvature and appropriate oscillations of the Gauss curvature, enable us to obtain a global smooth solution through delicate analysis, and consequently establish a global smooth isometric immersion of such surfaces.

\end{abstract}



\maketitle
\tableofcontents

%%%%%%%%%%%%%%%%%%%%%%%%%%%%%%%%%%%%%%%%%%%%%%%%%%
\section{Introduction}
%%%%%%%%%%%%%%%%%%%%%%%%%%%%%%%%%%%%%%%%%%%%%%%%%%%%%%%%%%%%%%%%%
%

In this paper, we study the global smooth isometric immersion in the three-dimensional Euclidean space $\R^3$
of two-dimensional complete simply connected Riemannian manifolds (or surfaces) with negative Gauss curvatures.  The classical theory of surfaces asserts that the isometric immersion of such surfaces exists if the Gauss-Codazzi system is satisfied (cf.  \cite{BS, Carmo, HH, PS, Mainardi}).
Let $(\M, g)$ be a surface with the metric
$$g=\sum_{i,j=1}^2g_{ij}dx_idx_j.$$
Assume that $(\M, g)$ is isometrically immersed in $\mathbb R^3$
and its second fundamental form is given by
\begin{equation*}
 I\!\!I= Ldx_1^2+2Mdx_1dx_2+Ndx_2^2.
\end{equation*}
Then, the Gauss-Codazzi system is of the following form:
\begin{equation}\label{e:GC}
\begin{split}
& \partial_{2}L-\partial_{1}M=\Gamma^1_{12}L+(\Gamma^2_{12}-\Gamma^1_{11})M-\Gamma^2_{11}N,\\
 &\partial_{2}M-\partial_{1}N=\Gamma^1_{22}L+(\Gamma^2_{22}-\Gamma^1_{21})M-\Gamma^2_{21}N,
 \end{split}
 \end{equation}
with
 \begin{equation}\label{e:Gauss}
 LN-M^2=K(g_{11}g_{22}-g_{12}^2),
\end{equation}
where $\partial_i$ denotes $\partial_{x_i},$
$\Gamma^i_{jk}\, (i,j,k=1,2)$ are the Christoffel symbols of the given metric $g$,  and $K$ is the Gauss curvature.
To establish the existence of global isometric immersion of a surface, we need to find a global solution to the above Gauss-Codazzi system.

The isometric embedding or immersion of Riemannian manifolds is a well-known fundamental problem in differential geometry.
For general $n$-dimensional Riemannian manifolds, Janet \cite{Janet} and Cartan \cite{Cartan} proved
the local analytic isometric embedding in $\R^{n_*}$, with $n_*$ denoting the Janet dimension $n(n+1)/2$.
Nash \cite{nash1956} established the smooth global isometric embedding of smooth $n$-dimensional manifold
in $\R^m$ for sufficiently large $m$. Gromov \cite{gromov70} and G\"unther \cite{Gun89, Gun90} improved the target dimension.
In fact, G\"unther provided an alternative proof
and improved the target dimension to $m\geq\max\{n_*+2n, n_*+n+5\}$. % Gromov \cite{gromov70}
On the other hand, Nash \cite{nash1954} and Kuiper \cite{Kui55} obtained the flexibility of
$C^1$ isometric embedding in $\R^m(m\geq n+1)$, which yields $C^1$ isometric embeddings
of $n$-dimensional Riemannian manifold into $\R^{2n}$ with Whitney's strong embedding theorem (cf.\cite{Wh44}).
The Nash-Kuiper theorem also holds for $C^{1,\theta}$ isometric embeddings with $\theta<(1+n+n^2)^{-1}$
for $n\geq 3$ (cf. \cite{CS19, CDS12}) and with $\theta<1/5$ for $n=2$ (cf. \cite{DIS15, CS19}).

It is a challenging problem to find smooth or sufficiently smooth local isometric embeddings
of $n$-dimensional Riemannian manifolds in the Euclidean spaces with the Janet critical dimension $n(n+1)/2$.
This is the Schl\"afli conjecture (cf.\cite{HH}).
Regarding the local and global smooth isometric immersion in the case of $n=2$,
the book by Han-Hong \cite{HH} provides an extensive review and discussions.
For $n=2$, the Gauss-Codazzi system for the isometric immersion of surfaces in $\R^3$ is elliptic for the positive curvature $K>0$,
is hyperbolic for the negative curvature $K<0$, and is of mixed type for the Gauss curvature changing signs.
For the curvature $K\ge 0$, there have been many results in the literature on local and global smooth embedding, and we 
refer the reader to \cite{Han, hhl, HK2, hw, Lin, Guan, gl, he, hz, n, w}, the book \cite{HH}, and their references.
When the curvature changes signs, sufficiently smooth local isometric embeddings of surfaces are obtained
in \cite{Lin0, Han, Khuri, HK1, Dong} with additional conditions on $\nabla K$.
When the curvature is negative, the smooth isometric immersion was studied in \cite{H, HH, CHW-smooth}, and
the $C^{1, 1}$ isometric immersion was obtained in \cite{CHW, CHW1, CSW, Christoforou, cs, LiS2020}.
The isometric immersion of surfaces with low regularity was also investigated in \cite{M1, sz, Chen-Li2018}.
For $n=3$ or $n=4$, the local smooth isometric immersion was studied in \cite{bgy, GY, NM, NM2, Poole, CCSWY1}.

 %%%%%%%%%%%%%%%%%%%%

This paper  focuses  on the {\em global smooth} isometric immersion in $\R^3$
of complete negatively curved surfaces, i.e., complete two-dimensional Riemannian manifolds with negative Gauss curvatures.
An example of such surfaces is given by the hyperbolic plane, with the Gauss curvature being $-1$.
The first relevant result was due to Hilbert \cite{hi},
who proved that the hyperbolic plane does not admit smooth isometric immersions in $\mathbb R^3$.
Efimov  \cite{Efimov1, Efimov2} generalized this result to various classes of complete negatively curved surfaces,
and found different conditions on the Gauss curvature under which no isometric immersions in $\R^3$ exist. More precisely, in \cite{Efimov1} and  \cite{Efimov2}, respectively, Efimov showed that there is no $C^3$ isometric immersion in $\R^3$ if the Gauss curvature is bounded above by a negative constant, or if the Gauss curvature $K$ is negative and satisfies
\begin{equation}\label{eq-Efimov}\sup |K|, \, \sup|\nabla |K|^{-1/2}|<\infty.\end{equation}
This condition allows $|K|$ to decay at the inverse quadratic order of the geodesic distance at infinity.
Specifically, if
$K=-\rho^{-2}$ for all large $\rho$ in some geodesic polar coordinates $(\theta, \rho)$ on $(\mathcal{M}, {g})$, %for all large $\rho$,
then \eqref{eq-Efimov} is satisfied and hence
there is no $C^3$ isometric immersion.

Concerning the existence of isometric immersions of
the negatively curved surfaces, Yau  \cite{Yau} proposed the following problem: %``
{\it Find a sufficient condition for a complete negative curved surface to be isometrically embedded in $\R^3$}.  He indicated that reasonable conditions should involve decay rates of the Gauss curvature at infinity.
Obviously, such decay rates should be faster than the inverse square of the geodesic distance at infinity.
Along this direction, Hong \cite{H} gave an affirmative answer and obtained the following result.

\smallskip
\smallskip

\noindent
{\bf Theorem A \cite{H}}.
{\it
Let $(\mathcal{M}, {g})$ be a smooth complete simply connected 2-dimensional surface with a negative Gauss curvature
$K$. %=-k^2(\theta, \rho).$
Assume that in some geodesic polar coordinates $(\theta, \rho)$ on $(\mathcal{M}, {g})$,
$K$ satisfies, for some constant $\gamma>0$,
\begin{equation}\label{c1}
\partial_\rho\log(|K|\rho^{2+\gamma})\leq0 \text{ for large } \rho,
\end{equation}
and
\begin{equation}\label{eq-a-hong}
	\partial_\theta^i\log |K|~(i=1,2), ~~\rho\partial_\theta\partial_\rho\log |K| \text{ are bounded.}
	\end{equation}
Then, $(\mathcal{M}, {g})$ has a smooth isometric immersion in $\mathbb{R}^3.$}
%\end{theorem*}
\smallskip

The condition (\ref{c1}) requires that the Gauss curvature $K$ itself should be decreasing along all geodesic rays
and do not allow oscillating factors.
It is a longstanding  open problem {\em whether the decay assumption \eqref{c1} can be relaxed and oscillations can be permitted in the Gauss curvature.}

\subsection{The main result}
In this paper, we will prove the existence of global isometric immersions of surfaces in $\R^3$ with a finite total Gauss curvature
and appropriate
oscillations of the Gauss curvature.
More precisely, we have the following theorem.

\begin{theorem}
\label{thrm-main1}
Let ($\mathcal{M},g$) be a smooth complete simply connected surface with a negative Gauss curvature $K$ and
\begin{equation}\label{e:inte-cond}
	\int_{\mathcal{M}}|K|dA_g<\infty,
\end{equation}
where $dA_g$ is the area element of $g$.
Assume that in some geodesic polar coordinates $(\theta, \rho)$ on $(\mathcal{M}, g)$,
$K$ has the decomposition
\begin{equation}\label{eq-decomposition}
\rho^{2+\gamma}|K|(\theta, \rho)=\overline{K}(\rho)a^2(\theta, \rho)\quad\text{for $\rho$ large},
\end{equation}
where $\gamma\in(0, 1)$ is a constant and $\overline{K}$ and $a$ are positive functions such that
\begin{equation}\label{eq-K-bar-bound}
\overline{K}(\rho)\text{ is monotone for $\rho$ large},
\end{equation}
and
\begin{align}
a,\, a^{-1}, \,\partial^i_\theta\log a, \, &\rho\partial_\theta^i\partial_\rho\log a \text{ are bound for }i=1, 2, 3,\label{eq-a-bul}\\
&\int_{1}^\infty\displaystyle\max_{\theta}|\partial_\rho a|d\rho<\infty.\label{eq-a-bv}
\end{align}
Then, $(\mathcal{M}, g)$ admits a smooth isometric immersion in $\mathbb{R}^3.$
\end{theorem}

The decomposition \eqref{eq-decomposition} was motivated by \cite{ac2005, dd01, dr97, ol70, han10}.
The functions $\overline{K}$ and $a$ are referred to {\it a monotone factor} (increasing or decreasing)
%{\it monotone factor} (increasing or decreasing)
and an {\it oscillating factor} (bounded from below and above), respectively. Set
\begin{equation}\label{eq-K-star}K_*(\rho)=\rho^{-2-\gamma}\overline{K}(\rho).\end{equation}
Then,  $K_*(\rho)$ is the {\it decay factor} for the curvature $K$.
By \eqref{eq-decomposition} and \eqref{eq-K-star}, we have
\begin{equation}\label{eq-decomposition-alter}
|K|(\theta, \rho)={K}_*(\rho)a^2(\theta, \rho)\quad\text{for $\rho$ large},
\end{equation}
To understand Theorem \ref{thrm-main1}, we make some remarks on assumptions \eqref{e:inte-cond}-\eqref{eq-a-bv}.

%\begin{remark}
First, the {\it finite total curvature} condition in \eqref{e:inte-cond} is independent of coordinates.
This condition was suggested by Hong \cite{h-private}.
In polar coordinates as in Theorem \ref{thrm-main1},
\eqref{e:inte-cond} can be reformulated as an integrability condition for the decay factor $K_*$ in the form
\begin{equation}\label{e:inte-polar}
	\int_0^\infty \rho K_*(\rho)d\rho<\infty.
\end{equation}
This indicates that the inverse square decay of $K_*(\rho)$ as $\rho\to\infty$ is the borderline case for \eqref{e:inte-polar}.
In fact, according to Lemma \ref{l:G-property} below, \eqref{e:inte-polar} implies
$$\lim_{\rho\to\infty}\rho^2K_*(\rho)=0.$$

Second, in \eqref{eq-decomposition}, we allow $\overline{K}(\rho)=\rho^{2+\gamma}{K}_*(\rho)$
to be either decreasing or increasing.
If we ignore the oscillating factor $a$ in \eqref{eq-decomposition}, then \eqref{c1} reduces to
$\partial_\rho(\rho^{2+\gamma}{K}_*(\rho))\le 0$. However, we also allow $\rho^{2+\gamma}{K}_*(\rho)$ to be increasing.
For example, for a fixed $\gamma_0>0$, consider
\begin{equation}\label{c2-1}K_*(\rho)=\frac{1}{\rho^2(\log\rho)^{2+\gamma_0}}\quad\text{for }\rho\ge 2.\end{equation}
Then, $K_*$ satisfies \eqref{e:inte-polar} and, for any $\gamma>0$,
$\rho^{2+\gamma}{K}_*(\rho)$ is increasing for large $\rho$.

It is easy to construct examples that are excluded by Theorem A but included in Theorem \ref{thrm-main1}.
In fact, $K$ does not satisfy \eqref{c1} if for any $\gamma>0$,
$\rho^{2+\gamma}|K|$ is strictly increasing or changes its sign infinitely many times as $\rho\to\infty$ for a fixed $\theta$.
For example, define
\begin{equation*}%\label{c2-2}
a(\rho)=\exp\left\{1+\int_1^\rho\alpha(s)ds\right\} \quad \text{for $\rho\geq2$,}\end{equation*}
where
$$\alpha(\rho) \text{ is  the standard smooth mollification of  the function } \sum_{n=1}^\infty(-1)^n{\bf1}_{[n,\,n+n^{-2}]}.
$$
Consider, for a fixed $\gamma_0>0$,
$$K=-\frac{a^2}{\rho^2(\log\rho)^{2+\gamma_0}}\quad\text{for }\rho\ge 2.$$
It is easy to verify that \eqref{e:inte-cond}-\eqref{eq-a-bv} are satisfied.

However, for the corresponding $K$ given by \eqref{eq-decomposition}, it is straightforward to check that,
for any constant $\gamma>0$, the derivative
	\begin{equation*}%\label{eq-verify}
\drho\log(\rho^{2+\gamma}|K|)=\frac{\gamma}{\rho}-\frac{2+\gamma_0}{\rho\log\rho}+2\alpha(\rho)
	\end{equation*}
changes its sign infinitely many times as $\rho\to \infty$ and thus violates \eqref{c1}.
In addition,  $|K|$ also oscillates as $\rho\to \infty$. % with $\delta_*=-2$.

Lastly, the condition \eqref{eq-a-bul} is parallel to \eqref{eq-a-hong} in Theorem A
but is imposed on the oscillating factor $a$.
We point out that \eqref{eq-a-bv} is a technical assumption closely related to the method employed to prove Theorem \ref{thrm-main1}. We will solve a symmetric quasilinear hyperbolic system
and prove the long-time existence of solutions.
With \eqref{eq-a-bv}, we will prove that solutions remain \lq\lq small\rq\rq\
if the initial values are \lq\lq small\rq\rq.
If  \eqref{eq-a-bv} is violated, solutions may become infinite near infinity even for small initial values
and linear equations with a similar structure  (cf.\cite{CS-acta-82}).

\subsection{Strategy and novelties of the proof of Theorem \ref{thrm-main1}}\label{s-strategy}
Since the monotonicity assumption in \eqref{eq-K-bar-bound} is given in the geodesic polar coordinates $(\theta,\rho)$ for large $\rho$,
it is natural to study the  Gauss-Codazzi system in $(\theta,\rho)$.
Unfortunately, there exists a singularity at the origin for the polar coordinates $(\theta,\rho)$. 
Similarly to \cite{H, HH}, we solve the Gauss-Codazzi system in a domain including the origin in the geodesic coordinates $(x,t)$
and in its complement in the geodesic polar coordinates $(\theta,\rho)$.
We need to analyze the coordinate transform carefully in order to patch solutions from two regions to form a solution in $\mathbb R^2$.

The crucial  part of the proof  is to prove the global existence of the Gauss-Codazzi system \eqref{e:GC}-\eqref{e:Gauss}
in a domain excluding the origin 
in the geodesic polar coordinates $(\theta, \rho)$, as stated in Proposition \ref{p:uv}.
The novelties for the proof of Proposition \ref{p:uv}  are outlined as follows.

In the proof of Theorem A of Hong \cite{H}, the comparison principle was essentially used to solve the Gauss-Codazzi system,
and a key step is to construct suitable dominating functions under the assumption \eqref{c1}.

{ However, due to the presence of the oscillating factor and slow decay allowed by \eqref{e:inte-polar}, it becomes challenging to construct suitable dominating functions. 
Therefore, instead of relying on the comparison principle, we resort to the energy method to study the Gauss-Codazzi system.}

To study the asymptotic behavior of solutions, we consider a special case of the system \eqref{e:GC}-\eqref{e:Gauss}
that the metric and Riemann invariants depend only on one variable $\rho$,
which is a system of ordinary differential equations, and its solutions can be obtained explicitly.
Such explicit expressions reveal an interesting phenomenon that the Riemann invariants decay at the same order as $\rho\to\infty$
but their certain linear combinations may decay faster.
We prove Proposition \ref{p:uv} by solving an equivalent differential system to the Gauss-Codazzi system,
which is given by \eqref{eq-u-general}-\eqref{eq-v-general} for specific linear combinations (denoted by $u$ and $v$) of the Riemann invariants,  and establish different decay rates.

The reformulated system \eqref{eq-u-general}-\eqref{eq-v-general} has the following advantages.  
(i) It is a symmetric hyperbolic system, which allows us to employ energy methods to derive {\it a priori} estimates of its solutions $(u,v)$;
(ii)  it has a partial damping effect so that some terms with non-integrable coefficients can be controlled;
(iii)  the positivity of $v$ can be directly justified and the equivalence between the two systems
   \eqref{eq-u-general}-\eqref{eq-v-general} and \eqref{e:GC}-\eqref{e:Gauss} is then established.

In the derivation of the {\it a priori} estimates for the case that $\rho^{2+\gamma}{K}_*(\rho)$ is increasing,
there are two major difficulties. Specifically, $v$ has no damping and the coefficients of the nonlinear terms are not integrable.
To overcome the difficulties, we  apply an energy method {and a continuity argument}. Rather than estimating $\|u\|_{H^1}$ and $\|v\|_{H^1}$ synchronously as usual, we first fully use the damping effect of $u$ to obtain a decay estimate of $\|u\|_{H^1}$ { if $\|(u, v)\|_{H^2}$ has  an {\it a priori}   bound}.   %a weaker bound}.  
Then, we estimate $\|v\|_{H^1}$ directly from the decay estimate of $\|u\|_{H^1}$.  Finally, we take advantage of the decay estimate of $\|u\|_{H^1}$ already established
and the damping term in the equation for $\partial^2_\theta u$ to control all nonlinear terms in the equations so that
{ the {\it a priori} bound can be improved.}

It is relatively easier to derive {\it a priori} estimates for the case that $\rho^{2+\gamma}{K}_*(\rho)$ is decreasing.
One major difficulty originates from the fact that the coefficients of the nonlinear terms are not integrable.
However, both equations for $u$ and $v$ have damping effects.
We can derive decay estimate of $\|(u, v)\|_{L^2}$  { provided that $\|(u, v)\|_{H^2}$ has  an {\it a priori} bound}.
Then, we apply the decay estimates already obtained for $\|(u, v)\|_{L^2}$
and the symmetry of the system to establish the desired decay estimates of $\|\partial_\theta(u, v)\|_{H^1}$.

To prove Proposition \ref{p:uv}, we require Lemma \ref{l:derivative-boundary}, which can be derived from Lemmas \ref{l:boundary-bounds}-\ref{l:boundary-bounds-1}. The key issue to get Lemmas \ref{l:boundary-bounds}-\ref{l:boundary-bounds-1} is to estimate the $C^2$-norms of the coefficients of the metric in the geodesic coordinates and the function $\Phi$ originated from the coordinate transformations  as in

Lemmas \ref{l:bphi-under-h1-h2}-\ref{l:dthe-drho}.
Due to the singularity of $\Phi$ at $x=0$ (see \eqref{rhot}),
we prove Lemmas \ref{l:bphi-under-h1-h2}-\ref{l:bphi-under-h1-h2-1}
in two regions $|x|\leq c$ and $|x|\geq c$ and derive some fine estimates under the conditions of Theorem \ref{thrm-main1}  through an iteration for $\dt B$ as in \cite{H}.

\subsection{Organization of the paper}
The rest of the paper is organized as follows.
In Section \ref{s:formula}, 	we reformulate the Gauss-Codazzi system as an equivalent symmetric quasilinear hyperbolic system with partial dampings.
In Section \ref{s-near}, we present the existence of solutions in geodesic coordinates in a region $\Omega_1$ containing the origin.
In Section \ref{s-property}, we derive some properties of metrics and curvatures under the assumptions \eqref{e:inte-cond}-\eqref{eq-a-bul}.
In Section \ref{s-transform}, we study the coordinate transformation $F$ from geodesic coordinates to geodesic polar coordinates. % for the preparation of estimates for initial boundary values on $\partial F(\Omega_1)$.
In Section \ref{s-boundary}, we establish estimates of initial boundary values on $\partial F(\Omega_1)$. 
In Section \ref{s-local}, we prove the local existence of solutions in Sobolev spaces in geodesic polar coordinates in the region $\tilde\Omega_2=F(\R^2_+\setminus\Omega_1)$. % away from the origin.
In Section \ref{s-global}, we employ an energy method and make use of the partial damping effect to prove the global existence of solutions in geodesic polar coordinates in $\tilde\Omega_2$.

%%%%%%%%%%%%%%%%%%%%%%%%%%%%%%%%%%%%%%%%%%%%%%%%%%
\section%{Reformulation and a Special Solution}
{Preliminaries}
\label{s:formula}
%%%%%%%%%%%%%%%%%%%%%%%%%%%%%%%%%%%%%%%%%%%%%%%%%%%%%%%%%%%%%%%%%
%

In this section, we reformulate the Gauss-Codazzi system in the geodesic coordinates and geodesic polar coordinates
and derive an equivalent system for establishing the {\em a priori} estimates. %we shall study in the rest of the paper. ???
Throughout the paper, we assume $(\M, g)$ is a complete simply connected smooth surface with Gauss curvature $K<0$.

\subsection{The Gauss-Codazzi system with negative Gauss curvature}

According to Lemma 10.2.1 in \cite{HH}, there exists a global geodesic coordinate system $(x,t)$
in $\M$.
In this coordinate system, the curvature is written as $K=-\kappa^2(x, t),$ with $\kappa(x, t)$ being a positive function on $\mathcal{M}$. The metric $g$ is then of the form
\begin{equation}\label{e:metric}
g=B^2(x, t)dx^2+dt^2,
\end{equation}  where $B$ is a positive function satisfying
\begin{equation}\label{e:gauss-b}
\begin{cases}
\dis \dt^2B=\kappa^2B,\\
\dis B(x,0)=1,\, \dt B(x,0)=0.
\end{cases}
\end{equation}
The Christoffel symbols  are then given by
\begin{equation*}
\Gamma^1_{11}=\partial_x\log B, \,\, \Gamma^2_{11}=-B\dt B,\,\, \Gamma^1_{12}=\partial_t\log B, \,\,
\Gamma^1_{22}=\Gamma^2_{12}=\Gamma^2_{22}=0.
\end{equation*}
Then, the Gauss-Codazzi system \eqref{e:GC}-\eqref{e:Gauss} is reduced to
\begin{equation}\label{gc}
	\begin{split}
\partial_tL-\partial_xM&=L\partial_t\log B-M\partial_x\log
B+NB\partial_tB,\\
\partial_tM-\partial_xN&=-M\partial_t\log B,
\end{split}
\end{equation}
and
\begin{align}\label{gs}
LN-M^2=-\kappa^2B^2.
\end{align}
For $K<0$, the Gauss-Codazzi system is of hyperbolic type, and the second fundamental form $I\!\!I=0$ generates two families of asymptotic
curves:
$$\frac {dx}{dt}=\frac {-M\pm B\kappa}L.$$
The two quantities on the above right-hand side are often referred to as the
Riemann invariants (cf. \cite{Smoller}).

To seek smooth solutions of a hyperbolic system, we usually consider the equations for the Riemann invariants $r$ and $s$ defined as follows:

\begin{equation}\label{1.4.1}
r=\frac {-M-B\kappa}L,\quad s=\frac
{-M+B\kappa}L.
\end{equation}
By a direct computation, we obtain
\begin{align}\label{e:rs}
\begin{split}
\partial_tr+s\partial_xr&=\frac{r-s}2(\partial_t+r\partial_x)\log \kappa
-(r+s)\partial_t\log B\\
&\qquad-rs\partial_x \log B -B\partial_tBr^2s,\\
\partial_ts+r\partial_xs&=\frac{s-r}2(\partial_t+s\partial_x)\log \kappa
-(r+s)\partial_t\log B\\
&\qquad-rs\partial_x \log B -B\partial_tBrs^2.
\end{split}\end{align}
Conversely, if $(r, s)$ is a solution to \eqref{e:rs}, subject to
the condition $s>r$, we define
\begin{equation*}%\label{1.4.2}
L=\frac{2}{s-r}B\kappa,\quad M=-\frac{s+r}{s-r}B\kappa, \quad
N=\frac{2rs}{s-r}B\kappa,
\end{equation*}
and verify by a straightforward computation that $L, M, N$ satisfy the Gauss-Codazzi system.
We conclude by the
fundamental theorem of the surface theory that $g$ admits a smooth
isometric immersion in $\mathbb R^3$ if there exist smooth
functions $s>r$ in $\mathbb R^2$ satisfying \eqref{e:rs}, which is (3.3.9) in \cite{HH}
and was first derived by {Rozhdestvenski\u\i} \cite{Rozh}.

On the other hand, in the geodesic polar coordinates  $(\theta, \rho)$, we write the curvature as
$$K=-k^2(\theta, \rho),$$ with $k$ a positive function on $\M$. The metric $g$ is then of the form
\begin{equation}\label{e:metric-polar}
g=G^2(\theta, \rho)d\theta^2+d\rho^2,\end{equation}  where $G$ is a positive function satisfying
\begin{equation}\label{e:gauss}
\begin{cases}
\dis \drho^2G=k^2G,\\
\dis G(\theta,0)=0,\, \drho G(\theta,0)=1.
\end{cases}
\end{equation}
The Christoffel symbols  are the following:
\begin{equation*}
\Gamma^1_{11}=\dthe\log G, \,\, \Gamma^2_{11}=-G\drho G,\,\, \Gamma^1_{12}=\partial_\rho\log G, \,\,
\Gamma^1_{22}=\Gamma^2_{12}=\Gamma^2_{22}=0.
\end{equation*}
The second-second fundamental form is written as
\begin{equation*}%\label{e:secondform}
 I\!\!I= \tilde Ld\theta^2+2\tilde M d\theta d\rho+\tilde N d\rho^2.
\end{equation*}
Then, the Gauss-Codazzi system \eqref{e:GC}-\eqref{e:Gauss} is reduced to
\begin{equation}\label{gc1}
	\begin{split}
\partial_\rho \tilde L-\partial_\theta \tilde M&=\tilde L\partial_\rho\log G-\tilde M\partial_\theta\log G+\tilde N G\partial_\rho G,\\
\partial_\rho \tilde M-\partial_\theta \tilde N&=- \tilde M\partial_\rho\log G,
\end{split}
\end{equation}
and
\begin{align}\label{gs1}
\tilde L\tilde N-\tilde M^2=-k^2G^2.
\end{align}
Similarly, as in the geodesic coordinates, we also consider the equations for the Riemann invariants $w$ and $z$ defined by
% We introduce new variables $w$ and $z$ as follows:
\begin{equation}\label{1.4.1-1}
w=\frac {-\tilde{M}-Gk} {\tilde L},\quad z=\frac
{-\tilde M+Gk}{\tilde L}.
\end{equation}
By a straightforward calculation, we obtain
\begin{align}\label{e:wz}\begin{split}
\partial_\rho w+z\partial_\theta w&=\frac{w-z}2(\partial_\rho+w\partial_\theta)\log k
-(w+z)\partial_\rho\log G\\
&\qquad-wz\partial_\theta \log G -G\partial_\rho G w^2z,\\
\partial_\rho z+w\partial_\theta z&=\frac{z-w}2(\partial_\rho+z\partial_\theta)\log k
-(w+z)\partial_\rho\log G\\
&\qquad-wz\partial_\theta \log G -G\partial_\rho Gwz^2.
\end{split}\end{align}
As we can see from \eqref{e:gauss}, $\drho\log G$  has a singularity when $\rho=0.$ Thus we always solve \eqref{e:wz} in the region $\{\rho>0\}$.   If $(w, z)$ is a solution to \eqref{e:wz} in $\{\rho>0\}$, subject to
the condition $z>w$, we define
\begin{equation*}%\label{1.4.2-1}
\tilde L=\frac{2}{z-w}Gk,\quad \tilde M=-\frac{z+w}{z-w}Gk, \quad
\tilde N=\frac{2wz}{z-w}Gk,
\end{equation*}
and verify by a straightforward computation that $\tilde L, \tilde M, \tilde N$ satisfy the Gauss-Codazzi system. Then, we can obtain the solutions to the Gauss-Codazzi system in $\{\rho>0\}$.

\subsection{An equivalent system}%\label{s:ode}

We note that $w+z$ and $z-w$ appear in significant positions in the system \eqref{e:wz}. It is natural to derive a differential system of $w+z$ and $z-w$. For later purposes, we derive a more general system.

Let $\alpha$ and $\beta$ be two constants. We introduce 
functions $u$ and $v$ by
\begin{equation*}% \label{eq-relation-rs-uv}
u=G^\alpha k^{-\beta}(w+z),\quad v=G^\alpha k^{-\beta}(z-w).
\end{equation*}
A straightforward computation yields
\begin{equation}\label{eq-u-general}\begin{split}
& \partial_\rho u+\frac{1}{2}G^{-\alpha}k^\beta u\partial_\theta u-\frac {1}{2}G^{-\alpha}k^\beta v\partial_\theta v\\
&\qquad=-u\partial_\rho\log(G^{2-\alpha}k^\beta)
-\frac{1}{2}G^{-\alpha}k^\beta u^2\partial_\theta\log(G^{1-\alpha}k^\beta)\\
&\qquad\quad+\frac12G^{-\alpha}k^\beta v^2\partial_\theta\log(G^{1-\alpha}k^{1+\beta})
-\frac14G^{1-2\alpha}\partial_\rho Gk^{2\beta}(u^2-v^2)u,%=:\mathbf{f}_1(u, v, \theta,\rho),
\end{split}\end{equation}
and
\begin{equation}\label{eq-v-general}\begin{split}
& \partial_\rho v+\frac{1}{2}G^{-\alpha}k^\beta u\partial_\theta v-\frac {1}{2}G^{-\alpha}k^\beta v\partial_\theta u\\
&\qquad=-v\partial_\rho\log(G^{-\alpha}k^{\beta-1})-\frac{1}{2}G^{-\alpha}k^\beta uv\partial_\theta\log k\\
&\qquad\quad+\frac14G^{1-2\alpha}\partial_\rho Gk^{2\beta}(u^2-v^2)v.%=:\mathbf{f}_2(u, v, \theta,\rho).
\end{split}\end{equation}
We aim to solve the Cauchy problem of \eqref{eq-u-general}-\eqref{eq-v-general} with appropriately chosen initial data so that $v>0$ in $\{\rho\geq R\}$ with some $R>0$ to be fixed.

We now make two important observations.
First, the differential system \eqref{eq-u-general}-\eqref{eq-v-general} is a quasilinear {\it symmetric} hyperbolic system in $(u,v)$. The symmetry is important and permits us to employ the energy method to derive a priori estimates.
Second, we can write the  equation \eqref{eq-v-general} as
$$\partial_\rho v+\frac{1}{2}G^{-\alpha}k^\beta u\partial_\theta v+\mathcal{F}v=0,$$
for some function $\mathcal{F}$. It is easy to conclude that $v$ remains positive if it is positive initially.

\subsection{A special solution}%\label{se-special}

To study behaviors of solutions near infinity, we consider a special case that $G$ and $k$ are functions of $\rho$ only,
and seek solutions $u$ and $v$ as functions of $\rho$.
Denote by $'$ the derivative with respect to $\rho$.
Then, the system \eqref{eq-u-general}-\eqref{eq-v-general} reduces to
\begin{align*}
u'
+u(\log(G^{2-\alpha}k^\beta))'
+\frac14G^{1-2\alpha}G'k^{2\beta}(u^2-v^2)u&=0,\\
v'+v(\log(G^{-\alpha}k^{\beta-1}))'
+\frac14G^{1-2\alpha}G'k^{2\beta}(u^2-v^2)v&=0.
\end{align*}
We can write these equations as
\begin{align*}
(\log(G^{2-\alpha}k^\beta u))'+\frac14G^{1-2\alpha}G'k^{2\beta}(u^2-v^2)&=0,\\
(\log (G^{-\alpha}k^{\beta-1}v))'+\frac14G^{1-2\alpha}G'k^{2\beta}(u^2-v^2)&=0.
\end{align*}
A simple comparison yields $G^2ku=cv$ for some constant $c$. Hence, with \eqref{e:gauss}, we have
$$(G^{-\alpha}k^{\beta-1}v)'+\frac14\Big(c^2\frac{G'}{G^3}-G'G''\Big)(G^{-\alpha}k^{\beta-1}v)^3=0.$$
This equation can be integrated to yield an explicit expression of  $G^{-\alpha}k^{\beta-1}v$. We now
prescribe the initial value
$$u(R)=u_0, \quad v(R)=v_0,$$
for some constants $u_0$ and $v_0$ with $v_0>0$. Then, we obtain
\begin{align}\label{eq-ODE-solution}\begin{split}
u&=2k^\beta(R)u_0\,G^{\alpha-2}k^{-\beta}\Big(4+k^2(R)u_0^2(1-G^{-2})-v_0^2k^{2\beta-2}(R)G'^2\Big)^{-1/2},\\
v&=2k^{\beta-1}(R)v_0\,G^{\alpha}k^{1-\beta}\Big(4+k^2(R)u_0^2(1-G^{-2})-v_0^2k^{2\beta-2}(R)G'^2\Big)^{-1/2},
\end{split}\end{align}
as long as the expressions on the right-hand sides make sense.
By \eqref{e:gauss}, both $G$ and $G'$ are increasing and $G\ge G(R)$
on $[R,\infty)$. Hence, $u$ and $v$
exist on $[R,\infty)$ if $v_0k^{\beta-1}(R)G'(\infty)<2$. By \eqref{e:gauss} again, $G'(\infty)$ is finite if and only if
\begin{equation}\label{eq-ODE-condition}
\int_R^\infty (Gk^2)(\rho)d\rho<\infty.\end{equation}
In conclusion, for the ODE case that $G$ is independent of $\theta$,
we find explicit global solutions $(u,v)$ of \eqref{eq-u-general}-\eqref{eq-v-general}
with $v>0$, provided $v_0>0$ is sufficiently small.

We now examine \eqref{eq-ODE-solution} closely. Assume that
\begin{equation*}
v_0k^{\beta-1}(R)\int_R^\infty (Gk^2)(\rho)d\rho<2.\end{equation*}
Then, $u$ and $v$ exist on $[R,\infty)$, and
$$|u|\le C|u_0|G^{\alpha-2}k^{-\beta},\quad |v|\le Cv_0G^{\alpha}k^{1-\beta}\quad\text{on }[R,\infty).$$
By \eqref{eq-ODE-condition}, we may assume $k=o(G^{-1})$. Note that
$$|u|\le C|u_0|G^{\alpha-2+2\beta}(kG^2)^{-\beta},\quad |v|\le Cv_0G^{\alpha-1+\beta}(Gk)^{1-\beta}\quad\text{on }[R,\infty).$$
To make $v$ bounded, we choose $\alpha=1-\beta.$ For the bound of $u, $ we have the following two choices. If $kG^2$ grows, we choose $\alpha=0, \beta=1$ and then $u$ decays. This demonstrates that certain linear combination of Riemann invariants decays faster than themselves. If $kG^2$ decays, we choose  $\alpha=1, \beta=0$ and then both $u$ and $v$ decay.
As a consequence, we need to choose different sets of parameters $(\alpha, \beta)$ according to whether
the function $kG^{2}$ decays or grows as $\rho \to\infty$.
The ODE analysis outlined above suggests different approaches according to the growth and decay of $kG^2$.

\section{Solutions near the Origin}\label{s-near}

To prove Theorem \ref{thrm-main1}, it is sufficient to show the global existence of a smooth solution to the Gauss-Codazzi system \eqref{e:GC}-\eqref{e:Gauss}.
Since \eqref{eq-decomposition}-\eqref{eq-a-bv} holds on the geodesic polar coordinates $(\theta,\rho)$, and there is a singularity at the origin in the  $(\theta,\rho)$ coordinates,  we only use the geodesic polar coordinates $(\theta,\rho)$ to solve the system \eqref{gc1}-\eqref{gs1} in a region $\Omega$ where the origin is excluded. We will study the system \eqref{gc}-\eqref{gs} in the region $\R^2\setminus\Omega$ including the origin in the geodesic coordinates $(x,t)$.

We only consider the domain $t\ge 0$ since the region $t\le 0$ can be treated similarly.  We divide the region $t\ge 0$ into two parts by
\begin{equation}\label{omega}
	\Omega_1=\{t\le t_0(x)\text{ and }\rho(x,t)\le R_1\},\quad  \Omega_2=\R_+^2\setminus\Omega_1,
\end{equation} where
\begin{equation*}
	t_0(x)=R(1+x^2)^{\frac\mu 2}, \quad
	R_1=\max\{\rho(x, t_0(x)); |x|\leq2\},\quad \mu=\tfrac{2}{2-3 \delta}>1,
\end{equation*}
and $R$ is a large constant to be determined later. Then $2R\leq R_1\leq c_0R$ for some constant $c_0$ and we have
\begin{equation*}
	\partial\Omega_1=\{\rho(x, t)=R_1, x\in[b_-, b_+] \}\cup \{t=t_0(x), x\in\R\setminus(b_-, b_+)\},
\end{equation*}
where $\rho(b_-, t_0(b_+))=\rho(b_+, t_0(b_+))=R_1,$ with two constants $b_-<0<b_+$.

 With $(r,s)$ given by \eqref{1.4.1}, set
\begin{align*}
	p=Br, \quad q=Bs. %\quad \text{ and }  \quad \tilde{p}=Gw, \quad \tilde{q}=Gz.
\end{align*}
By \eqref{e:rs}, a simple computation yields that $(p,\, q)$ satisfies
\begin{align}
	\dis \dt p+\frac{q}{B}\dx p&=\frac{p-q}{2}(\dt+\frac{p}{B}\dx)\log\kappa-q(1+p^2)\dt\log B, \label{e:p}\\
	\dis \dt q+\frac{p}{B}\dx q&=\frac{q-p}{2}(\dt+\frac{q}{B}\dx)\log\kappa-p(1+q^2)\dt\log B.\label{e:q}
\end{align}

We impose the following initial data:
\begin{equation}\label{e:initial-pq}
	p(x, 0)=-\phi(x), ~~q(x, 0)=\phi(x),
\end{equation}
where $\phi(x)$ is a smooth and even function in $\R$,  and for $x\geq0$,
\begin{align*}
	\phi(x)=\dis\frac{1}{64\pi^2}\!\int_{x}^\infty e^{-\eta^2}\omega(\eta)\!\int_\eta^\infty
	\frac{\exp\{-700h_1(y+1)(t_0(y+1)+R_1)-y^2\}}{(t_0(y+1)+R_1)^4 h_2^4(y+1)}dyd\eta.
\end{align*}
Here $\omega(\eta)$ is a smooth cutoff function such that $\omega(\eta)=1$ for $\eta\geq1$ and $\omega(\eta)=0$ for $\eta$ near 0. It is easy to see that $\phi(x)$ is a constant near 0. Define $h_1, h_2$ as smooth and even functions in $\R$ by, for $y\geq0$,
\begin{align*}
	&h_1(y)\geq 1+\sup_{\substack{(x, t)\in\Omega_1, |x|\leq y\\ i=0, 1,2, j=1, 2}}\{\dx^i\dt^{j-1}(B^{-1}), 2|\dx^i\dt^j\log k|, 2|\dx^{i+1}\log k|,|\dx^i\dt^j\log B|\},\\
	& h_2(y)\geq1+h_1(y)+\sup_{\substack{(x, t)\in\Omega_1\\ |x|\leq y}}\{|t_0'(x)|(t_0(x)+R_1)^{-1}+|x|^2\},
\end{align*}
%Here $\epsilon$ is a positive constant to be fixed.
and they are increasing in $[0, \infty).$ Note that $h_1$ includes all the coefficients in \eqref{e:p}-\eqref{e:q} and their $x$-derivatives of second order and $t$-derivatives of first order.
We have the following lemma on the existence of solutions in  $\Omega_1$.

\begin{lemma}\label{l:omega-1}
	For any  $R>1$, the Cauchy problem of \eqref{e:p}-\eqref{e:q} with \eqref{e:initial-pq} in $\Omega_1$
	admits a unique smooth solution $(p,\, q)$ such that
	\begin{itemize}
		\item[(i)] $q-p>0$;
		\item[(ii)] $\partial\Omega_1$ is space-like in $t$ for \eqref{e:p}-\eqref{e:q};
		\item[(iii)] for $i=0, 1, 2$,
		\begin{equation}\label{e:p-q-x-derivatives}
			|\dx^ip|+|\dx^iq|\leq C[(t_0(x)+R)h_2(x)]^{-4}\quad\text{on }\partial\Omega_1.
		\end{equation}
	\end{itemize}	
\end{lemma}

\begin{proof}
This lemma can be proved by similar arguments for Lemma 10.3.4 in \cite{HH} .  The proof is then omitted.
\end{proof}

In the rest of the paper,  we will seek smooth solutions in $\Omega_2.$ To this end,
%get solutions in $ \Omega_2$ under the conditions in Theorem \ref{thrm-main1},
we shall transform the equations in $\Omega_2$ to that in $\tilde\Omega_2$ in geodesic polar coordinates. The proof of the existence of solutions in $\tilde\Omega_2$  is completely different from \cite{H, HH}. After delicate analysis of the data on the boundary of $\tilde\Omega_2$ in  Section \ref{s-transform} and Section \ref{s-boundary}, we prove the global existence in $\tilde\Omega_2$ via an energy method in Section \ref{s-local} and Section \ref{s-global}.

%%%%%%%%%%%%%%
\section{Properties of Metrics and Curvatures}\label{s-property}
%%%%%%%%%%%%%%%

Starting from this section, we consider the geodesic polar coordinates. We first establish some properties of metrics and curvatures under the assumptions \eqref{e:inte-cond}-\eqref{eq-a-bul}, to prepare for the estimates of initial boundary values and the global existence of solutions in $\tilde\Omega_2$.

For simplicity, write
\begin{equation}\label{delta}\delta=\gamma/2.\end{equation}
Then, $\delta\in(0, 1/2)$, since $\gamma$ is assumed to satisfy $0<\gamma<1$ as in Theorem \ref{thrm-main1}. Recall that
$$K(\theta, \rho)=-k^2(\theta, \rho),\, \rho^{2+\gamma}|K|=\overline{K}(\rho)a^2(\theta, \rho),\, K_*=\rho^{-2-\gamma}\overline{K}.$$
Write
\begin{equation}\label{k-star}k_*(\rho)=\sqrt{K_*(\rho)}.\end{equation}
Then,
$$k=ak_*.$$ Hence,  $\rho^{1+\delta}k_*(\rho)$ and $\overline{K}(\rho)$ have the same type of monotonicity;
namely, they increase or decrease simultaneously.

\begin{lemma}\label{l:G-property}
Assume that all the conditions in Theorem \ref{thrm-main1} are fulfilled. Then, %for sufficiently large constant $T,$
\begin{equation}
\label{e:k2t-integrable}
\int_{0}^\infty \rho k_*^2(\rho)d\rho \text{ is finite},
\end{equation}
\begin{equation}\label{eq-limit-krho}
\lim\limits_{\rho\to\infty}k_*(\rho)\rho=0,
\end{equation}
and
\begin{equation}\label{e:bt-integrable}
\frac{\drho G}{G}\geq\frac{1}{\rho}, \quad \int_{1}^\infty\max_\theta\left|\frac{\drho G}{G}-\frac{1}{\rho}\right|d\rho \text{ is finite}.
\end{equation}
Moreover,
\begin{equation}
\label{e:g-upper-lower-bound}
1\leq \partial_\rho G \leq\exp\max_\theta\left\{\int_0^\infty k^2\rho d\rho\right\}, \quad \rho\leq G \leq \rho\exp\max_\theta\left\{\int_0^\infty k^2\rho d\rho\right\},
\end{equation}
and
\begin{equation}\label{e:bx-bound}
\dthe^i\log G,~G\dthe^i\drho\log G \text{ are bounded for $\rho\geq 1$ and $1\leq i\leq3$}.
\end{equation}
\end{lemma}

The proof of Lemma \ref{l:G-property} is straightforward and relies on \eqref{e:inte-cond}-\eqref{eq-a-bul}.

\begin{proof} %[Proof of Lemma \ref{l:G-property}]
	By \eqref{e:gauss}, we have
	\begin{equation}\label{e:b}
		G(\theta, \rho)=\rho+\int_0^\rho\int_0^s(Gk^2)(\theta, \tau)d\tau ds.
	\end{equation}
	An exchange of the order of integration in \eqref{e:b} yields
	\begin{equation}\label{e:b-1}
		%\begin{split}
		G(\theta, \rho)%&=1+\int_0^t(t-s)(Bk^2)(x, s)ds\\
		=\rho+\rho\int_0^\rho(Gk^2)(\theta, s)ds-\int_0^\rho s(Gk^2)(\theta, s)ds.
		%\end{split}
	\end{equation}
	Moreover,
	\begin{equation*}%\label{e:bt-integral}
		\int_0^{2\pi} \drho G(\theta, \rho)dx=2\pi+\int_0^{2\pi}\int_0^\rho(Gk^2)(\theta, s)dsd\theta.
	\end{equation*}
	Thus
	\begin{equation}
		\label{e:bt-integral-1}
		\int_0^{2\pi} \drho G(\theta, \rho)d\theta\geq 2\pi.
	\end{equation}
	Since  $G(\theta, 0)=0$, an integration of \eqref{e:bt-integral-1} over $[0, \rho]$ yields
	\begin{equation}\label{e:b-integral}
		\int_0^{2\pi} G(\theta, \rho)d\theta\geq 2\pi\rho.
	\end{equation}
	
	On the other hand, due to \eqref{eq-a-bul},  there exists a constant $C_1>0$ such that
	$$C_1^{-1}\leq a\leq C_1.$$
	Set
	$$J_0=\int_{\mathcal{M}}|K|dA_g=\int_0^{2\pi}\int_0^\infty(Gk^2)(\theta, s)dsd\theta.$$
	By \eqref{e:b-integral}, we obtain%, for any $t>2T_0$,
	\begin{align*}
		J_0=\int_0^{2\pi}\int_0^\infty k_*^2(s)(Ga^2)(\theta, s)dsdx
		%\geq \frac{1}{C_1^2}\int_0^{2\pi}\int_0^\infty k_*^2(s)G(\theta, s)dsdx\\
		%&\quad =\frac{1}{C_1^2}\int_0^\infty k_*^2(t)\int_0^{2\pi} G(\theta, t)dxdt
		\geq \frac{1}{C_1^2}\int_0^\infty k_*^2(s)\int_0^{2\pi} G(\theta, s)d\theta ds
		\geq\frac{2\pi}{C_1^2}\int_0^\infty sk_*^2(s)ds.
	\end{align*}
	This implies
	$$\int_0^\infty sk_*^2(s)ds\le C_1^2J_0,$$
	%Then by \eqref{e:b-integral}, we obtain%, for any $t>2T_0$,
	%\begin{align*}
	%J_0&=\int_0^{2\pi}\int_0^\infty k_*^2(s)(Ga)(\theta, s)dsdx\\
	%%\geq \frac{1}{C_1^2}\int_0^{2\pi}\int_0^\infty k_*^2(s)G(\theta, s)dsdx\\
	%%&\quad =\frac{1}{C_1^2}\int_0^\infty k_*^2(t)\int_0^{2\pi} G(\theta, t)dxdt
	%&\geq \frac{1}{C_1^2}\int_{2R_0}^\infty k_*^2(s)\int_0^{2\pi} G(\theta, s)d\theta ds
	%\geq\frac{J_0}{4C_1^2}\int_{2R_0}^\infty sk_*^2(s)ds.
	%\end{align*}
	%This implies
	%$$\int_{2R_0}^\infty sk_*^2(s)ds\le 4C_1^2,$$
	and hence \eqref{e:k2t-integrable} holds.
	Moreover,  for any $\rho>0$,
	\begin{align*}
		\int_0^\rho sk^2(\theta, s)ds\leq C_1\int_0^\rho sk_*^2(s)ds%=C_1^2\int_1^Tk_*^2(s)sds+C_1^2\int_T^\rho k_*^2(s)sds
		\leq C_2,
	\end{align*}
	for some positive constant $C_2$.
	
	Next, it is easy to get \eqref{eq-limit-krho} if $\rho^{1+\delta}k_*$ is monotonically decreasing. If $\rho^{1+\delta}k_*$ is increasing, then for any $\tau>\rho>1,$ we have
	\[ k_*^2(\tau)\tau^{2+2\delta}\geq k_*(\rho)\rho^{2+2\delta},\]
	which implies
	\[ k_*^2(\tau)\tau\geq k_*(\rho)\rho^{2+2\delta}\tau^{-1-2\delta}.\]
	Integrating the above inequality on $(\rho, \infty)$ yields
	\[\int_\rho^\infty k_*^2(\tau)\tau d\tau\geq k_*(\rho)\rho^{2+2\delta}\int_\rho^\infty\tau^{-1-2\delta}d\tau=\frac1{2\delta}k_*^2(\rho)\rho^2.\]
	With \eqref{e:k2t-integrable}, we have \eqref{eq-limit-krho}.
	
	In addition, integrating \eqref{e:gauss} yields
	\begin{align*}
		\drho G&=1+\int_0^\rho k^2Gds=1+\int_0^\rho Gd\left(\int_0^sk^2d\tau\right)\\
		&=1+G\int_0^\rho k^2d\tau-\int_0^\rho \partial_s G\int_0^s k^2d\tau ds\\
		&=1+\int_0^\rho \partial_s Gds\int_0^\rho k^2d\tau-\int_0^\rho \partial_s G\int_0^s k^2d\tau ds\\
		&=1+\int_0^\rho \partial_s G\int^\rho_s k^2d\tau ds.
	\end{align*}
	Thus, by the same iteration  for the proof of (1.8) in \cite{H}, we have, for any $\theta\in [0,2\pi]$ and $\rho>0$,
	\begin{equation*}%\label{e:bt-lower-upper}
		1\leq \drho G(\theta, \rho)\leq\exp\int_0^\rho sk^2(\theta , s)ds.
	\end{equation*}
	Then,
	\begin{equation*}%\label{e:B-lower-upper}
		\rho\leq G(\theta, \rho)\leq\int_0^\rho \exp\int_0^\tau sk^2(x, t)ds d\tau.
	\end{equation*}
	Hence, for any $\theta\in[0, 2\pi]$ and any $\rho>0$,
	\begin{equation*}%\label{e:bt-bound}
		1\leq \drho G(\theta, \rho)\leq \exp\max_{\theta}\int_0^\rho sk^2(\theta , s)ds\leq e^{C_2},
	\end{equation*}
	and then, by a simple integration,
	\begin{equation}\label{e:b-bound}
		\rho\leq G(\theta, \rho)\leq \rho\exp\max_{\theta}\int_0^\rho sk^2(\theta , s)ds\leq e^{C_2}\rho.
	\end{equation}
	Then, we obtain \eqref{e:g-upper-lower-bound}. By  \eqref{e:b-1}, we have, for $\rho>0$,
	\begin{equation}\label{e:bt-over-t}
		\frac{\drho G}{G}-\frac{1}{\rho}=\frac{\rho^{-2}\int_0^\rho sGk^2ds}{1+\int_0^\rho Gk^2(1-s\rho^{-1})ds}>0.
	\end{equation}
	Thus, for any $\rho>1$,
	\begin{equation*}
		\max_\theta\left|\frac{\drho G}{G}-\frac{1}{\rho}\right|\leq\frac{e^{C_2}C_1^2\rho^{-2}\int_0^\rho s^2k_*^2ds}{1+C_1^{-2}\int_0^\rho sk_*^2(1-s\rho^{-1})ds}.
	\end{equation*}

	By an exchange of the order of integration and \eqref{e:k2t-integrable}, we have, for any $\rho>1$,
	\begin{align*}
		\int_1^\rho s^{-2}\int_1^s\tau^2k_*^2d\tau ds\leq \int_1^\rho\tau^2k_*^2\int_{\tau}^\rho s^{-2}ds d\tau
		\leq \int_1^\rho \tau k_*^2d\tau\leq C_2.
	\end{align*}
	This implies
	\begin{align*}
		\int_1^\rho \max_\theta\left|\frac{\partial_s G}{G}-\frac{1}{s}\right|ds
		\leq C\int_1^\rho s^{-2}\int_0^1\tau^2k_*^2d\tau ds+C\int_1^\rho s^{-2}\int_1^s\tau^2k_*^2d\tau ds\leq C.
	\end{align*}
	Hence, \eqref{e:bt-integrable} holds.
	
	Finally, we prove \eqref{e:bx-bound}. By \eqref{e:gauss}, we have
	\begin{equation}\label{e:btt-equation}
		\drho\drho\log G=k^2-(\drho\log G)^2.
	\end{equation}
	Differentiating \eqref{e:btt-equation} with respect to $\theta$ yields an equation for $\partial_\theta\partial_\rho\log G$ given by
	$$\partial_\rho(\partial_\theta\partial_\rho\log G)=2k\dthe k-2\partial_\theta\partial_\rho\log G\partial_\rho\log G.$$
	%With $\partial_\theta\partial_\rho\log G(x,0)=0$,
	Solving the above equation directly, we have
	\begin{equation}\label{e:logb-xt}
		\partial_\theta\partial_\rho\log G=\frac{\partial_\theta\partial_\rho\log G(\theta,1)}{G^2}+\frac{2}{G^2}\int_1^\rho k\dthe kG^2ds.
	\end{equation}
	By \eqref{eq-a-bul}, we get $|\dthe a|\leq Ca$, and then $|\dthe k|\leq Ck\leq Ck_*$.
	Thus, by \eqref{e:b-bound} and \eqref{e:k2t-integrable}, for $\rho\geq1,$
	\begin{align*}
		G|\dthe\drho\log G|\leq\frac{C}{G}+\frac{C}{G}\int_1^\rho k_*^2G^2ds
		\leq  C+C\int_1^\rho k_*^2sds\leq C,
	\end{align*}
	since $G$ is increasing in $(0, \infty)$ and $\log G(\theta, 1)$ is a $C^3$ function for $\theta\in[0, 2\pi]$.
	Next,  an integration of \eqref{e:logb-xt} with respect to $\rho$ yields
	\begin{equation}\label{e:logb-x}
		\dthe\log G=\dthe\log G(\theta, 1)+ \int_1^\rho \frac{\partial_\theta\partial_\rho\log G(\theta,1)}{G^2} ds+2\int_1^\rho\frac{1}{G^2}\int_1^sk\dthe kG^2d\tau ds.
	\end{equation}
	Since $a, a^{-1}$ are bounded, by \eqref{e:b-bound}, $|\dthe k|\leq Ck_*$, and an  exchange of the order of integration, we have, for $\rho\geq 1,$
	\begin{equation*}
		\begin{split}
			|\partial_\theta\log G|&\leq C+C\int_1^\rho\frac{1}{s^2}ds+\int_1^\rho\frac{C}{s^2}\int_1^sk_*^2(\tau)\tau^2d\tau ds\\
			&\leq C+C\int_1^\rho k^2_*(\tau)\tau^2\int_\tau^\rho\frac{1}{s^2}dsd\tau\\
			&\leq C+C\int_1^\rho k_*^2\tau d\tau\leq C.
		\end{split}
	\end{equation*}
	A differentiation of \eqref{e:logb-x} with respect to $\theta$ yields
	\begin{align*}
		\dthe^2\log G&=\dthe^2\log G(\theta, 1)+ \int_1^\rho \frac{\partial_\theta^2\partial_\rho\log G(\theta,1)}{G^2} ds
		-2\int_1^\rho \frac{\partial_\theta\partial_\rho\log G(\theta,1)\dthe\log G}{G^2} ds
		\\
		&\qquad -4\int_1^\rho\frac{\dthe\log G}{G^2}\int_1^sk\dthe kG^2d\tau ds
		+2\int_1^\rho\frac{1}{G^2}\int_1^s\dthe(k\dthe kG^2)d\tau ds,
	\end{align*}
	and
	\begin{align*}
		\dthe^3\log G
		&=\dthe^3\log G(\theta, 1)+ \int_1^\rho \frac{\partial_\theta^3\partial_\rho\log G(\theta,1)}{G^2} ds
		-4\int_1^\rho \frac{\partial_\theta^2\partial_\rho\log G(\theta,1)\dthe\log G}{G^2} ds\\
		&\qquad -2\int_1^\rho \frac{\partial_\theta\partial_\rho\log G(\theta,1)\dthe^2\log G}{G^2} ds+4\int_1^\rho \frac{\partial_\theta\partial_\rho\log G(\theta,1)(\dthe\log G)^2}{G^2} ds\\
		&\qquad -4\int_1^\rho\frac{\dthe^2\log G-2(\dthe\log G)^2}{G^2}\int_1^s2k\dthe kG^2d\tau ds\\
		&\qquad-8\int_1^\rho\frac{\dthe\log G}{G^2}\int_1^s\dthe(k\dthe kG^2)d\tau ds
		+2\int_1^\rho\frac{1}{G^2}\int_1^s\dthe^2(k\dthe kG^2)d\tau ds,
	\end{align*}
	where
	$$\dthe(k\dthe kG^2)=[k\dthe^2k+(\dthe k)^2]G^2+2k\dthe kG^2\dthe\log G,$$
	and
	\begin{align*}
		\dthe^2(k\dthe kG^2)&=[k\dthe^3k+3\dthe k\dthe^2k]G^2+4[k\dthe^2k+(\dthe k)^2]G^2\dthe\log G\\
		&\qquad +2k\dthe^2k[G^2\dthe^2\log G+2G^2(\dthe\log G)^2].
	\end{align*}
	By \eqref{eq-a-bul}, we have $|\dthe k|+|\dthe^2k|+|\dthe^3k|\leq Ck_*$, and hence, when $\rho\geq1$, for $i=2, 3$ inductively,
	\begin{equation*}
		|\partial_\theta^i\log G|\leq C.
	\end{equation*}
	Similarly, differentiating \eqref{e:logb-xt} with respect to
	$\theta$, we have,  for $i=2, 3$,
	\begin{align*}
		|G\partial_\theta^{i}\partial_\rho\log G|\leq C.
	\end{align*}
	Therefore,  when $\rho\geq1,$ $\partial_\theta^i\log G$ and $G\partial_\theta^i\partial_\rho\log G$ are bounded, for $i=1, 2, 3.$
\end{proof}
Here in the proof of Lemma \ref{l:G-property} and hereafter, $C$ is a constant that may change from line to line but is independent of $R$.

\section{The Coordinate Transformation}\label{s-transform}

In this section, we will study the coordinate transformation
\begin{equation}\label{F}
F:(x, t)\mapsto(\theta, \rho),\end{equation}
from the geodesic coordinates to the geodesic polar coordinates.

By the definitions of both coordinate systems and the triangle inequality, we immediately have
\begin{equation*}
	t,\,\tfrac12|x|\leq\rho(x, t)\leq t+|x|.
\end{equation*}
By \eqref{e:metric} and \eqref{e:metric-polar}, we also have
$$g=G^2(\rho, \theta)d\theta^2+d\rho^2=B^2(x, t)dx^2+dt^2,$$
and then
\begin{equation}\label{coordinates}
	\rho_t^2+G^2\theta_t^2=1, ~\rho_x^2+G^2\theta_x^2=B^2, ~\rho_t\rho_x+G^2\theta_t\theta_x=0.
\end{equation}
In addition, the geodesic equation of the $t$-curve yields %helps us get
\begin{equation}\label{tgeodesic}
	\rho_{tt}=G_\rho G\theta_t^2.
\end{equation}
Solving $\theta_t$ from the first equation of \eqref{coordinates} and the equation  \eqref{tgeodesic},
we have
\begin{equation}\label{rhot}
	\rho_t=\tanh \Phi, \text{ with } \Phi=\int_0^t\partial_\rho\log Gds, \text{ for } x\neq0.
\end{equation}
Inserting  \eqref{rhot} into  \eqref{coordinates} and \eqref{tgeodesic}, we  get
\begin{equation}\label{thetat}
	\theta_t=\frac{\xi}{G\cosh\Phi}, \quad \rho_x=-\frac{\xi B}{\cosh\Phi}, \quad \theta_x=\frac{B}{G}\tanh\Phi,
\end{equation}
where $\xi=1$ or $-1$.

The function $\Phi$ plays an important role in later studies and has the following property.

\begin{lemma}\label{l:phi}
	Assume that all the conditions in Theorem \ref{thrm-main1} are fulfilled. Then, for $t>0$,
	\begin{equation}\label{e:sinh-cosh-tanh-phi}
		\begin{split}
			&e^{\Phi}\geq\frac{t+|x|}{|x|}\,, \quad 0\leq 1-\tanh\Phi\leq \frac{2|x|}{t+|x|}\,,\\
			&\sinh\Phi\geq \frac{1}{2}\left(\frac{t}{|x|}+\frac{t}{t+|x|}\right), \quad \cosh\Phi\geq\frac{1}{2}\left(\frac{t}{|x|}+1\right).
		\end{split}
	\end{equation}
\end{lemma}
\begin{proof}
	By the definition of $\Phi$ and Lemma \ref{l:G-property}, we have
	$$\Phi\geq\int_0^t\frac{1}{\rho}ds\geq\int_0^t\frac{1}{s+|x|}ds\geq\log\left(\frac{t+|x|}{|x|}\right).$$
	%which we remark that for $x=0$ also holds.
	Hence,
	$$e^{\Phi}\geq\frac{t}{|x|}+1, \quad e^{-\Phi}\leq\frac{|x|}{t+|x|}\leq 1,$$
	and thus
	\begin{align*}
		&\sinh\Phi=\frac{e^{\Phi}-e^{-\Phi}}{2}\geq \frac{1}{2}\left(\frac{t}{|x|}+1-\frac{|x|}{t+|x|}\right)=\frac{1}{2}\left(\frac{t}{|x|}+\frac{t}{t+|x|}\right),\\
		&\cosh\Phi=\frac{e^{\Phi}+e^{-\Phi}}{2}\geq\frac{1}{2}\left(\frac{t}{|x|}+1\right),\\
		&0\leq 1-\tanh\Phi=1-\frac{e^{\Phi}-e^{-\Phi}}{e^{\Phi}+e^{-\Phi}}\leq\frac{2}{e^\Phi+e^{-\Phi}}\leq \frac{2|x|}{t+|x|}.
	\end{align*}
	Therefore, \eqref{e:sinh-cosh-tanh-phi} holds.
\end{proof}

%\noindent
We note that the Jacobian of the transformation $F$ as in \eqref{F} is given by
\begin{equation}\label{j}
	%\textrm{Jac}F
	\textrm{J}_F=\det\left(
	\begin{matrix}
		\theta_x&\theta_t\\
		\rho_x&\rho_t\\
	\end{matrix}
	\right)=\frac{B}{G}.
\end{equation}
Then,
\begin{equation}\label{xtheta}
	\begin{split}
		x_\theta=\frac{G}{B}\tanh\Phi, \quad & x_\rho=-\frac{\xi}{B\cosh\Phi},\\
		t_\theta=\frac{\xi G}{\cosh\Phi}, \quad & t_\rho=\tanh\Phi.
	\end{split}
\end{equation}
On the other hand, by the second fundamental forms
$$I\!\!I=Ldx^2+2Mdxdt+Ndt^2=\tilde{L}d\theta^2+2\tilde{M}d\theta d\rho+\tilde{N}d\rho^2,$$
we have
\begin{equation*}
	\left(
	\begin{matrix}
		L&M\\
		M&N\\
	\end{matrix}
	\right)=
	\left(
	\begin{matrix}
		\theta_x&\rho_x\\
		\theta_t&\rho_t\\
	\end{matrix}
	\right)
	\left(
	\begin{matrix}
		\tilde{L}&\tilde{M}\\
		\tilde{M}&\tilde{N}\\
	\end{matrix}
	\right)
	\left(
	\begin{matrix}
		\theta_x&\theta_t\\
		\rho_x&\rho_t\\
	\end{matrix}
	\right).
\end{equation*}
We recall  \eqref{1.4.1} and \eqref{1.4.1-1}. Then, the tangents $(r, s)$ 
and $(w, z)$ 
of the asymptotic curves in both coordinates $(x, t)$ and $(\theta, \rho)$ are given by
\begin{align*}
	r=\frac{-M-\kappa B}{L}, &\quad s=\frac{-M+\kappa B}{L},\\
	w=\frac{-\tilde M-kG}{\tilde L}, &\quad z=\frac{-\tilde M+kG}{\tilde L}.
\end{align*}
In addition, we also study the vector fields under the transformation.
A vector field  $V$ in $\R^2$ is said to be a normalized
vector if $V$ is of the form $V=\partial_t+\zeta(x, t)\partial_x.$
The differential map by $F_*(V)$ is defined by
\begin{equation*}
	F_*(V)=(\rho_t+\zeta\rho_x)\partial_\rho+(\theta_t+\zeta\theta_x)\partial_\theta,
\end{equation*}
and then, we have
\begin{equation*}
	F_*(V)=(\rho_t+\zeta\rho_x)(\partial_\rho+\tilde{\zeta}\partial_\theta),
\end{equation*}
with
\begin{equation*}
	\tilde{\zeta}=\frac{\theta_t+\zeta\theta_x}{\rho_t+\zeta\rho_x}
	\text{ provided } \rho_t+\zeta\rho_x\neq0.
\end{equation*}
Define
$$\tilde{F}_*(\zeta)=\frac{\theta_t+\zeta\theta_x}{\rho_t+\zeta\rho_x}
	\quad\text{if } \rho_t+\zeta\rho_x\neq0.$$
	Then, for any $\zeta_1, \zeta_2$,
\begin{equation*}
	\tilde\zeta_1-\tilde\zeta_2=\frac{B(\zeta_1-\zeta_2)}{G(\rho_t+\zeta_1\rho_x)(\rho_t+\zeta_2\rho_x)}.
\end{equation*}
In particular, we have
\begin{align}
	w=\tilde{F}_*(r), \quad z=\tilde{F}_*(s), \label{e:zw-diff}
\end{align}
and hence
\begin{align}
	 z-w=\frac{B(s-r)}{G(\rho_t+r\rho_x)(\rho_t+s\rho_x)}.\label{e:zw-diff1}
\end{align}
The following lemma ensures that $\tilde F_*$ makes sense.

\begin{lemma}\label{l:transform-1}
	Assume that all the conditions in Theorem \ref{thrm-main1} are fulfilled. Let $(p,\, q)$ be the solution obtained in Lemma \ref{l:omega-1}. If $R$ is sufficiently large, then
	\begin{align}
		\rho_t+\frac{p}{B}\rho_x\geq\frac{1}{2},\quad \rho_t+\frac{q}{B}\rho_x\geq\frac{1}{2}\quad\text{on }\partial\Omega_2.\label{e:boundary-pq}
	\end{align}
\end{lemma}

\begin{proof}
	By \eqref{rhot} and \eqref{thetat}, we have
	\begin{align*}
		\rho_t+\frac{p}{B}\rho_x=\tanh\Phi-\frac{p}{B}\frac{\xi B}{\cosh\Phi}\geq 1-|1-\tanh\Phi|-\frac{|p|}{\cosh\Phi}.
	\end{align*}
	Note that $e^{\Phi}\geq R+1$ on $\partial\Omega_2$ and then
	%\begin{align*}
	$|1-\tanh\Phi|\leq\frac{4}{R}.$
	%\end{align*}
	Hence, by Lemma \ref{l:omega-1}(iii),
	\begin{align*}
		\rho_t+\frac{p}{B}\rho_x\geq1-\frac{4}{R}-\frac{C}{R\cosh\Phi}\geq1-\frac{C}{R}\geq\frac{1}{2},
	\end{align*}
	if $R$ is chosen sufficiently large.  We can get
	the second inequality of \eqref{e:boundary-pq} similarly.
\end{proof}

We will derive some estimates of the function $\Phi$ and metrics in geodesic coordinates  under the assumptions \eqref{e:inte-cond}-\eqref{eq-a-bul}, to prepare for the estimates of initial boundary values.
Recall $\delta=\gamma/2$ and  $0<\delta<1/2$.
	
If $\rho^{1+\delta}k_*$ is increasing, we have the following lemma.

\begin{lemma}\label{l:bphi-under-h1-h2} Assume that all the conditions in Theorem \ref{thrm-main1} are fulfilled and $\rho^{1+\delta}k_*$ is increasing, for some $\delta\in (0,1/2)$. Let $c$ be an arbitrary positive constant. Then, for large $t,$
	\begin{enumerate}
		\item if $|x|\leq c, $ then
		\begin{equation}\label{e:b-x-less}
			C^{-1}t\leq B\leq  C t,\quad  |\partial_\theta\log B|\leq  C,
		\end{equation}
		and
		\begin{equation}\label{e:phi-x-less}
			\begin{split}
				& |\partial_\theta\Phi|\leq  Ck_*^2(t+c)(t+c)^2\sinh\Phi,\\
				&\left|\partial_\theta^i\left(\frac1{\sinh\Phi}\right)\right|\leq  Ck_*^2(t+c)(t+c)^2, \quad i=1, 2;
			\end{split}
		\end{equation}
		\item if $|x|\geq c,$ then
		\begin{equation}\label{e:b-x-more}
			1\leq B\leq \frac{ Ct}{|x|}+1,\quad  |\partial_\theta\log B|\leq  C(t+|x|),
		\end{equation}
		and
		\begin{equation}
			\label{e:phi-x-more}
			|\partial_\theta\Phi|\leq  C, \quad \left|\partial_\theta^i\left(\frac1{\sinh\Phi}\right)\right|
			\leq \frac{ C}{\sinh\Phi\tanh^i\Phi},
			\quad i=1, 2.
		\end{equation}
	\end{enumerate}
\end{lemma}

\begin{proof} The proof consists of several steps.
	
	{\it Step 1. We prove \eqref{e:b-x-less}.} By $|x|\leq c$, it is easy to see
	$$t\leq \rho(x, t)\leq t+|x|\leq t+c.$$
	Since $k_*(\rho)\rho^{1+\delta}$ is increasing for $\rho\geq \lambda_0$ for some large constant $\lambda_0$, we have
	\begin{equation*}%\label{e:k-1}
		\begin{split}
			k_*(\rho)\geq k_*(t)\frac{t^{1+\delta}}{\rho^{1+\delta}}\geq k_*(t)\frac{t^{1+\delta}}{(t+c)^{1+\delta}},
		\end{split}
	\end{equation*}
	and
	\begin{equation*}%\label{e:k-2}
		\begin{split}
			k_*(\rho)\leq k_*(t+c)\frac{(t+c)^{1+\delta}}{\rho^{1+\delta}}\leq k_*(t+c)\frac{(t+c)^{1+\delta}}{t^{1+\delta}}.
		\end{split}
	\end{equation*}
	Note that  $\kappa(x, t)=k(\theta, \rho)=a(\theta, \rho)k_*(\rho).$ Since $C_1^{-1}\leq a\leq C_1$, we have
	\begin{equation}\label{eq-kappa-bound}
		\frac{1}{C}\leq\frac{\kappa(x, t)}{k_*(\rho)}\leq C,
	\end{equation}
	and then, for $|x|\leq c$ and $t\geq \lambda_0$, we get
	\begin{equation}\label{e:major-kappa}
		\frac{k_*(t)}{C}\leq \kappa(x, t)\leq Ck_*(t+c).
	\end{equation}
	By \eqref{eq-a-bul}, \eqref{e:major-kappa}, and \eqref{e:k2t-integrable}, we have, for $|x|\leq c,$
	\begin{equation}
		\label{e:kt-inte}
		\int_0^\infty s\kappa^2(x, s)ds\leq C+ C\int_{\lambda_0}^\infty k_*^2(s+c)sds\leq C+C\int_{\lambda_0+c}^\infty (k_*^2(s)s+k_*^2(s))ds\leq C,
	\end{equation}
	and
	\begin{equation}\label{e-kappa}
		\begin{split}
			&\int_0^t  \kappa^2(x, s)ds\geq\int_0^t\frac1Ck_*^2(s)ds,\\
			&\int_0^t  \kappa^2(x, s)ds\leq C+\int_1^\infty s\kappa^2(x, s)ds\leq C.
		\end{split}
	\end{equation}
	Recall (1.8) in \cite{H}:
	\begin{equation}\label{eq-hong-1.8}
		\int_0^t\kappa^2(x, s)ds\leq \dt B\leq \int_0^t\kappa^2(x, s)ds\exp\left\{\int_0^ts\kappa^2(x, s)sds\right\}.
	\end{equation}
	By \eqref{e:kt-inte}, \eqref{e-kappa}, and \eqref{eq-hong-1.8}, we obtain
	\begin{equation}\label{e-bt}
		C_5\leq \dt B\leq C_6,
	\end{equation}
	for some positive constant $C_5$ and $C_6$. An integration yields
	\begin{equation}\label{eq-B-low-upper}
		C^{-1}t\leq B\leq Ct,
	\end{equation}
	if $t$ is large. This is the first part of \eqref{e:b-x-less}.
	
	To estimate $\partial_\theta\log B=\frac{\partial_\theta B}{B}$, we  bound $\partial_\theta B$ first. In fact, by integrating \eqref{e:gauss-b}, we have
	\[B=B(x, \lambda_0)+\dt B(x, \lambda_0)(t-\lambda_0)+\int_{\lambda_0}^t\int_{\lambda_0}^sk^2Bd\tau ds,\]
	and hence, after differentiating with respect to $\theta$,
	\begin{equation}\label{eq-blam0}
		\begin{split}
			\dthe B&=\dx B(x, \lambda_0)x_\theta+\partial_{xt}^2B(x, \lambda_0)x_\theta(t-\lambda_0)+\dt B(x, \lambda_0)t_\theta\\
			&\qquad+\int_{\lambda_0}^t\int_{\lambda_0}^sk^2\dthe Bd\tau ds+\int_{\lambda_0}^t\int_{\lambda_0}^s2k\dthe kBd\tau ds+t_\theta\int_{\lambda_0}^tk^2Bd\tau.
		\end{split}
	\end{equation}
	For large $\rho\geq t$, by \eqref{e:sinh-cosh-tanh-phi}, we have %when $\rho\geq t\geq R$,
	$$e^{\Phi}\geq\frac{\rho}{c}, \quad 0\leq e^{-\Phi}\leq 1,$$
	and then, for $|x|\leq c$ and $t$ large,
	\begin{equation}\label{e:sinh-cosh-tanh-phi-1}
		\begin{split}
			&\sinh\Phi\geq \frac{\rho}{4c},\quad \cosh\Phi\geq\frac{\rho}{2c},\quad
			1-\frac{2c}{\rho}\leq\tanh\Phi\leq1. %\quad1\leq\coth\Phi\leq\frac{\rho}{\rho-2c}.
		\end{split}
	\end{equation}
	This implies
	\begin{equation}\label{e:t-theta-bound}
		|x_\theta|=\frac{G}{B}\tanh\Phi\leq C,\quad |t_\theta|=\frac{G}{\cosh\Phi}\leq \frac{C\rho}{\rho/2c}\leq C.
	\end{equation}
	Note that $\dt B(x, \lambda_0),\,\dx B(x, \lambda_0)$, and $\partial_{xt}^2B(x, \lambda_0)$ are bounded for $|x|\leq c$. In addition, for $\rho\geq\lambda_0,$ we also have $|\dthe k|\leq Ck_*(\rho).$ By \eqref{e-bt}, we get $B\leq Ct+1.$ Hence by \eqref{eq-blam0}, we obtain
	\begin{equation}\label{e-Btheta}
		\begin{split}
			|\dthe  B|&\leq C\int_{\lambda_0}^t\int_{\lambda_0}^s k_*^2(\tau+c)|\dthe B|d\tau ds+C\int_{\lambda_0}^t\int_{\lambda_0}^s  k_*^2\tau d\tau ds+Ct+C\\
			&\leq C\int_{\lambda_0}^t\int_{\lambda_0}^s  k_*^2(\tau +c)|\dthe B|d\tau ds+Ct+C.
		\end{split}
	\end{equation}
	Set
	$$\mathfrak{B}(t)=\int_{\lambda_0}^t\int_{\lambda_0}^s  k_*^2(\tau+c)|\dthe B|d\tau ds+Ct+C.$$
	Then, $$\dt^2 \mathfrak{B}=k_*^2(t+c)|\dthe B|, \quad |\dthe B|\leq C\mathfrak{B},$$
	and thus
	$$\dt^2\mathfrak{B}\leq Ck_*^2(t+c)\mathfrak{B}.$$
	Applying the same iteration method of deriving (1.8) in \cite{H} yields
	\begin{align*}
		\dt \mathfrak{B}\leq C\int_{\lambda_0}^tk_*^2(s+c)ds \exp\left\{\int_{\lambda_0}^tsk_*^2(s+c)ds\right\}\leq C.
	\end{align*}
	An integration of the above inequality implies $|\dthe B|\leq \mathfrak{B}\leq Ct$ for $t\geq\lambda_0.$
	Finally by \eqref{eq-B-low-upper}, we have, for $t$ large,
	\begin{equation*}
		%\label{e:logb-theta-bound}
		|\partial_\theta\log B|\leq\frac{|\dthe B|}{B}\leq C.
	\end{equation*}
	Hence, the second part of  \eqref{e:b-x-less} is proved.
	
	{\it Step 2. We prove \eqref{e:phi-x-less}.}
	%{\sf Proof of \eqref{e:phi-x-less}:}
	For the estimate of $\partial_\theta\Phi$, by \eqref{tgeodesic}-\eqref{xtheta}, we have
	\begin{equation}\label{e:phi-1-derivative}
		\begin{split}
			\partial_\theta\Phi&=GG_\rho\theta_t^2t_\theta\cosh^2\Phi+(GG_\rho\theta_x\theta_t-B_tG\theta_t)x_\theta\cosh^2\Phi\\
			&=\frac{\xi \partial_\rho G}{\cosh\Phi}+\xi (\tanh\Phi\partial_\rho\log G-\dt\log B)G\sinh\Phi.
		\end{split}
	\end{equation}
	%which is from \cite{H}.
	Then,
	\begin{equation}\label{e-bound-phi}
		\begin{split}
			|\partial_\theta\Phi|&\leq
			\frac{\partial_\rho G}{\cosh\Phi}+(|\tanh\Phi-1|\partial_\rho\log G+|\rho^{-1}-t^{-1}|)G\sinh\Phi\\
			&\qquad +(|\partial_\rho\log G-\rho^{-1}|+|t^{-1}-\dt\log B|)G\sinh\Phi.
		\end{split}
	\end{equation}
	In \eqref{e-bound-phi}, $|\dt\log B-t^{-1}|$ is the key term. To control it, we integrate \eqref{e:gauss-b} to get
	\begin{equation}\label{eq-bt}	
		\dt B(x, t)=\int_0^t(B\kappa^2)(x, s)ds,
	\end{equation}
	and
	\begin{align*}
		B&=1+\int_0^t\int_0^s(B\kappa^2)(x, \tau)d\tau ds\\
		&=1+t\int_0^t(B\kappa^2)(x, s)ds-\int_0^ts(B\kappa^2)(x, s)ds.
	\end{align*}
	Hence,
	\begin{align*}
		\frac{\dt B}{B}-\frac{1}{t}
		&=\frac{-t^{-2}+t^{-2}\int_0^tsB\kappa^2ds}{\int_0^tB\kappa^2ds+t^{-1}-t^{-1}\int_0^tsB\kappa^2ds}\\
		&=\frac{-t^{-2}+t^{-2}\int_0^tsB\kappa^2ds}{t^{-1}+\int_0^t(1-\frac{s}{t})B\kappa^2ds}.
	\end{align*}
	By \eqref{e:major-kappa} and the increasing property of  $k_*(\rho)\rho^{1+\delta}$ for $\rho>\lambda_0$, we further deduce, for $t>0,$
	\begin{equation}\label{e:bt-diff}
		\begin{split}
			\left|\frac{\dt B}{B}-\frac{1}{t}\right|
			&\leq \frac{t^{-2}+Ct^{-2}\int_0^ts^2k^2_*(s+c)ds}{t^{-1}+C^{-1}\int_0^tBk_*^2ds-Ct^{-1}\int_0^tsk_*^2(s)ds}\\
			&\leq \frac{C}{t^2}\left(1+\int_0^ts^2k^2_*(s+c)ds\right),\\
			%	&\leq Ck^2_*(t+c)(t+c),
		\end{split}
	\end{equation}
	where we used
	$$\lim_{t\to\infty}\frac1t\int_0^tk_*^2(s)s^2ds=0 \text{ (since } \lim_{t\to\infty}k_*(t)t=0)$$
	to get the second inequality. Hence, for $t$ sufficiently large,
	\begin{equation}\label{e-bt1t}
		|\dt\log B-t^{-1}|\leq Ck^2_*(t+c)(t+c).
	\end{equation}
	In addition, by \eqref{e:bt-over-t}, we get, for large $\rho$,
	\begin{equation}\label{eq-logg-rho-difference}
		|\partial_\rho\log G-\rho^{-1}|\leq \frac{C}{\rho^{2}}\int_0^\rho s^2k_*^2(s)ds\leq Ck_*^2(\rho)\rho.
	\end{equation}
	By Lemma \ref{l:G-property}, \eqref{e-bound-phi}, \eqref{e-bt1t}, and \eqref{eq-logg-rho-difference}, we have, for $t$ large,
	\begin{equation}\label{e:phi-derivative-bound}
		\begin{split}
			|\partial_\theta\Phi|&\leq
			C\rho^{-1}+C(\rho^{-1}+k_*^2(\rho)\rho^2+k_*^2(t+c)(t+c)^2)\sinh\Phi\\
			&\leq  C(k_*^2(\rho)\rho^2+k_*^2(t+c)(t+c)^2)\sinh\Phi\\
			&\leq Ck_*^2(t+c)(t+c)^2\sinh\Phi.
		\end{split}
	\end{equation}
	Moreover, it follows from \eqref{e:sinh-cosh-tanh-phi-1} and \eqref{e:phi-1-derivative} that, for large $t$,
	\begin{equation}\label{e:sinhphi-theta}
		\left|\partial_\theta\left(\frac{1}{\sinh\Phi}\right)\right|= \left|-\frac{\cosh\Phi}{\sinh^2\Phi}\partial_\theta\Phi\right|\leq \frac{C|\partial_\theta\Phi|}{\sinh\Phi}\leq Ck^2_*(t+c)(t+c)^2.
	\end{equation}
	By \eqref{e:phi-1-derivative}, we have
	\begin{equation*} \partial_\theta\left(\frac{1}{\sinh\Phi}\right)=-\frac{\cosh\Phi}{\sinh^2\Phi}\partial_\theta\Phi
		=\frac{-\xi\partial_\rho G}{\sinh^2\Phi}-\xi(\tanh\Phi\partial_\rho\log G-\partial_t\log B)G\coth\Phi.
	\end{equation*}
	Taking one more derivative, we further get
	\begin{equation}\label{e:1-sinhphi-derivative-2}
		\begin{split}
			\partial_\theta^2\left(\frac{1}{\sinh\Phi}\right)&=\frac{-\xi\partial_\theta\partial_\rho G}{\sinh^2\Phi}
			+\frac{2\xi\partial_\rho G\cosh\Phi\partial_\theta\Phi}{\sinh^3\Phi}
			+\xi G\coth\Phi\cdot\partial_\theta(\dt\log B)\\
			&\qquad -\xi (\tanh\Phi\partial_\rho\log G-\dt\log B)(\partial_\theta G\coth\Phi
			-G\partial_\theta\Phi\sinh^{-2}\Phi)\\
			&\qquad -\xi(\partial_\theta\Phi\partial_\rho\log G\cosh^{-2}\Phi+\tanh\Phi\partial_\rho\partial_\theta\log G)G\coth\Phi.
		\end{split}
	\end{equation}
	Note that $|\dthe(\dt\log B)|$ is the key term in \eqref{e:1-sinhphi-derivative-2}.  Rewrite the Gauss equation \eqref{e:gauss-b} as
	$$\dt\dt\log B+(\dt\log B)^2=\frac{\dt^2 B}{B}=\kappa^2=k^2.$$
	Setting $H=\dt\log B$, by \eqref{e:kt-inte}, \eqref{eq-B-low-upper} and \eqref{eq-bt}, we get $H\leq Ct^{-1}$ for $t>0$. Obviously,
	$$\dt H+H^2=k^2. $$
	It also follows that
	$$\dt(t^2H)-2tH+t^2H^2=t^2k^2.$$
	Integrating the above equation yields
	\begin{equation}
		\label{e:bt2}
		t^2H=\lambda_0^2H(x, \lambda_0)+\int_{\lambda_0}^t\big(s^2k^2-s^2H^2+2sH\big)ds.
	\end{equation}
	Differentiating \eqref{e:bt2} with respect to $\theta$ and rearranging the terms on the right-hand side, we further have
	\begin{align*}
		t^2\partial_\theta H+2tHt_\theta
		&=\lambda_0^2\partial_xH(x,\lambda_0)x_\theta+\big(t^2k^2-t^2H^2+2tH\big)t_\theta\\
		&\qquad+\int_{\lambda_0}^t2(s^2\partial_\theta H)\left(\frac{1}{s}-H\right)ds+\int_{\lambda_0}^ts^22k\partial_\theta k ds.
	\end{align*}
	Since $\partial_xH(x,\lambda_0)$ is bounded for $|x|\leq c$, using $G\leq C\rho,$ $H\leq Ct^{-1}$, \eqref{e:t-theta-bound}, and, for $\rho\geq\lambda_0,$
	$$t^2|\partial_\theta k|\leq C\rho^2k_*(\rho)\leq C(t+c)^2k_*(t+c),$$
	we derive
	\begin{equation}\label{e:e:t-b-theta}
		\big(t^2|\partial_\theta H|\big)\leq \eta(t)+\int_{\lambda_0}^t\zeta(s)\cdot\big(s^2|\partial_\theta H|\big)ds,
	\end{equation}
	where
	\begin{align*}
		\eta(t)=C+C\int_{\lambda_0}^t(c+s)^2k_*^2(s+c) ds,\quad 	\zeta(t)=\left|\frac{2}{t}-2H\right|.
	\end{align*}
	Since $\eta(t)$ is an increasing function of $t$, we may apply Gronwall's inequality to \eqref{e:e:t-b-theta} to get
	\begin{align*}
		t^2|\partial_\theta H|\leq \eta(t)\exp\left\{\int_{\lambda_0}^t\zeta(s)ds\right\}
		\leq C\left(1+\int_{\lambda_0}^t(c+s)^2k_*^2(s+c)ds\right),
	\end{align*}
	where we used \eqref{e:bt-diff} to deduce
	\begin{align*}
		\int_{\lambda_0}^t\zeta(s)ds\leq 2\int_{\lambda_0}^t\left|\frac{1}{s}-H\right|ds\leq C.
	\end{align*}
	Therefore, we have
	\begin{equation}
		\label{e:bttheta-bound}
		\begin{split}
			|\partial_\theta(\dt\log B)|&=|\partial_\theta H|\leq
			\frac{C}{t^2}\left(1+\int_0^t(c+s)^2k_*^2(s+c)ds\right)\\
			&\leq  Ck_*^2(t+c)(t+c).
		\end{split}
	\end{equation}
	Furthermore, by \eqref{e:logb-xt}, we have, for $\rho\geq1,$
	\begin{equation}\label{loggrhothe}
		\begin{split}
			|\partial_\theta\partial_\rho\log G|&\leq\frac{|\dthe\drho\log G(\theta, 1)|}{G^2}+\frac{1}{G^2}\int_1^\rho k|\partial_\theta k|G^2ds\\
			&\leq  \frac{C}{\rho^2}+\frac{C}{\rho^2}\int_1^\rho k_*^2(s)s^2ds.
		\end{split}
	\end{equation}
	Then, for large $\rho$,
	\begin{equation}\label{e:log-g-rho-theta}
		|\partial_\theta\partial_\rho\log G|\leq Ck_*^2(\rho)\rho.
	\end{equation}
	Thus, by using \eqref{e:sinh-cosh-tanh-phi-1}, \eqref{eq-logg-rho-difference}, \eqref{e:bt-diff},
	\eqref{e:phi-derivative-bound}, \eqref{e:1-sinhphi-derivative-2}, \eqref{e:bttheta-bound}, and \eqref{e:log-g-rho-theta}, we obtain, for $t$ large enough,	
	\begin{equation}\label{e:sinhphi-theta-2}
		\begin{split}
			\left|\partial_\theta^2\left(\frac{1}{\sinh\Phi}\right)\right|&\leq \frac{C}{\rho^2}+\frac{Ck^2_*(t+c)(t+c)^2}{\sinh\Phi}+Ck_*^2(\rho)\rho^2+Ck_*^2(t+c)(t+c)^2\\
			&\leq Ck_*^2(t+c)(t+c)^2,
		\end{split}
	\end{equation}
	Combining \eqref{e:sinhphi-theta} and \eqref{e:sinhphi-theta-2}, we get \eqref{e:phi-x-less}.

	{\it Step 3. We prove \eqref{e:b-x-more}.}
	Note that $\rho^{1+\delta}k_*(\rho)$ is increasing if $\rho\geq \lambda_0$. By the triangle inequality: $\tfrac13(t+|x|)\leq \rho\leq t+|x|,$
	we can drive
	\begin{equation}\label{eq-k*-bound}
		k_*(\rho)\leq k_*(t+|x|)\frac{(t+|x|)^{1+\delta}}{\rho^{1+\delta}}\leq Ck_*(t+|x|).
	\end{equation}
	Hence, by \eqref{e:k2t-integrable}, we get, for any $|x|\geq c,$
	\begin{equation}\label{eq-krho-inte}
		\int_{\lambda_0}^t(\rho k_*^2(\rho))(s,x)ds\leq C\int_{\lambda_0}^t(s+|x|)k_*^2(s+|x|)ds=C\int_{\lambda_0+|x|}^{t+|x|}sk_*^2(s)ds\leq C.
	\end{equation}
	For any $|x|\geq2\lambda_0,$ $\rho\geq\frac12|x|\geq \lambda_0.$ Then by \eqref{eq-k*-bound}, we get
	\begin{equation}\label{eq-krho-inte-1}
		\int^{\lambda_0}_0(\rho k_*^2(\rho))(s,x)ds\leq C\int^{\lambda_0}_0(s+|x|)k_*^2(s+|x|)ds=C\int^{\lambda_0+|x|}_{|x|}sk_*^2(s)ds\leq C.
	\end{equation}
	Note that $\{c\leq|x|\leq2\lambda_0, 0\leq t\leq \lambda_0\}$ is a bounded region. If $c\leq |x|\leq 2\lambda_0,$ then  $$\int^{\lambda_0}_0\rho k_*^2(\rho)ds\leq C.$$
	Thus, combining \eqref{eq-krho-inte} and \eqref{eq-krho-inte-1}, we further have, for any $|x|\geq c,$
	\begin{equation}\label{eq-krho-inte-0}
		\int_0^t\rho k_*^2(\rho)ds\leq \int_0^{\lambda_0}\rho k_*^2(\rho)ds+\int_{\lambda_0}^t\rho k_*^2(\rho)ds\leq C.
	\end{equation}
	It follows from \eqref{eq-kappa-bound} and \eqref{eq-krho-inte-0} that
	\begin{equation}\label{eq-kappa2-integrable}
		\int_{0}^ts\kappa^2(x, s)ds\leq C\int_{0}^t\rho k^2_*(\rho)ds\leq C,
	\end{equation}
	and
	\begin{equation}\label{eq-kappa2-upper}
		\int_0^t\kappa^2(x, s)ds\leq C\int_0^t k^2_*(\rho)ds\leq \frac{C}{|x|}\int_0^t\rho k^2_*(\rho)ds\leq \frac{C}{|x|}.
	\end{equation}
	In view of \eqref{eq-hong-1.8}, \eqref{eq-kappa2-integrable}, and \eqref{eq-kappa2-upper}, we have
	\begin{equation*}%\label{eq-bt-bound}
		0\leq \dt B\leq \frac{C}{|x|}.
	\end{equation*}
	This implies the first inequality in \eqref{e:b-x-more} by integration over $[0, t]$.
	
	Next we estimate $\dthe\log B.$ First, for $|x|\leq 2\lambda_0$, we have $|\dthe\log B|\leq C$ as in the proof of \eqref{e:b-x-less}. Thus, we only need to prove for the case $|x|\geq2\lambda_0,$ in which we have $\rho\geq\lambda_0$ and then
	$$|t_\theta|=\frac{G}{\cosh\Phi}\leq \frac{C\rho|x|}{t+|x|}\leq C|x|,\quad  B\leq \frac{Ct}{|x|}+1, \quad k\leq Ck_*(t+|x|).$$
	Note that
	$$B=1+\int_0^t\int_0^sk^2 Bd\tau ds.$$
	Then,
	\begin{equation}\label{eq-b-theta}
		\partial_\theta B=\int_0^t\int_0^sk^2\dthe B d\tau ds+\int_0^t\int_0^s(\dthe k^2)B d\tau ds+t_\theta\int_0^t k^2 Bd\tau,
	\end{equation}
	and hence
	\begin{align*}
		|\dthe  B|&\leq C\int_0^t\!\int_0^s k_*^2(\tau+|x|)|\dthe B|d\tau ds+C\int_0^t\!\int_0^s k_*^2(\tau+|x|)\tau d\tau ds+C|x|\int_0^tk_*^2(\tau+|x|)\tau d\tau\\
		&\leq C\int_0^t\int_0^s k_*^2(\tau +|x|)|\dthe B|d\tau ds+C(t+|x|).
	\end{align*}
	Set
	$$\tilde F=\int_0^t\int_0^s k_*^2(\tau+|x|)|\dthe B|d\tau ds+C(t+|x|).$$
	Then,  $$\dt^2\tilde F=k_*^2(t+|x|)|\dthe B|\leq Ck_*^2(t+|x|)\tilde F.$$
	Thus, we can also employ the same iteration approach for deriving (1.8) in \cite{H}  to deduce
	\begin{align*}
		\dt\tilde F\leq C\int_0^tk_*^2(s+|x|)ds \exp\left\{\int_0^tCsk_*^2(s+|x|)ds\right\}\leq C.
	\end{align*}
	An integration yields $\tilde F\leq C(t+|x|)$. Then,  $|\dthe B|\leq \tilde F\leq C(t+|x|).$
	By $B\geq1$, we have
	\begin{equation*}
		%\label{e:logb-theta-bound-1}
		|\partial_\theta\log B|\leq\frac{|\dthe B|}{B}\leq C(t+|x|).
	\end{equation*}
	We obtain the second inequality in \eqref{e:b-x-more}.
	
	{\it Step 4. We prove \eqref{e:phi-x-more}.}
	%{\sf Proof of \eqref{e:phi-x-more}:}
	By $\rho\geq\frac12|x|\geq\frac c2,$ $\Phi$ has no singularity. Thus, by differentiating $\Phi$
	in \eqref{rhot} with respect to  $\theta$, we have
	\begin{equation}\label{e:phi-theta-x-large}
		\begin{split}
			\partial_\theta\Phi=\frac{\xi\partial_\rho G}{\cosh\Phi}+\int_0^t\partial_\theta\partial_\rho\log G d\tau.
		\end{split}
	\end{equation}
	Recall that, for $|x|\geq c$ and $  t>0$,
	\begin{equation}\label{e:sinh-cosh-tanh-phi-2}
		\sinh\Phi\geq \frac{1}{2}\left(\frac{t}{|x|}+\frac{t}{t+|x|}\right),\quad
		\cosh\Phi\geq\frac{1}{2}\left(\frac{t}{|x|}+1\right).
		%&0\leq 1-\tanh\Phi=1-\frac{e^{\Phi}-e^{-\Phi}}{e^{\Phi}+e^{-\Phi}}\leq\frac{2}{e^\Phi+e^{-\Phi}}\leq \frac{2|x|}{t+|x|}.
	\end{equation}
	In addition, similarly as in \eqref{loggrhothe},  from $\rho>\frac12c$ and $\rho(x,t)\geq\frac13(t+c)$, we have
	\begin{equation}\label{logg1}
		\int_0^t|\partial_\theta\partial_\rho\log G|d\tau\leq \int_0^t\left(\frac{C}{\rho^2}+\frac{C}{\rho^2}\int_{\frac12c}^\rho k_*^2(s)s^2ds\right)d\tau\leq C.
	\end{equation}
	By Lemma \ref{l:G-property} and \eqref{logg1},  we derive
	$$|\partial_\theta\Phi|\leq\frac{C|x|}{t+|x|}+C\leq C.$$
	%Besides, from \eqref{e:phi-theta-x-large}, $\partial_\theta\Phi$ can be viewed as a function of $(\theta, t).$
	Differentiating \eqref{e:phi-theta-x-large} with respect to $\theta$ implies
	\begin{align*}
		\partial_\theta^2\Phi
		=&\int_0^t\dthe^2\drho\log G d\tau+\dthe\drho\log G\frac{\xi G}{\cosh\Phi}+\frac{\xi\dthe\drho G}{\cosh\Phi}-\frac{\xi \drho G\sinh\Phi}{\cosh^2\Phi}\dthe\Phi.
	\end{align*}
	Following the same way for \eqref{logg1}, we can show
	$$\int_0^t|\dthe^2\drho\log G|\leq C.$$
	Thus, by \eqref{eq-krho-inte-0}, \eqref{e:sinh-cosh-tanh-phi-2}, and Lemma \ref{l:G-property},  we get, for large $t,$
	\begin{align*}
		|\partial^2_\theta\Phi|\leq C+\frac{Ck_*^2(\rho)\rho^2|x|}{t+|x|}+C+\frac{C|x|}{t+|x|}\leq C.
	\end{align*}
	Moreover, differentiating $\frac1{\sinh\Phi}$ with respect to $\theta,$ we also have
	\begin{align*}
		\left|\partial_\theta\left(\frac{1}{\sinh\Phi}\right)\right|&\leq\frac{\cosh\Phi}{\sinh^2\Phi}|\partial_\theta\Phi|\leq \frac{C}{\sinh\Phi\tanh\Phi},\\
		\left|\partial_\theta^2\left(\frac{1}{\sinh\Phi}\right)\right|&\leq\frac{\cosh\Phi}{\sinh^2\Phi}|\partial_\theta^2\Phi|+\frac{|\partial_\theta\Phi|^2|\tanh^2\Phi-2|}{\sinh\Phi\tanh^2\Phi}\leq \frac{C}{\sinh\Phi\tanh^2\Phi}.
	\end{align*}
	Therefore, we conclude \eqref{e:phi-x-more}.
\end{proof}

We next establish similar estimates for the case that $\rho^{1+\delta}k_*$ is decreasing.

\begin{lemma}\label{l:bphi-under-h1-h2-1} Assume that all the conditions in Theorem \ref{thrm-main1} are fulfilled and $\rho^{1+\delta}k_*$ is decreasing, for some $\delta\in (0,1/2)$. Let $c$ be an arbitrary positive constant. Then, for large $t$,
	\begin{enumerate}
		\item if $|x|\leq c,$
		\begin{equation}\label{e:b-x-less-1}
			C^{-1}{t}\leq B\leq  Ct,\quad  |\partial_\theta\log B|\leq  C,
		\end{equation}
		and
		\begin{equation}\label{e:phi-x-less-1}
			\begin{split}
				|\partial_\theta\Phi|\leq  C\rho^{-\delta}\sinh\Phi,\quad
				\left|\partial_\theta^i\left(\frac1{\sinh\Phi}\right)\right|\leq  C\rho^{-\delta},\quad i=1, 2;
			\end{split}
		\end{equation}
		\item if $|x|\geq c,$
		\begin{equation}\label{e:b-x-more-1}
			1\leq B\leq \frac{ Ct}{|x|}+1,\quad  |\partial_\theta\log B|\leq C(t+|x|).
		\end{equation}
		and
		\begin{equation}
			\label{e:phi-x-more-1}
			|\partial_\theta\Phi|\leq  C, \quad \left|\partial_\theta^i\left(\frac1{\sinh\Phi}\right)\right|
			\leq \frac{ C}{\sinh\Phi\tanh^i\Phi},
			\quad i=1, 2.
		\end{equation}
	\end{enumerate}
\end{lemma}
\begin{proof} By \eqref{eq-a-bul} %the conditions on $a$
	and the decreasing property of $\rho^{1+\delta}k_*$ for $\rho\geq \lambda_0$, we can derive
	$$\frac1Ck_*(\rho)\leq k(\theta, \rho)\leq Ck_*(\rho)\leq \frac{C}{\rho^{1+\delta}}\quad  \text{ for } \rho\geq \lambda_0.$$
	Since $|x|\leq c$, $t\leq \rho\leq t+c.$ For $t\geq \lambda_0$, we have
	$$\frac{k_*(t+c)}{C}\leq \kappa(x, t)=k(\theta, \rho)\leq \frac{C}{t^{1+\delta}}.$$
	Then,  $$\int_0^ts\kappa^2(s,x)ds\leq C.$$
	By \eqref{eq-hong-1.8}, we have $\frac{t}{C}\leq B\leq Ct$ for large $t$. In addition, if we take $t$ sufficiently large, then
	\begin{align*}
		\left|\frac{\dt B}{B}-\frac{1}{t}\right|\leq\frac{t^{-2}+Ct^{-2}\int_0^ts^2\kappa^2(x, s)ds}{\int_0^t\kappa^2ds+t^{-1}-Ct^{-1}\int_0^ts\kappa^2(x, s)ds}\leq \frac{C}{t^{1+\delta}}.
	\end{align*}
	By \eqref{e:bt-over-t},  for sufficiently large $\rho,$ we can also get
	\begin{align*}
		\left|\frac{\partial_\rho G}{G}-\frac{1}{\rho}\right|\leq \frac{C}{\rho^{1+\delta}}.
	\end{align*}
	By \eqref{e-Btheta}, we obtain
	\begin{align*}
		|\dthe B|\leq C\int_0^t\int_0^s\exp\left\{\int_0^\tau \frac{Czdz}{(1+z)^{2+2\delta}}\right\}\frac{d\tau}{(1+\tau)^{2+2\delta}} ds+Ct\leq Ct.
	\end{align*}
	Thus, $|\dthe\log B|\leq C$. Then, we conclude \eqref{e:b-x-less-1}. By \eqref{e-bound-phi}, we get
	$$|\partial_\theta\Phi|\leq C(t^{-1-\delta}+\rho^{-1-\delta})\rho\sinh\Phi\leq C \rho^{-\delta}\sinh\Phi.$$
	From \eqref{e:bt2}, we have
	\begin{align*}
		|\partial_\theta(\dt\log B)|\leq \frac{C}{(1+t)^2}\left(1+\int_0^t\frac{1}{(1+s)^{2\delta}}ds\right)\leq \frac{C}{(1+t)^{1+\delta}}.
	\end{align*}
	Similar to \eqref{e:sinhphi-theta} and \eqref{e:sinhphi-theta-2},
	we can derive the other part of \eqref{e:phi-x-less-1}.	Using $k_*(\rho)\leq \frac{C}{\rho^{1+\delta}}\leq \frac{C}{t^{1+\delta}}$
	for $\rho>t\geq \lambda_0$, with the same argument for \eqref{e:b-x-more} and \eqref{e:phi-x-more}, we can verify \eqref{e:b-x-more-1} and \eqref{e:phi-x-more-1}.
\end{proof}

Moreover, we analyze asymptotic behaviors of the boundary curves.

\begin{lemma}\label{l:dthe-drho}
	Assume that all the conditions in Theorem \ref{thrm-main1} are fulfilled. Then, for $\rho\geq R$ with $R$ sufficiently large,
	\begin{equation}\label{e:dtheta-drho}
		\rho|\theta_1'(\rho)|\,\text{and } \rho|\theta_2'(\rho)| \, \text{ are uniformly bounded on $\Gamma_1$}.
	\end{equation}
\end{lemma}

\begin{proof}  A direct calculation yields
	\begin{align*}
		\frac{d\theta}{d\rho}&=\frac{\theta_tt_0'(x)+\theta_x}{\rho_tt_0'(x)+\rho_x}=\frac{\frac{\xi t_0'(x)}{G\cosh\Phi}+\frac{B}{G}\tanh\Phi}{\tanh\Phi t_0'(x)-\frac{\xi B}{\cosh\Phi}}\\
		&=\frac{1}{G}\frac{\xi{t_0'(x)} +{B}\sinh\Phi}{{t_0'(x)}\sinh\Phi -\xi {B}}.
	\end{align*}
	Note that $x<b_-$ on $\Gamma_1$ and $x>b_+$ on $\Gamma_2$. On $\Gamma_1\cup\Gamma_2$, by
	\begin{align*}
		\sinh\Phi&=\frac{1}{2}(e^{\Phi}-e^{-\Phi})\geq\frac{1}{2}\left(\frac{t}{|x|}+1-e^{-\Phi}\right)\geq\frac{t}{2|x|},\\
		t_0'(x)&=\mu Rx(1+x^2)^{\frac{\mu}{2}-1}=\frac{\mu x t_0(x)}{1+x^2},
	\end{align*}
	we have
	\begin{align*}
		|t_0'(x)|\sinh\Phi\geq
		%&\frac{\mu |x|t_0(x)}{1+x^2}\frac{t_0(x)}{2|x|}\frac{|x|}&{t_0(x)+|x|}\frac{t_0(x)+|x|}{|x|}\\
		\frac{\mu |x|t_0(x)}{1+x^2}\frac{t_0(x)}{2|x|}
		\geq\frac{\mu R|x|^{\mu+1}}{4(1+x^2)}\frac{t_0(x)}{|x|}
		\geq\frac{R}{C}\frac{t_0(x)}{|x|},
	\end{align*}
	where we used $t_0(x)\geq R|x|^\mu$ and $\mu>1.$ By Lemmas \ref{l:bphi-under-h1-h2}-\ref{l:bphi-under-h1-h2-1}, we get $B\leq \frac{Ct}{|x|}+1 \text{ on }\Gamma_1\cup\Gamma_2.$ Taking  $R$ sufficiently large, we obtain $|t_0'(x)|\sinh\Phi\geq 2B.$ By $G\geq \rho,$  we further get
	\begin{align*}
		\left|\frac{d\theta_i(\rho)}{d\rho}\right|\leq\frac{2}{\rho}\frac{|t_0'(x)|+C(\frac{t_0(x)}{|x|}+1)\sinh\Phi}{|t_0'(x)|\sinh\Phi}\leq \frac{4}{R\rho}+\frac{C}{ \rho}\left(\frac{t_0(x)}{|x|}+1\right)\frac{1+x^2}{\mu |x|t_0(x)}\leq\frac{C}{\rho},
	\end{align*}
	for $i=1, 2.$ Then, \eqref{e:dtheta-drho} follows.
\end{proof}

\section{Initial Boundary Values}\label{s-boundary}

In this section, we will study the initial boundary data on $\partial\tilde\Omega_2$ in the geodesic polar coordinates, which are transformed from the solutions on $\partial\Omega_1=\partial\Omega_2$ in the geodesic coordinates.

Set
\[\begin{cases}
	(\bar p_1(x), \bar q_1(x))=(p(x, t_0(x)), q(x,t_0(x))),\text{ when }x\leq b_-<0,\\
	(\bar p_0(x), \bar q_0(x))=(p(x, t), q(x, t)),\text{ when }b_-<x<b_+,\,\rho(x, t)=R_1,\\
	(\bar p_2(x), \bar q_2(x))=(p(x, t_0(x)), q(x,t_0(x))),\text{ when }x> b_+>0,
\end{cases}
\]
which are the boundary data on $\partial\Omega_1.$ It is easy to see that  on $\partial\Omega_2=\partial \Omega_1$, $G\geq\rho\geq R$ and then $\textrm{J}_F>0$ due to \eqref{F}. Thus, we can make the coordinate transformation $F$. Denote $\tilde\Omega_2=F(\Omega_2)$.
\begin{figure}[H]
	\centering
	% Requires \usepackage{graphicx}
	\includegraphics[width=\textwidth]{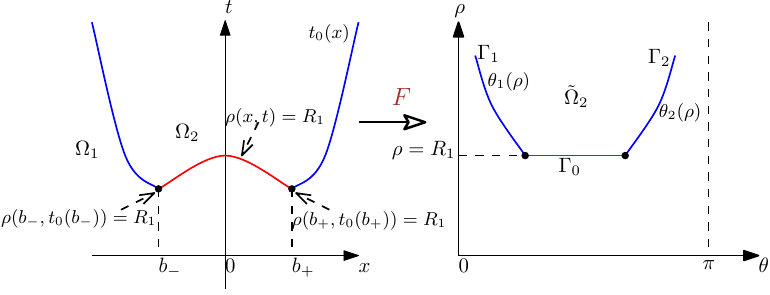}
	\caption{Regions in coordinate transformation}\label{f-region}
\end{figure}
\noindent
As shown in Figure \ref{f-region}, the boundary curve in the geodesic polar coordinate is $\partial\tilde\Omega_2=\Gamma_1\cup\Gamma_0\cup\Gamma_2$, with
\begin{align*}
	\Gamma_1=&\{(\theta_1(\rho), \rho): \rho>R_1\},\\
	\Gamma_0=&\{(\theta, R_1): \theta_1(R_1)\leq \theta\leq\theta_2(R_1)\},\\
	\Gamma_2=&\{(\theta_2(\rho), \rho): \rho>R_1\}.
\end{align*}
Here $\Gamma_1$ is the image of the curve $ \{(x, t_0(x)): x< b_-\}$ and  $\Gamma_2$ is the image of the curve $ \{(x, t_0(x)): x> b_+\}$.

To extend the solution from $\Omega_1$ to $\Omega_2$,  we shall also make variable transformation $\tilde{F}_*: (p,\, q)\rightarrow (\tilde{p},\, \tilde{q})$ near the boundary $\partial\Omega_2$ as given in \eqref{e:zw-diff}:
\begin{align}
	\frac{\tilde{p}}{G}=\tilde{F}_*(\frac{p}{B})=\frac{\theta_t+\frac{p}{B}\theta_x}{\rho_t+\frac{p}{B}\rho_x}, \quad \frac{\tilde{q}}{G}=\tilde{F}_*(\frac{q}{B})=\frac{\theta_t+\frac{q}{B}\theta_x}{\rho_t+\frac{q}{B}\rho_x}. \label{e:tilde-pq}\end{align}
Then,
\begin{align}	\tilde{q}-\tilde{p}=\frac{q-p}{(\rho_t+\frac{p}{B}\rho_x)(\rho_t+\frac{q}{B}\rho_x)}\,.\label{e:tilde-diff}
\end{align}
The boundary data after transformation $\tilde F_*$ will be given by
\begin{equation}\label{eq-boundary-data}
\begin{cases}
(G\tilde F_*(\frac{\bar p_1}B), G\tilde F_*(\frac{\bar q_1}B)):=(\tilde{\bar p}_1(\rho), \tilde{\bar q}_1(\rho)),\\
(G\tilde F_*(\frac{\bar p_0}B), G\tilde F_*(\frac{\bar q_0}B)):=(\tilde{\bar p}_0(\theta), \tilde{\bar q}_0(\theta)),\\
(G\tilde F_*(\frac{\bar p_2}B), G\tilde F_*(\frac{\bar q_2}B)):=(\tilde{\bar p}_2(\rho), \tilde{\bar q}_2(\rho)).
	\end{cases}
\end{equation}

Moreover, $\partial\tilde\Omega_2=\Gamma_1\cup\Gamma_0\cup\Gamma_2$  is space-like in $\rho$ for the above data.

\begin{lemma}\label{l-space}
	Assume that all the conditions in Theorem \ref{thrm-main1} are fulfilled. Let $(p,\, q)$ be the solution obtained in Lemma \ref{l:omega-1}. Then,
	\begin{equation}\label{eq-space-like}
		\begin{split}
			&\theta_1'(\rho)-\frac{\tilde {\bar p}_1(\rho)}{G(\theta_1(\rho),\rho)}, \quad
			\theta_1'(\rho)-\frac{\tilde {\bar q}_1(\rho)}{G(\theta_1(\rho),\rho)}\leq-\frac BG\frac{C|x|}{t_0(x)},\\
			&\theta_2'(\rho)-\frac{\tilde{\bar p}_2(\rho)}{G(\theta_2(\rho),\rho)},\quad
			\theta_2'(\rho)-\frac{\tilde{\bar q}_2(\rho)}{G(\theta_2(\rho),\rho)}\geq \frac BG\frac{C|x|}{t_0(x)}.
		\end{split}
	\end{equation}
\end{lemma}

\begin{proof}
	Viewing $x$ as an independent variable, we have
	\begin{align*}
		\theta_1'(\rho)-\frac{\tilde {\bar p}_1(\rho)}{G(\theta_1(\rho),\rho)}&=\frac{\theta_tt_0'(x)+\theta_x}{\rho_tt_0'(x)+\rho_x}-\frac{\theta_t+\frac{\bar p_1}{B}\theta_x}{\rho_t+\frac{\bar p_1}{B}\rho_x}\\
		&=\frac{(\theta_x\rho_t-\rho_x\theta_t)(1-\frac{\bar p_1}{B}t_0'(x))}{(\rho_tt_0'(x)+\rho_x)(\rho_t+\frac pB\rho_x)}\\
		&=\frac{\frac{B}{G}(1-\frac{\bar  p_1}{B}t_0'(x))}{(\rho_tt_0'(x)+\rho_x)(\rho_t+\frac pB\rho_x)}.
	\end{align*}
By \eqref{e:p-q-x-derivatives}, we get
\begin{equation}\label{eq-space-1}
\frac {|\bar p_1|}{B}|t_0'(x)|\leq\frac{C|t_0'(x)|}{((t_0(x)+R)h_2(x))^4}\leq\frac14.	\end{equation}
On the other hand, we have
\begin{align*}
	\rho_tt_0'(x)+\rho_x&=t_0'(x)\tanh\Phi -\frac{\xi B}{\cosh\Phi}\\
	&=-\frac{|x|t_0(x)}{1+x^2}-(1-\tanh\Phi)\frac{|x|t_0(x)}{1+x^2}-\frac{\xi B}{\cosh\Phi}.
\end{align*}
By Lemma \ref{l:phi}, Lemma \ref{l:bphi-under-h1-h2}, and Lemma \ref{l:bphi-under-h1-h2-1}, we get
\[B\leq \frac {Ct}{|x|}, \quad \cosh\Phi\geq \frac{1}{2}\frac{t}{|x|},\]
and hence, by taking $R$ sufficiently large,
\begin{equation}\label{eq-space-2}
	\rho_tt_0'(x)+\rho_x\geq-\frac{|x|t_0(x)}{2(1+x^2)}-C\geq-\frac{t_0(x)}{C|x|}.
	\end{equation}
Combining \eqref{eq-space-1}, \eqref{eq-space-2}, and Lemma \ref{l:transform-1}, we obtain
\begin{align*}
		\theta_1'(\rho)-\frac{\tilde {\bar p}(\rho)}{G(\theta_1(\rho),\rho)}\leq-\frac BG\frac{C|x|}{t_0(x)}.
\end{align*}
We can prove other inequalities similarly.
\end{proof}

Note that, by  \eqref{e:tilde-pq} and \eqref{eq-boundary-data}, we have
\begin{equation}\label{eq-boundary-data-1}
\begin{cases}
	(\tilde{\bar p}_1(\rho), \tilde{\bar q}_1(\rho))=(\tilde p(\theta_1(\rho), \rho), \tilde q(\theta_1(\rho), \rho)),\\
	(\tilde{\bar p}_0(\theta), \tilde{\bar q}_0(\theta))=(\tilde p(\theta, R_1), \tilde q(\theta, R_1)),\\
	(\tilde{\bar p}_1(\rho), \tilde{\bar q}_1(\rho))=(\tilde p(\theta_1(\rho), \rho), \tilde q(\theta_1(\rho), \rho)).
\end{cases}
\end{equation}
Moreover, by \eqref{e:tilde-pq}, we also have
\begin{equation}\label{e:tilde-pq-g}
	\begin{split}
		\tilde{p} &=G\tilde{F}_*\left(\frac{p}{B}\right)=\frac{(\sinh\Phi)^{-1}\xi+p}{1-(\sinh\Phi)^{-1}\xi p},\\
		\tilde{q} &=G\tilde{F}_*\left(\frac{q}{B}\right)=\frac{(\sinh\Phi)^{-1}\xi+q}{1-(\sinh\Phi)^{-1}\xi q}.
	\end{split}
\end{equation}
Then,
\begin{equation}\label{e:tilde-pq-kg}
	\begin{split}
		\frac{\tilde{p}}{kG}&=\frac{1}{kG}\frac{(\sinh\Phi)^{-1}\xi+p}{1-(\sinh\Phi)^{-1}\xi p},\\
		\frac{\tilde{q}}{kG}&=\frac{1}{kG}\frac{(\sinh\Phi)^{-1}\xi+q}{1-(\sinh\Phi)^{-1}\xi q}.
	\end{split}
\end{equation}
By
\eqref{e:p-q-x-derivatives}, we have $|p|\leq \frac{1}{R}$ on $\partial\Omega_2$.  By \eqref{e:sinh-cosh-tanh-phi}, we have $\sinh\Phi\geq \frac12R$  on $\Gamma_1\cup\Gamma_2$, and then
$$1-\frac{|\xi p|}{\sinh\Phi}\geq1-\frac{1}{2R}>\frac{1}{2}.$$

We are ready to analyze the initial boundary data in $(\theta, \rho)$.
We first consider the case that $\overline K$ is increasing.

\begin{lemma}\label{l:boundary-bounds}
Assume that all the conditions in Theorem \ref{thrm-main1} are fulfilled. Let $(p,\, q)$ be the solution obtained in Lemma \ref{l:omega-1}. Assume that $\overline{K}$ is increasing. Then, $\tilde p<\tilde q$ on $\partial\tilde\Omega_2$, and
	\begin{align}
		&|\dthe^i (\tfrac{\tilde{p}}{kG})|, |\dthe^i (\tfrac{\tilde{q}}{kG})|\leq CR^{-\delta/2}\rho^{-\delta/2},  \quad i=0, 1, 2, \text{ on } \Gamma_1\cup\Gamma_2, \label{e:uv-gamma1-increase}\\
		&|\theta_j'(\rho)||\dthe^i (\tfrac{\tilde{p}}{kG})|^2, |\theta_j'(\rho)| |\dthe^i (\tfrac{\tilde{q}}{kG})|^2\leq CR^{-\delta}\rho^{-1-\delta},  \quad i=0, 1, 2; j=1, 2, \text{ on } \Gamma_1\cup\Gamma_2, \label{e:uv-gamma1-increase-1}\\
		&|\dthe^i (\tfrac{\tilde{p}}{kG})|, |\dthe^i (\tfrac{\tilde{q}}{kG})|\leq Ck_*(2R_1)R,  \quad i=0, 1, 2, \text{ on } \Gamma_0. \label{e:uv-gamma2-increase}
	\end{align}
\end{lemma}

\begin{proof} 
	The proof below consists of several steps. First, by \eqref{e:tilde-diff} and Lemma \ref{l:omega-1}(i), we have $\tilde q>\tilde p$.

	{\it Step 1. We establish estimates on $\Gamma_1\cup\Gamma_2$.} We first estimate $\frac{p}{kG}$.
	Note that
	$$\rho(x, t_0(x))\leq|x|+t_0(x)\leq 2t_0(x)\quad\text{on }\Gamma_1\cup\Gamma_2.$$
	Using
	$$ |x|=\sqrt{\left(\frac{t_0(x)}{R}\right)^{\frac2\mu}-1}\leq\left(\frac{t_0(x)}{R}\right)^{\frac1\mu},
	$$
	and Lemma \ref{l:phi}, we have
	\begin{equation}\label{e:sinh-phi}
		\sinh\Phi\geq\frac{t_0(x)}{2|x|}\geq\frac12R^{\frac1\mu}(t_0(x))^{1-\frac1\mu}
		\geq\frac{R^{\frac1\mu }}{C}\rho^{1-\frac1\mu}.
	\end{equation}
	Note that $k_*(\rho)\rho^{1+\delta}\geq1$ for some $\delta\in(0, 1/2)$ and $\frac1\mu=1-\frac32\delta>\frac12\delta.$ By \eqref{e:sinh-phi}, we have
	\begin{equation}\label{eq-kgsinh}
		\frac{1}{kG\sinh\Phi}\leq\frac{C}{R^{\frac1\mu }k_*\rho^{2-\frac1\mu }}\leq\frac{C}{R^{\frac1\mu}\rho^{-1-\delta+2-\frac1\mu}}=
		\frac C{R^{\frac1\mu}\rho^{\frac\delta2}}\leq CR^{-\delta/2}\rho^{-\delta/2}.
	\end{equation}
	%where we used $\mu=\frac{2}{2-3\delta}$.
	Thus, by the  first equality in $\eqref{e:tilde-pq-kg}$, \eqref{e:p-q-x-derivatives}, and $t_0\leq \rho(t_0, x)\leq 2t_0,$ we get
	\begin{equation}
		\label{e:tilde-p-kg-c0}
		\frac{|\tilde{p}|}{kG}\leq \frac{C}{kG\sinh\Phi}+\frac{C|p|}{kG}\leq
		\frac{C}{kG\sinh\Phi}+\frac{C}{Rt_0^4k_*}\leq CR^{-\delta/2}\rho^{-\delta/2}.
	\end{equation}
	
	For the estimates of derivatives, a straightforward differentiation yields
	\begin{equation}\label{e:tilde-pq-kg-derivative}
		\begin{split}
			&\partial_\theta\left(\frac{\tilde{p}}{kG}\right)=\frac{-\dthe\log (kG)}{kG}\tilde{p}+\frac{\dthe\tilde{p}}{kG},\\
			&\partial_\theta^2\left(\frac{\tilde{p}}{kG}\right)=\frac{-\dthe^2\log (kG)}{kG}\tilde{p}+\frac{2\dthe\log (kG)\dthe\tilde{p}}{kG}+\frac{(\dthe\log (kG))^2}{kG}\tilde{p}+\frac{\dthe^2\tilde{p}}{kG}.
		\end{split}
	\end{equation}
	We need to calculate $\dthe\tilde{p}$ and $\dthe^2\tilde{p}$. By the first equality in $\eqref{e:tilde-pq-g}$ we get
	\begin{equation}\label{e:tilde-p-derivative-1}
		\begin{split}
			\dthe\tilde{p}&=\left(\dthe p+\xi\dthe\left(\frac{1}{\sinh\Phi}\right)\right)\left(1-\frac{\xi p}{\sinh\Phi}\right)^{-1}\\
			&\quad+\left(p+\frac{\xi}{\sinh\Phi}\right)\left(1-\frac{\xi p}{\sinh\Phi}\right)^{-2}\left(\frac{\xi\dthe p}{\sinh\Phi}+\xi p\dthe\left(\frac{1}{\sinh\Phi}\right)\right),
		\end{split}
	\end{equation}
	and
	\begin{equation}\label{e:tilde-p-derivative-2}
		\begin{split}
			\dthe^2\tilde{p}&=\left(\dthe^2 p+\xi\dthe^2\left(\frac{1}{\sinh\Phi}\right)\right)\left(1-\frac{\xi p}{\sinh\Phi}\right)^{-1}\\
			&\qquad+2\left(\dthe p+\xi\dthe\left(\frac{1}{\sinh\Phi}\right)\right)\left(1-\frac{\xi p}{\sinh\Phi}\right)^{-2}\left(\frac{\xi\dthe p}{\sinh\Phi}+\xi p\dthe\left(\frac{1}{\sinh\Phi}\right)\right)\\
			&\qquad+2\left(p+\frac{\xi}{\sinh\Phi}\right)\left(1-\frac{\xi p}{\sinh\Phi}\right)^{-3}\left(\frac{\xi\dthe p}{\sinh\Phi}+\xi p\dthe\left(\frac{1}{\sinh\Phi}\right)\right)^2\\
			&\qquad+\left(p+\frac{\xi}{\sinh\Phi}\right)\left(1-\frac{\xi p}{\sinh\Phi}\right)^{-2}\left(\frac{\xi\dthe^2 p}{\sinh\Phi}+\xi p\dthe^2\left(\frac{1}{\sinh\Phi}\right)\right)\\
			&\qquad+\left(p+\frac{\xi}{\sinh\Phi}\right)\left(1-\frac{\xi p}{\sinh\Phi}\right)^{-2}\left(2\xi\dthe p\dthe\left(\frac{1}{\sinh\Phi}\right)\right).
		\end{split}
	\end{equation}
	We estimate $\partial_\theta p$ and $\dthe^2 p$ first. In view of Lemma \ref{l:G-property} and  Lemma \ref{l:bphi-under-h1-h2}, we get, on $\partial\tilde{\Omega}_2$,
	\begin{equation}\label{e:p-theta-bound}
		\begin{split}
			|\partial_\theta p|&=|\dx p x_\theta+\dt p t_\theta|
			\leq\frac{C}{(t_0^4h_2^4)(x)}(|x_\theta|+h_1(x)|t_\theta|)\\
			&\leq \frac{C}{(t_0^4h_2^4)(x)}\left(\frac{G}{B}+\frac{Gh_1(x)}{\cosh\Phi}\right)
			\leq\frac{C(\rho+\rho h_1(x))}{(t_0^4h_2^4)(x)}
			\leq\frac{C}{\rho^2},
		\end{split}
	\end{equation}
	where we used $G\leq C\rho$, $B\geq1$, \eqref{e:p-q-x-derivatives} and
	\begin{equation*}%\label{eq-usecondition}
		t_0(x)\leq\rho(x, t_0(x))\leq 2t_0(x),\quad \cosh\Phi\geq\frac{t_0(x)+|x|}{2|x|}\geq R,\quad h_2(x)\geq h_1(x)+|x|^2.
	\end{equation*}
	Recall that
	$$x_\theta=\frac{G}{B}\tanh\Phi,\quad t_\theta=\frac{\xi G}{\cosh\Phi}.$$
	Then,
	\begin{align*}
		&x_{\theta\theta}=\frac{G}{B}\tanh\Phi(\dthe\log G-\dthe\log B)+\frac{G\dthe\Phi}{B\cosh^2\Phi},\\
		&t_{\theta\theta}=\frac{\xi\dthe G}{\cosh\Phi}-\frac{\xi G\dthe\Phi\sinh\Phi}{\cosh^2\Phi}.
	\end{align*}
	Hence,
	\begin{equation}\label{e:p-theta-bound-1}
		\begin{split}
			|\partial_\theta^2 p|&=|\dx^2 p x_\theta^2+\dt^2 p t_\theta^2+\dx p x_{\theta\theta}+\dt pt_{\theta\theta}|\\
			&\leq\frac{C}{(t_0^4h_2^4)(x)}(|x_\theta|^2+|x_{\theta\theta}|+h_1(x)(|t_\theta|^2+|t_{\theta\theta}|))\\
			&\leq\frac{C}{(t_0^4h_2^4)(x)}\left(\frac{G^2}{B^2}+\frac{G}{B}(|\partial_\theta\log G|
			+|\partial_\theta\log B|)+\frac{G|\partial_\theta\Phi|}{B\cosh^2\Phi} \right)\\
			&\qquad +\frac{Ch_1(x)}{(t_0^4h_2^4)(x)}\left(\frac{G^2}{\cosh^2\Phi}
			+\frac{G|\partial_\theta\Phi|\sinh\Phi}{\cosh^2\Phi}+\frac{|\partial_\theta G|}{\cosh\Phi}\right)\\
			&\leq \frac{C}{(t_0^4h_2^4)(x)}\left(\rho^2+(t+|x|)\rho
			+\frac{\rho}{R}+\frac{\rho^2h_1(x)|x|}{(t_0(x)+|x|)^2} \right)
			\leq\frac{C}{\rho^2}.
		\end{split}
	\end{equation}
	By $x<b_-<0$ on $\Gamma_1$ and $x>b_+>0$ on  $\Gamma_2$, \eqref{e:b-x-more} and \eqref{e:phi-x-more} hold on  $\Gamma_1\cup\Gamma_2$.
	By \eqref{e:tilde-p-derivative-1}-\eqref{e:p-theta-bound-1}, %\eqref{e:tilde-p-derivative-2}, \eqref{e:p-theta-bound} and \eqref{e:p-theta-bound-1},
	we can derive
	\begin{align*}
		|\dthe\tilde{p}|&\leq C|\dthe p|+C\left|\dthe\left(\frac{1}{\sinh\Phi}\right)\right|\leq \frac{C}{\rho^2}+\frac{C}{\sinh\Phi},\\
		|\dthe^2\tilde{p}|&\leq  C\left(|\dthe p|+C\left|\dthe\left(\frac{1}{\sinh\Phi}\right)\right|\right)^2+C|\dthe^2 p|+C\left|\dthe^2\left(\frac{1}{\sinh\Phi}\right)\right|\\
		&\leq\frac{C}{\rho^2}+\frac{C}{\sinh\Phi}+\frac{C}{\rho^2\sinh\Phi}.
	\end{align*}
	With Lemma \ref{l:G-property}, \eqref{e:tilde-pq-kg-derivative}, \eqref{eq-kgsinh}, and \eqref{e:sinh-phi}, we further get
	\begin{equation}\label{e:tilde-p-kg-c12}
		\left|\partial_\theta\left(\frac{\tilde{p}}{kG}\right)\right|, \left|\partial_\theta^2\left(\frac{\tilde{p}}{kG}\right)\right|\leq\frac{C}{k_*(\rho)\rho\sinh\Phi}+\frac{C}{k_*(\rho)\rho^3}\leq CR^{-\delta/2} \rho^{-\delta/2}.
	\end{equation}
	%Analogously,
	Similarly, we also have
	\begin{equation}\label{e:tilde-q-kg-c012}
		\left|\partial_\theta^i\left(\frac{\tilde{q}}{kG}\right)\right|\leq CR^{-\delta/2}\rho^{-\delta/2},\, i=0, 1, 2.
	\end{equation}
	By \eqref{e:tilde-p-kg-c0}, \eqref{e:tilde-p-kg-c12}, and \eqref{e:tilde-q-kg-c012}, we obtain \eqref{e:uv-gamma1-increase}. By Lemma \ref{l:dthe-drho} and \eqref{e:uv-gamma1-decrease}, we get \eqref{e:uv-gamma1-increase-1}.
	\smallskip
	
	{\it Step 2. We establish estimates on $\Gamma_0$.}
	%{\sf Bounds on $\Gamma_0$:}
	On $\Gamma_0,$ $\rho=R_1$ and then $\sinh\Phi\geq R$. Similar to \eqref{e:tilde-p-kg-c0}, we have by \eqref{e:p-q-x-derivatives}  that
	\begin{equation}\label{e:tilde-p-kg-c0-1}
		\left|\frac{\tilde{p}}{kG}\right|\leq \frac{C}{k_*(R_1)R_1^2}+\frac{C}{k_*(R_1)R_1^5}\leq Ck_*(2R_1)R,
	\end{equation}
	since $\rho^{1+\delta}k_*(\rho)$ is increasing. Using (1) in Lemma \ref{l:bphi-under-h1-h2} (since $|x|\leq b=\max\{|b_-|, b_+\}$ on $\Gamma_0$), we have, on $\Gamma_0$,
	\begin{align*}
		&|\dthe\Phi|\leq Ck_*^2(R_1+b)(R_1+b)^2\sinh\Phi,\\
		&\left|\partial_\theta^i\left(\frac{1}{\sinh\Phi}\right)\right|\leq Ck_*^2(R_1+b)(R_1+b)^2.
	\end{align*}
	In the same way as estimating $\dthe p$ and $\dthe^2 p$ on $\Gamma_1\cup\Gamma_2$, we can derive
	\begin{align*}
		&|\dthe p|\leq \frac{CR}{R_1^4h_2^4(x)}(1+h_1(x))\leq \frac{C}{R^2},\\
		&|\dthe^2 p|\leq \frac{CR^2}{R_1^4h_2^4(x)}(1+k_*^2(R_1+b)(R_1+b)^2+h_1(x)k_*^2(R_1+b)(R_1+b)^2)\leq \frac{C}{R^2}.
	\end{align*}
	Then, by \eqref{e:tilde-p-derivative-1}-\eqref{e:tilde-p-derivative-2}, we get
	\begin{align*}
		|\dthe \tilde p|, |\dthe^2 \tilde p| \leq \frac{C}{R^2}+ Ck_*^2(R_1+b)(R_1+b)^2.
	\end{align*}
	Thus, by \eqref{e:tilde-pq-kg-derivative}, we have
	\begin{equation}
		\label{e:tilde-p-kg-c12-1}
		\left|\partial_\theta\left(\frac{\tilde{p}}{kG}\right)\right|, \left|\partial_\theta^2\left(\frac{\tilde{p}}{kG}\right)\right|\leq \frac{C}{k_*(R_1)R_1R^2}+\frac{Ck_*^2(R_1+b)(R_1+b)^2}{k_*(R_1)R_1}\leq Ck_*(2R_1)R.
	\end{equation}
	Similarly, we have, on $\Gamma_0$,
	\begin{equation}\label{e:tilde-q-kg-c012-1}
		\left|\partial_\theta^i\left(\frac{\tilde{q}}{kG}\right)\right|\leq Ck_*(2R_1)R, \quad i=0, 1, 2.
	\end{equation}
	By \eqref{e:tilde-p-kg-c0-1}, \eqref{e:tilde-p-kg-c12-1}, and \eqref{e:tilde-q-kg-c012-1}, we obtain \eqref{e:uv-gamma2-increase}.
\end{proof}

We next consider the case that $\overline K$ is decreasing.

\begin{lemma}\label{l:boundary-bounds-1}
	Assume that all the conditions in Theorem \ref{thrm-main1} are fulfilled. Let $(p,\, q)$ be the solution obtained in Lemma \ref{l:omega-1}.  Assume that $\overline{K}$ is decreasing. Then, $\tilde p<\tilde q$ on $\partial\tilde\Omega_2$,
	\begin{align}
		&|\dthe^i \tilde p|, |\dthe^i \tilde q| \leq CR^{-\delta/2}\rho^{-\delta},  \quad i=0, 1, 2, \text{ on } \Gamma_1\cup\Gamma_2, \label{e:uv-gamma1-decrease}\\
		&|\theta'_j(\rho)|\dthe^i \tilde p|^2, |\theta'_j(\rho)||\dthe^i \tilde q|^2\leq CR^{-\delta
		}\rho^{-1-2\delta},  \quad i=0, 1, 2; j=1, 2, \text{ on } \Gamma_1\cup\Gamma_2, \label{e:uv-gamma1-decrease-1} \\
		&|\dthe^i \tilde p|, |\dthe^i \tilde q|\leq CR^{-\delta},  \quad i=0, 1, 2, \text{ on } \Gamma_0. \label{e:uv-gamma2-decrease}
	\end{align}
\end{lemma}

\begin{proof} %[Proof of Lemma \ref{l:boundary-bounds-1}]
	The proof below consists of several steps.
	It is easy to derive $\tilde q>\tilde p$ by \eqref{e:tilde-diff} and Lemma \ref{l:omega-1}(i).
	
	{\it Step 1. We establish estimates on $\Gamma_1\cup\Gamma_2$.} 	
	%{\sf Bounds on $\Gamma_1\cup\Gamma_2$:}
	We estimate $\tilde{p}$ first.
	By the first equality in $\eqref{e:tilde-pq-g}$, Lemma \ref{l:omega-1} (iii), \eqref{e:sinh-phi}, and $\frac1\mu=1-\frac32\delta>\frac12\delta,$ we have
	\begin{equation}
		\label{e:p-tilde-gamma1-c0}
		|\tilde{p}|\leq C|p|+\frac{C}{\sinh\Phi}\leq \frac{C}{(t_0^4h_2^4)(x)}+\frac{C}{R^{\frac1\mu }\rho^{1-\frac1\mu }}\leq CR^{-\delta/2}\rho^{-\delta}.
	\end{equation}
	Similarly to  \eqref{e:p-theta-bound} and \eqref{e:p-theta-bound-1}, we have
	\begin{align*}
		|\dthe p|\leq\frac{C}{\rho^2},\quad
		%\leq&\frac{C\sigma}{(t_0^4h^4_2)(x)}\frac{\rho|x|}{t_0(x)+|x|}(1+h_1(x))\leq\frac{C\sigma}{\rho^2}.\\
		|\dthe^2 p|\leq\frac{C}{\rho^2}.
		%\leq&\frac{C\sigma}{(t_0^4h_2^4)(x)}\frac{\rho^2|x|}{(t_0(x)+|x|)^2}(1+h_1(x))\leq\frac{C\sigma}{\rho^2}.
	\end{align*}
	Thus, by \eqref{e:sinh-phi}, \eqref{e:tilde-p-derivative-1}-\eqref{e:tilde-p-derivative-2}, and (2) in Lemma \ref{l:bphi-under-h1-h2-1},
	we can derive
	\begin{equation}\label{e:p-tilde-gamma1-c12}
		|\dthe\tilde{p}|, |\dthe^2\tilde{p}|\leq \frac{C}{\rho^2}+\frac{C}{\sinh\Phi}\leq CR^{-\delta/2}\rho^{-\delta},
	\end{equation}
	and similarly,
	\begin{equation}\label{e:q-tilde-gamma1-c012}
		|\dthe^i\tilde{q}|\leq CR^{-\delta/2}\rho^{-\delta}, \quad i=0, 1, 2.
	\end{equation}
	Using  \eqref{e:p-tilde-gamma1-c0}, \eqref{e:p-tilde-gamma1-c12}, and \eqref{e:q-tilde-gamma1-c012}, we obtain \eqref{e:uv-gamma1-decrease}. Then, \eqref{e:uv-gamma1-decrease-1} follows from Lemma \ref{l:dthe-drho} and  \eqref{e:uv-gamma1-decrease}.

	{\it Step 2. We establish estimates on $\Gamma_0$.} 	
	%{\sf Bounds on $\Gamma_0$:}
	By the first equality in $\eqref{e:tilde-pq-g},$ \eqref{e:p-q-x-derivatives}, and $\sinh\Phi\geq R$, we have
	\begin{equation}
		\label{e:p-tilde-gamma2-c0}
		|\tilde{p}|\leq \frac{C}{R_1^4}+\frac{C}{R}\leq CR^{-\delta}.
	\end{equation}
	By (1) in Lemma \ref{l:bphi-under-h1-h2-1}, we get
	\begin{align*}
		|\dthe\Phi|\leq CR_1^{-\delta}\sinh\Phi,\end{align*}
	and then
	\begin{align*}	\left|\dthe^i\left(\frac{1}{\sinh\Phi}\right)\right|\leq CR_1^{-\delta},\quad i=0, 1, 2.
	\end{align*}
	Similarly to  \eqref{e:p-theta-bound}-\eqref{e:p-theta-bound-1}, %for the estimates of $\dthe p$ and $\dthe^2 p$ on $\Gamma_1$,
	we can derive
	\begin{align*}
		|\dthe^ip|\leq \frac{C}{R_1^4h_2^4(x)}(R^2+R_1^{-\delta}h_1(x))\leq\frac{C}{R^2}, \quad i=1, 2.
	\end{align*}
	By \eqref{e:tilde-p-derivative-1}-\eqref{e:tilde-p-derivative-2}, we get
	\begin{equation}
		\label{e:p-tilde-gamma2-c12}
		\begin{split}
			|\dthe\tilde{p}|&\leq C|\dthe p|+C\left|\dthe\left(\frac{1}{\sinh\Phi}\right)\right|
			\leq CR^{-\delta},\\
			|\dthe^2\tilde{p}|&\leq C |\dthe^2p|+C\left|\dthe^2\left(\frac{1}{\sinh\Phi}\right)\right|
			+C\left(|\dthe p|+\left|\dthe\left(\frac{1}{\sinh\Phi}\right)\right|\right)^2
			\leq CR^{-\delta}.
		\end{split}
	\end{equation}
	Similarly, we deduce
	\begin{equation}
		\label{e:q-tilde-gamma2-c012-1}
		|\dthe^i\tilde{q}|\leq CR^{-\delta}, \quad i=0, 1, 2.
	\end{equation}
	With \eqref{e:p-tilde-gamma2-c0}, \eqref{e:p-tilde-gamma2-c12}, and \eqref{e:q-tilde-gamma2-c012-1} established, we can obtain \eqref{e:uv-gamma2-decrease}.
\end{proof}

%%%%%%%%%%
\section{The Local Existence of Solutions}\label{s-local}
%%%%%%%%%%

In this section and the next, we will prove the existence of solutions in $\tilde\Omega_2$ in Sobolev spaces,
extending solutions near $\partial\tilde\Omega_2$ to the entire $\tilde\Omega_2$.
We will discuss the local existence in this section and the global existence in the next section.

Note that in $\tilde \Omega_2$ the transformed solution
$$\tilde p=G\tilde F(\frac pB)=Gw,\quad \tilde q=G\tilde F(\frac qB)=Gz,$$
where $w$ and $z$ are from \eqref{1.4.1-1}, satisfies
 \begin{align}
	\dis \drho \tilde p+\frac{\tilde q}{G}\dthe \tilde p&=\frac{\tilde p-\tilde q}{2}\Big(\drho+\frac{\tilde p}{G}\dthe\Big)\log k-\tilde q(1+\tilde p^2)\drho\log G,%:=\mathfrak{f}_1(\tilde p,\tilde q,\theta,\rho),
	\label{e:tilde-p}
	\\
	\dis \drho \tilde q+\frac{\tilde p}{G}\dthe \tilde q&=\frac{\tilde q-\tilde p}{2}\Big(\drho+\frac{\tilde q}{G}\dthe\Big)\log k-\tilde p(1+\tilde q^2)\drho\log G.%:=\mathfrak{f}_2(\tilde p,\tilde q,\theta,\rho).
	\label{e:tilde-q}
\end{align}
We will  study an initial boundary value problem of the system \eqref{e:tilde-p}-\eqref{e:tilde-q} in $\tilde\Omega_2$
with the data \eqref{eq-boundary-data} prescribed on  $\partial\tilde\Omega_2$. As discussed before, if $\overline{K}$ is increasing, then $\rho^{1+\delta}k_*$ is increasing and $G^2k_*$ grows. We set
\begin{equation}\label{uv-1}
	u=\frac{\tilde{p}}{kG}+\frac{\tilde{q}}{kG}=\frac1k(w+z), \quad v=\frac{\tilde{q}}{kG}-\frac{\tilde{p}}{kG}=\frac1k(z-w),
\end{equation}
which satisfy \eqref{eq-u-general}-\eqref{eq-v-general} with $(\alpha,\beta)=(0,1).$
If  $\overline{K}$ is decreasing, then $\rho^{1+\delta}k_*$ is decreasing and $kG^2$ may decay. We then set
\begin{equation}\label{uv-2}
	u=\tilde{p}+\tilde{q}=G(w+z), \quad v=\tilde{q}-\tilde{p}=G(z-w),
\end{equation}
which satisfy \eqref{eq-u-general}-\eqref{eq-v-general} with $(\alpha,\beta)=(1,0)$. Hence, we will solve the system \eqref{e:tilde-p}-\eqref{e:tilde-q}  in the region $\tilde\Omega_2$ by solving the initial boundary value problem of the system \eqref{eq-u-general}-\eqref{eq-v-general} with the data prescribed on  $\partial\tilde\Omega_2$:
	\begin{equation}\label{eq-initial}
	\begin{cases}
		u(\theta_1(\rho),\rho)=(\tilde{\bar p}_1+\tilde{\bar q}_1)G^{\alpha-1}k^{-\beta}(\theta_1(\rho), \rho),\,\,&
		v(\theta_1(\rho),\rho)=(\tilde{\bar q}_1-\tilde{\bar p}_1)G^{\alpha-1}k^{-\beta}(\theta_1(\rho), \rho),\\
		u(\theta,R_1)=(\tilde{\bar p}_0+\tilde{\bar q}_0)G^{\alpha-1}k^\beta(\theta, R_1),\,\,&
		v(\theta,R_1)=(\tilde{\bar q}_0-\tilde{\bar p}_0)G^{\alpha-1}k^\beta(\theta, R_1),\\
		u(\theta_2(\rho),\rho)=(\tilde{\bar p}_2+\tilde{\bar q}_2)G^{\alpha-1}k^{-\beta}(\theta_2(\rho), \rho),\,\,&
		v(\theta_2(\rho),\rho)=(\tilde{\bar q}_2-\tilde{\bar p}_2)G^{\alpha-1}k^{-\beta}(\theta_2(\rho), \rho),
	\end{cases}
\end{equation}
with $(\tilde{\bar p}_i, \tilde{\bar q}_i), (i=0, 1, 2)$ given by \eqref{eq-boundary-data} or \eqref{eq-boundary-data-1}.

\smallskip
In the rest of this section, we do not need to distinguish $R$ and $R_1$, so we omit the sub-index $`` 1"$ of $R_1$. For simplicity we also write
\begin{align}\label{IC}\begin{split}
		(u_0, v_0)(\theta)&=(u, v)(\theta, R),\\
		(u_i^0, v_i^0)(\rho)&=(u, v)(\theta_i(\rho),\rho), \,\,i=1, 2.
\end{split}\end{align}
We also write $u(\rho)=u(\cdot, \rho)$, $v(\rho)=v(\cdot, \rho)$,  etc. Obviously, by Lemma \ref{l:boundary-bounds} and Lemma \ref{l:boundary-bounds-1}, we have
\[|(u_0, v_0)|,\, |(u_1^0, v_1^0)|,\, |(u_2^0, v_2^0)|\leq 1.\]
Set
$$I(\rho)=[\theta_1(\rho), \theta_2(\rho)],$$
and, for any $n\in\mathbb{N}$,
\begin{align*}
	\|(u(\rho), v(\rho))\|_n&=\|u(\rho)\|_n+\|v(\rho)\|_n,\\
	|(u(\rho), v(\rho))|_n&=|u(\rho)|_n+|v(\rho)|_n,
\end{align*}
where, for $\mathfrak{u}=u$ or $\mathfrak{u}=v$
\begin{align*}
	\|\mathfrak{u}(\rho)\|_n&:=\sum_{i=0}^n\left(\int_{I(\rho)}(\dthe^i\mathfrak{u})^2(\theta, \rho)d\theta\right)^{1/2}, %\quad
	%\|v(\rho)\|_n:=\sum_{i=0}^n\left(\int_{I(\rho)}(\dthe^iv)^2(\theta, \rho)d\theta\right)^{1/2},
	\\
	|\mathfrak{u}(\rho)|_n&:=\max\left\{\max_{\theta\in I(\rho)}|\dthe^i \mathfrak{u}|, i=0, 1, \cdots, n\right\}. %,\quad
	%|v(\rho)|_n:=\max\left\{\max_{\theta\in I(\rho)}|\dthe^i v|, i=0, 1, \cdots, n\right\}.
\end{align*}
We often write $\|\cdot\|_0$ for $\|\cdot\|$.

Set, for any $\vsi>0$,
$$\Sigma_\vsi=\{(\theta,\rho):\rho\in[R, R+\vsi], \,\theta\in I(\rho)\},$$
and
\begin{equation*}
	\begin{split}
		M_\vsi=\sup_{(\theta,\rho)\in \Sigma_\vsi}\Big\{&|\dthe^i(G^{-\alpha}k^\beta)|,\, |\dthe^i(\partial_\rho\log(G^{2-\alpha}k^\beta))|,\, |\dthe^i(\partial_\rho\log(G^{-\alpha}k^{\beta-1}))|,\\
		&|\dthe^i(G^{-\alpha}k^\beta\partial_\theta\log(G^{1-\alpha}k^\beta))|,\, |\dthe^i(G^{-\alpha}k^\beta\partial_\theta\log(G^{1-\alpha}k^{1+\beta}))|,\\
		&|\dthe^i(G^{1-2\alpha}\partial_\rho Gk^{2\beta})|, \, |\dthe^i(G^{-\alpha}k^\beta\partial_\theta\log k)| , \, i=0, 1, 2\Big\}.
	\end{split}
\end{equation*}

\begin{lemma}\label{l:local}
	There exists a small constant $\vsi>0$, depending on the boundary data in \eqref{eq-initial} (or \eqref{IC}),
	such that  there exists a unique solution $(u,v)\in C^1([R, R+\vsi]; H^2(I(\rho)))^2$ to the system \eqref{eq-u-general}-\eqref{eq-v-general} with \eqref{eq-initial}, satisfying
	\begin{equation}\label{e:local-estimate}
		\begin{split}
			\|(u(\rho), v(\rho))\|_2
			&\leq 4Q_0 \quad \text{for any } \rho\in[R, R+\vsi],\\
			v(\theta, \rho)&>0\quad \text{if } (\theta, \rho)\in \Sigma_\vsi,
		\end{split}
	\end{equation}
	where $Q_0$ is given by
	$$Q_0=\|(u_0, v_0)\|_2+C_0\max\{\|(u_i^0, v_i^0)\|_{H^2([R, R+\vsi])},\,i=1, 2 \},$$ for some constant $C_0$. %Here $\Sigma=\{(\theta, \rho); \rho\in[R, R+\vsi], \theta\in[\theta_1(\rho), \theta_2(\rho)]\}$.
\end{lemma}

\begin{proof}
	We will apply the fixed point theorem to prove Lemma \ref{l:local}.  Set
	$$S(\vsi)=\left\{U=(u, v)\in C^1([R, R+\vsi]; H^2(I(\rho)))^2: \, \max_{\rho\in[R,R+\vsi]}\|(u(\rho), v(\rho))\|_2\leq 4Q_0\right\},$$
	and let $U_1=(u_1, v_1)$ be a smooth extension defined in $\Sigma_\vsi$
	such that $\|U_1\|_2\leq 2Q_0.$ For $n\geq1,$ assume that $U_n$ is constructed in $S(\vsi)$ and let $U_{n+1}=(u_{n+1}, v_{n+1})$ be the solution of the following linear system:
	\begin{align}\label{e:un}
		\partial_\rho u_{n+1}+\frac{1}{2}G^{-\alpha}k^\beta u_{n}\partial_\theta u_{n+1}-\frac {1}{2}G^{-\alpha}k^\beta v_{n}\partial_\theta v_{n+1}
		&=\mathfrak{f}_1(U_n, U_{n+1}, \theta, \rho),\\
		\label{e:vn}
		\partial_\rho v_{n+1}+\frac{1}{2}G^{-\alpha}k^\beta u_{n}\partial_\theta v_{n+1}-\frac {1}{2}G^{-\alpha}k^\beta v_{n}\partial_\theta u_{n+1}
		&=\mathfrak{f}_2(U_n, U_{n+1}, \theta, \rho),
	\end{align}
	where
	\begin{align*}
		\mathfrak{f}_1(U_n, U_{n+1}, \theta, \rho)&=-u_{n+1}\partial_\rho\log(G^{2-\alpha}k^\beta)
		-\frac{1}{2}G^{-\alpha}k^\beta u_nu_{n+1}\partial_\theta\log(G^{1-\alpha}k^\beta)\\
		&\qquad+\frac12G^{-\alpha}k^\beta v_nv_{n+1}\partial_\theta\log(G^{1-\alpha}k^{1+\beta})
		-\frac14G^{1-2\alpha}\partial_\rho Gk^{2\beta}(u_{n}^2-v_{n}^2)u_{n+1}, \\
		\mathfrak{f}_2(U_n, U_{n+1}, \theta, \rho)
		&=-v_{n+1}\partial_\rho\log(G^{-\alpha}k^{\beta-1})+\frac{1}{2}G^{-\alpha}k^\beta u_{n}v_{n+1}\partial_\theta\log k\\
		&\qquad-\frac14G^{1-2\alpha}\partial_\rho Gk^{2\beta}(u_{n}^2-v_{n}^2)v_{n+1},
	\end{align*}
	for  $(\theta, \rho)\in \Sigma_\vsi$,
	with the same smooth initial boundary  value as in \eqref{IC}, i.e.,
	\begin{align*}U_{n+1}|_{\rho=R}&=(u_0, v_0),\\
		U_{n+1}(\theta_i(\rho)), \rho))&=(u_i^0, v_i^0), \ i=1, 2.\end{align*}
	Set
	\[w^+_{n+1}=u_{n+1}+v_{n+1},\quad w^-_{n+1}=u_{n+1}-v_{n+1}.\]
	Then,
	\begin{align}\label{eq-pqn}\begin{split}
			\drho w^+_{n+1}+\frac12G^{-\alpha}k^\beta w^-_n\dthe w^+_{n+1}
			&=\mathfrak{f}_*(w_n^+, w_n^-, w_{n+1}^+, w_{n+1}^-, \theta,\rho),\\
			\drho w^-_{n+1}+\frac12G^{-\alpha}k^\beta w^+_n\dthe w^-_{n+1}
			&=\mathfrak{f}_{**}(w_n^+, w_n^-, w_{n+1}^+, w_{n+1}^-, \theta,\rho),
	\end{split}\end{align}
	where
	\begin{align*}
		\mathfrak{f}_*(w_n^+, w_n^-, w_{n+1}^+, w_{n+1}^-, \theta,\rho)&=(\mathfrak{f}_1+\mathfrak{f}_2)(U_n, U_{n+1}, \theta, \rho),\\
		\mathfrak{f}_{**}(w_n^+, w_n^-, w_{n+1}^+, w_{n+1}^-, \theta,\rho)&=(\mathfrak{f}_1-\mathfrak{f}_2)(U_n, U_{n+1}, \theta, \rho),
	\end{align*}
	with the initial boundary data
	\[w_{n+1}^+\big|_{\Gamma_i}=(u_{n+1}+v_{n+1})\big|_{\Gamma_i},\,\,
	w_{n+1}^-\big|_{\Gamma_i}=(u_{n+1}-v_{n+1})\big|_{\Gamma_i},\quad i=0, 1, 2.
	\]
	This is a generalized Cauchy problem of a linear hyperbolic system for $(w_{n+1}^+, w_{n+1}^-)$. Indeed, by \eqref{eq-space-like} and \eqref{eq-initial}, we have
	\begin{align*}
		&\theta_1'(\rho)-\frac12G^{-\alpha}k^\beta w^-_n(\theta_1(\rho),\rho)=\theta_1'(\rho)-\tilde{\bar p}_1G^{-1}(\theta_1(\rho),\rho)\le-\frac{CB|x|}{Gt_0(x)},\\
		&\theta_1'(\rho)-\frac12G^{-\alpha}k^\beta w^+_n(\theta_1(\rho),\rho)=\theta_1'(\rho)-\tilde{\bar q}_1G^{-1}(\theta_1(\rho),\rho)\le-\frac{CB|x|}{Gt_0(x)},
	\end{align*}
	and
	\begin{align*}
		&\theta_2'(\rho)-\frac12G^{-\alpha}k^\beta w^-_n(\theta_2(\rho),\rho)=\theta_2'(\rho)-\tilde{\bar p}_2G^{-1}(\theta_2(\rho),\rho)\geq \frac{CB|x|}{Gt_0(x)},\\
		&\theta_2'(\rho)-\frac12G^{-\alpha}k^\beta w^+_n(\theta_2(\rho),\rho)=\theta_2'(\rho)-\tilde{\bar q}_2G^{-1}(\theta_2(\rho),\rho)\geq \frac{CB|x|}{Gt_0(x)}.
	\end{align*}
	Hence, $\Gamma_1$ and $\Gamma_2$ are space-like. By integrating along characteristic curves, we obtain the local existence of sufficiently smooth $U_{n+1}.$ This is standard for the generalized Cauchy problem of linear hyperbolic systems (cf. Chapter 5.4 in \cite{GuLi}).
	
	Next we prove $U_{n+1}\in S(\vsi)$. Multiplying $\eqref{e:un}$ and $\eqref{e:vn}$ by $u_{n+1}$ and $v_{n+1}$, respectively, summing the resulting equations, and integrating by parts, we get
	\begin{equation}\label{eq-u0}
		\frac{d}{d\rho}\|U_{n+1}\|^2\leq C(|U_n|_0, M_\vsi)\|U_{n+1}\|^2+C(M_\vsi)\sum_{i=1}^2(|\theta_i'|+|u_i^0|+|v^0_i|)(|u_i^0|^2+|v^0_i|^2).
	\end{equation}
	Multiplying $\dthe\eqref{e:un}$ and $\dthe\eqref{e:vn}$ by $\dthe u_{n+1}$ and $\dthe v_{n+1}$, respectively, summing the resulting equations, and integrating by parts, we get
	\begin{align*}
		\frac{d}{d\rho}\|\dthe U_{n+1}\|^2&\leq
		C(|U_n|_1, M_\vsi)(\|\dthe U_{n+1}\|^2+\|U_{n+1}\|^2)\\
		&\qquad+\sum_{i=1}^2|\dthe U_{n+1}|^2(\theta_i(\rho), \rho)(|\theta_i'(\rho)|+|u_i^0|+|v^0_i|).
	\end{align*}
	We will estimate $\dthe U_{n+1}(\theta_i(\rho), \rho).$ Note that
	\begin{align*}
		(u^0_1)'(\rho)=\frac{d}{d\rho}u_{n+1}(\theta_1(\rho),\rho)=\dthe u_{n+1}(\theta_1(\rho),\rho)\theta'_1(\rho)+\drho u_{n+1}(\theta_1(\rho),\rho).
	\end{align*}
	Solving $\drho u_{n+1}(\theta_1(\rho),\rho)$ from \eqref{e:un}, we have
	\begin{equation}\label{eq-u01}
		(u^0_1)'(\rho)=(\theta'_1(\rho)-\frac{1}{2}G^{-\alpha}k^\beta u_{n})\dthe u_{n+1}+\frac {1}{2}G^{-\alpha}k^\beta v_{n}\partial_\theta v_{n+1}+\mathfrak{f}_1.
	\end{equation}
	Similarly,
	\begin{equation}\label{eq-v01}
		(v^0_1)'(\rho)=(\theta'_1(\rho)-\frac{1}{2}G^{-\alpha}k^\beta u_{n})\dthe v_{n+1}
		+\frac {1}{2}G^{-\alpha}k^\beta v_{n}\partial_\theta u_{n+1}+\mathfrak{f}_2.
	\end{equation}
	All the functions on the right-hand sides of \eqref{eq-u01} and \eqref{eq-v01} take values on $(\theta_1(\rho),\rho).$
	By \eqref{eq-u01} and \eqref{eq-v01}, it is  straightforward to get
	\begin{align*}
		\dthe u_{n+1}=\frac{(u^0_1)'(\rho)-\mathfrak{f}_1+(v^0_1)'(\rho)-\mathfrak{f}_2}{2(\theta'_1(\rho)-\frac{1}{2}G^{-\alpha}k^\beta (u_{n}-v_n))}
		+\frac{(u^0_1)'(\rho)-\mathfrak{f}_1-(v^0_1)'(\rho)+\mathfrak{f}_2}{2(\theta'_1(\rho)-\frac{1}{2}G^{-\alpha}k^\beta( u_{n}+v_n))},\\
		\dthe v_{n+1}=\frac{(u^0_1)'(\rho)-\mathfrak{f}_1+(v^0_1)'(\rho)-\mathfrak{f}_2}{2(\theta'_1(\rho)-\frac{1}{2}G^{-\alpha}k^\beta (u_{n}-v_n))}
		-\frac{(u^0_1)'(\rho)-\mathfrak{f}_1-(v^0_1)'(\rho)+\mathfrak{f}_2}{2(\theta'_1(\rho)-\frac{1}{2}G^{-\alpha}k^\beta( u_{n}+v_n))}.
	\end{align*}
	As a consequence, $\dthe U_{n+1}(\theta_1(\rho), \rho)$ can be expressed in terms of $((u^0_1)', (v^0_1)')$ and $(u^0_1, v^0_1)$. Then, we have
	\[|\dthe U_{n+1}(\theta_1(\rho), \rho)|\leq C(M_\vsi)(|(u^0_1, v^0_1)|+|((u^0_1)',(v^0_1)')|).\]
	Similarly, we can derive
	\[|\dthe U_{n+1}(\theta_2(\rho), \rho)|\leq C(M_\vsi)(|(u^0_2, v^0_2)|+|((u^0_2)',(v^0_2)')|).\]
	Hence,
	\begin{equation}\label{eq-u1}
		\begin{split}
			\frac{d}{d\rho}\|\dthe U_{n+1}\|^2&\leq
			C(|U_n|_1, M_\vsi)(\|\dthe U_{n+1}\|^2+\|U_{n+1}\|^2)\\
			&\qquad+\sum_{i=1}^2C(M_\vsi)(|(u^0_i, v^0_i)|^2+|((u^0_i)',(v^0_i)')|^2)(|\theta_i'(\rho)|+|u_i^0|+|v^0_i|).
		\end{split}
	\end{equation}
	Multiplying $\dthe^2\eqref{e:un}$ and $\dthe^2\eqref{e:vn}$ by $\dthe^2 u_{n+1}$ and $\dthe^2 v_{n+1}$, respectively, summing the resulting equations, and integrating by parts, we get
	\begin{align*}
		\frac{d}{d\rho}\|\dthe^2 U_{n+1}\|^2&\leq
		C(|U_n|_1, \|U_n\|_2, M_\vsi)(\|\dthe^2 U_{n+1}\|^2+\|\dthe U_{n+1}\|^2+\| U_{n+1}\|^2)\\
		&\qquad+C\sum_{i=1}^2|\dthe^2 U_{n+1}|^2(\theta_i(\rho), \rho)(|\theta_i'(\rho)|+|u_i^0|+|v^0_i|).
	\end{align*}
	To calculate $\dthe^2 U_{n+1}^2(\theta_i(\rho), \rho)$, we differentiate $\dthe U(\theta_i(\rho), \rho)$ with respect to $\rho$ and then solve $\drho\dthe U(\theta_i(\rho),\rho)$ from $\dthe\eqref{e:un}$ and $\dthe\eqref{e:vn}$. Thus,
	\[|\dthe^2 U_{n+1}(\theta_i(\rho), \rho)|\leq C(M_\vsi)(|(u^0_i, v^0_i)|+|((u^0_i)',(v^0_i)')|+|((u^0_i)'',(v^0_i)'')|).\]
	Hence,
	\begin{equation}\label{eq-u2}
		\begin{split}
			\frac{d}{d\rho}\|\dthe^2 U_{n+1}\|^2&\leq
			C(|U_n|_1, \|U_n\|_2, M_\vsi)(\|\dthe^2 U_{n+1}\|^2+\|\dthe U_{n+1}\|^2+\| U_{n+1}\|^2)\\
			&\qquad+\sum_{i=1}^2C(M_\vsi)|((u^0_i)'',(v^0_i)'')|^2(|\theta_i'(\rho)|+|u_i^0|+|v^0_i|)\\
			&\qquad+\sum_{i=1}^2C(M_\vsi)(|(u^0_1, v^0_1)|^2+|((u^0_1)',(v^0_1)')|^2)(|\theta_i'(\rho)|+|u_i^0|+|v^0_i|).
		\end{split}
	\end{equation}
	Combining \eqref{eq-u0}, \eqref{eq-u1}, and \eqref{eq-u2} together, we obtain
	\begin{align*}
		\frac{d}{d\rho}\|U_{n+1}\|_2^2\leq C(\|U_n\|_2, M_\vsi)\|U_{n+1}\|_2^2+\sum_{i=1}^2\sum_{j=0}^2C(M_\vsi)(|\tfrac{d^j}{d\rho^j}(u^0_i, v^0_i)|^2)(|\theta_i'(\rho)|+|u_i^0|+|v^0_i|).
	\end{align*}
	Note that $\theta'_i(\rho)$ are uniformly bounded and $|u_i^0|\leq 1, |v_i^0|\leq 1$.
	A simple integration yields
	\begin{equation}\label{eq-U-h2}
		\|U_{n+1}\|_2^2\leq
		e^{\vsi C(\|U_n\|_2, M\vsi)}\left(\|(u_0, v_0)\|_{H^2(I(R))}^2 +C(M_\vsi)\sum_{i=1}^2\|(u^0_i, v^0_i)\|^2_{H^2([R,R+\vsi])}\right).
	\end{equation}
	Since $U_n\in S(\vsi)$, then $\|U_{n}(\rho)\|_2\leq 4Q_0,$
	and hence
	\begin{equation}\label{eq-U-2}
		\|U_{n+1}\|_2\leq
		e^{C(M_\vsi, Q_0)\vsi}\left(\|(u_0, v_0)\|_{H^2(I(R))} +C_0\sum_{i=1}^2\|(u^0_i, v^0_i)\|_{H^2([R,R+\vsi])}^2\right)\leq 4Q_0,
	\end{equation}
	by choosing $\vsi$ sufficiently small.

	Next, we prove that $\{U_n\}$ is a contraction sequence in $S(\vsi)$. By a similar argument, we have
	\begin{align*}
		&\frac12\frac{d}{d\rho}\|U_{n+2}-U_{n+1}\|_2^2
		\leq CM_\vsi Q_0^2\|U_{n+2}-U_{n+1}\|_2^2\\
		&\qquad +CM_\vsi Q_0^2\|U_{n+2}-U_{n+1}\|_2\|U_{n+1}-U_{n}\|_2.
	\end{align*}
	Since $(U_{n+2}-U_{n+1})|_{\partial\tilde\Omega_2}=0$,
	an integration over $[R, R+\vsi]$ yields
	\begin{align*}
		\|U_{n+2}-U_{n+1}\|_2\leq \max_{\rho\in[R,R+\vsi]}\|U_{n+1}-U_{n}\|_2(e^{\vsi C M_\vsi Q_0^2}-1)
		\leq\frac12\max_{\rho\in[R,R+\vsi]}\|U_{n+1}-U_{n}\|_2.
	\end{align*}
	Hence,
	\begin{align*}
		\max_{\rho\in[R,R+\vsi]}\|U_{n+2}-U_{n+1}\|_2\leq\frac12\max_{\rho\in[R,R+\vsi]}\|U_{n+1}-U_{n}\|_2,
	\end{align*}
	provided that $\vsi$ is small enough. Therefore, the fixed point theorem yields a unique limit  $U\in C^0([R, R+\vsi]; H^2(I(\rho)))$ of $\{U_n\}$  that solves \eqref{eq-u-general}-\eqref{eq-v-general} with \eqref{IC}. Moreover, we can show $U\in C^1([R, R+\vsi]; H^2(I(\rho)))$ by considering the equations for $U$.
	
	For the lower bound of $v,$ we rewrite \eqref{eq-v-general} as
	\begin{equation}\label{e:v-rewrite}
		\drho v+\frac{1}{2}G^{-\alpha}k^\beta u\dthe v=v\mathcal{F},
	\end{equation}
	where
	\begin{align*}
		\mathcal{F}=\frac {1}{2}G^{-\alpha}k^\beta\partial_\theta u-\partial_\rho\log(G^{-\alpha}k^{\beta-1})+\frac{1}{2}G^{-\alpha}k^\beta u\partial_\theta\log k+\frac14B^{1-2\alpha}\partial_\rho Gk^{2\beta}(u^2-v^2).
	\end{align*}
	Integrating \eqref{e:v-rewrite} along the characteristic curve $(X(\tau; \theta, \rho), \tau)$ defined by
	\begin{align*}
		\begin{cases}
			\displaystyle\frac{d}{d\tau}X(\tau; \theta, \rho)=\frac12(G^{-\alpha}k^\beta)(X, \tau)u(X, \tau),\\
			\displaystyle X(\rho; \theta, \rho)=\theta,
		\end{cases}
	\end{align*}
	we get
	\begin{align*}
		v(X(\rho; \theta, \rho), \rho)&=v(X(0; \theta, \rho), 0)\exp\left\{\int_R^\rho \mathcal{F}(X(\tau; \theta, \rho), \tau)d\tau\right\}\\
		&\geq v(X(0; \theta, \rho), 0)e^{- C_0M_\vsi Q^2_0\vsi}>0,
	\end{align*}
	where  $|\mathcal{F}|_0\leq CM_\vsi Q^2_0$.
\end{proof}

\section{The Global Existence of Solutions}\label{s-global}

In this section, we prove the global existence of solutions in $\tilde\Omega_2$ via an energy approach as explained in Section \ref{s-strategy}, which yields solutions in $\Omega_2,$  and then complete the proof of Theorem \ref{thrm-main1}.
We will adopt the same notations as in Section \ref{s-local}.

We obtain solutions in $\tilde\Omega_2$ by solving an initial boundary value problem for \eqref{eq-u-general}-\eqref{eq-v-general} with  \eqref{eq-initial} in the entire region $\Sigma_\infty$ given by
\[\Sigma_\infty=\{(\theta, \,\rho): \rho\in[R,\infty), \theta\in I(\rho)\}.\]
Since the local existence in the region $\Sigma_\vsi$ has already been obtained in Lemma \ref{l:local}, it remains to derive {\it a priori} estimates so that solutions can be extended to the region $\Sigma_\infty$.  Note that \eqref{eq-u-general}-\eqref{eq-v-general} is a symmetric hyperbolic system. We will employ the energy method to establish {\it a priori} estimates.
This is significantly different from the approach in \cite{H} that  mainly relies on the comparison principle.
Proposition \ref{p-priori} is the main result in this section
and is proved with help of the finite total curvature condition \eqref{e:inte-cond},
the monotonicity of $\overline K$, and \eqref{eq-a-bul}-\eqref{eq-a-bv}.

\smallskip

Given the estimates satisfied by the boundary data on $\partial\tilde\Omega_2$ in Lemma \ref{l:boundary-bounds} and Lemma \ref{l:boundary-bounds-1},  we set
\begin{equation}\label{r_1}
\eps=\eps(R)=C\max\{R^{-\delta/2}, k_*(2R)R\}.
\end{equation}
Furthermore, set
\begin{equation}\label{varphi}
\va(\rho)=\exp\bigg\{\int_R^\rho\bigg(\max_{\theta}\bigg|\frac{\partial_\tau a}{a}\bigg|+2\max_{\theta}\bigg|\frac{\partial_\tau G}{G}-\frac1\tau\bigg|\bigg)d\tau\bigg\}.\end{equation}
Introduce constants $\Lambda$, $\Theta$, and $A_0$ by
$$\Lambda=\max\bigg\{\va(\infty), \,\,1+\!\!\int_{R}^\infty \tau  k_*^2(\tau)d\tau,\,\,\, \sum_{i=1}^2\sup_{\theta, \rho}(2\rho|\dthe^i\drho\log G|+2\rho|\dthe^i\drho\log a|)\bigg\},$$
and
\begin{equation}\label{A}
	A_0=(10\Theta)^8, \quad  \Theta=\max\bigg\{\frac{\Lambda}{(1-2\delta)\delta}, {C_0+1}\bigg\},
\end{equation}
{where $C_0$ is from \eqref{e:local-estimate}.} It is not difficult to check that $\Lambda$ is finite by \eqref{eq-a-bul}-\eqref{eq-a-bv} and Lemma \ref{l:G-property}. 

{We proceed to prove the {\it a priori} estimates through the continuity argument. Our aim is to show that for any $R'>R$ it holds that 
\begin{equation}\label{e:priori-h2}
	\|(u(\rho), v(\rho))\|_2\leq \frac12A_0\eps\quad\text{for any }\rho\in[R, R'],
\end{equation}
provided $\eps$ is small enough independent of $R'.$ Due to the local existence  Lemma \ref{l:local}, 
$\|(u(\rho), v(\rho))\|_2$ is continuous with respect to $\rho\in[R,R']$ if  $R'-R$ is small. Besides, Lemma \ref{l:local} also implies the set of $R'$ satisfying \eqref{e:priori-h2} is nonempty and relatively open in $[R,\infty)$. To show that the set is also relatively closed, and thus is the entire interval $[R,\infty)$,  it suffices to prove the following proposition.
}

\begin{proposition}\label{p-priori}
For any $R'>R$, assume that $(u, v)\in C^1([R, R'], H^2(I(\rho)))$ is the  solution  to \eqref{eq-u-general}-\eqref{eq-v-general}, % in $C^1([R, R'], H^2(I(\rho)))$
with the initial boundary data  \eqref{eq-initial}. Then, there exists a small constant $\eps_1$, independent of $R'$, such that, for $\eps$ and $A_0$ given by \eqref{r_1} and  \eqref{A}, respectively, if $0<\eps\le\eps_1$ and
\begin{equation}\label{e:priori-h2-assumption-uv}
%\overline{M}_{R'}:=\sup_{R\le \rho\le R'}
\|(u(\rho), v(\rho))\|_2\leq A_0\eps\quad\text{for any }\rho\in[R, R'],
\end{equation}
then
\begin{equation*}%\label{e-h2}
\|(u(\rho), v(\rho))\|_2\leq \frac1{2}A_0\eps \quad\text{for any }\rho\in[R, R'].
\end{equation*}  
\end{proposition}

%The {\it a priori} assumption \eqref{e:priori-h2-assumption-uv} is motivated
%by \eqref{e:local-estimate}. 
With the assumption \eqref{e:priori-h2-assumption-uv} and the Sobolev embedding on $I(\rho)\subset[0, \pi]$, we have
\begin{equation}
\label{e:priori-c1-assumption-uv}
|(u, v)(\rho)|_1\leq C\eps,
\end{equation}
where $C$ is a constant depending only on $A_0.$

We first prove a useful lemma.

\begin{lemma}\label{l:gronwall}
	Let $f(t)$ and $h(t)$ be positive continuous functions and $E(t)$ be a $C^1$ positive function on $[t_0, \infty)$ satisfying
	\begin{equation}\label{e:e-inequality}
		\frac{d E(t)}{dt}\leq f(t)\sqrt{E(t)}+h(t).
	\end{equation}
	Then,
	\begin{equation*}%\label{e:e-bound}
		\sqrt{E(t)}\leq \sqrt{E(t_0)}+\sqrt{\int_{t_0}^th(s)ds}+\frac{1}{2}\int_{t_0}^tf(s)ds.
	\end{equation*}
	%and also
	%\begin{equation}\label{e:e-bound-1}
	%E(t)\leq 2E(t_0)+2\int_{t_0}^th(s)ds+\left(\int_{t_0}^tf(s)ds\right)^2.
	%\end{equation}
\end{lemma}

\begin{proof}
	Integrating  \eqref{e:e-inequality} over $(t_0,t)$ yields
	\begin{equation*}
		%\label{e:e-integral-bound}
		E(t)\leq E(t_0)+\int_{t_0}^t(f(s)\sqrt{E(s)}+h(s))ds=:V(t).
	\end{equation*}
Then,
%	\begin{align*}
%		V'(t)=f(t)\sqrt{E(t)}+h(t)\leq f(t)\sqrt{V(t)}+h(t)
%	\end{align*}
%and
	\begin{align*}
		\frac{V'(t)}{2\sqrt{V(t)}}=\frac{\sqrt{E(t)}}{2\sqrt{V(t)}} f(t)+\frac{h(t)}{2\sqrt{V(t)}}
		\leq \frac{1}{2}f(t)+\frac{h(t)}{2\sqrt{E(t_0)+\int_{t_0}^th(s)ds}},
	\end{align*}
and hence
	\begin{equation*}%\label{e:v-bound}
		\begin{split}
\sqrt{E(t)}&\le\sqrt{V(t)}\leq \sqrt{E(t_0)+\int_{t_0}^th(s)ds}+\frac{1}{2}\int_{t_0}^tf(s)ds\\
&\le \sqrt{E(t_0)}+\sqrt{\int_{t_0}^th(s)ds}+\frac{1}{2}\int_{t_0}^tf(s)ds.
\end{split}
	\end{equation*}
This is the desired estimate.
\end{proof}

Recall the initial boundary data in \eqref{eq-initial}:
	\begin{equation*}
		\begin{cases}
			u(\theta_1(\rho),\rho)=(\tilde{\bar p}_1+\tilde{\bar q}_1)G^{\alpha-1}k^{-\beta}(\theta_1(\rho), \rho),\,\,&
			v(\theta_1(\rho),\rho)=(\tilde{\bar q}_1-\tilde{\bar p}_1)G^{\alpha-1}k^{-\beta}(\theta_1(\rho), \rho),\\
			u(\theta,R_1)=(\tilde{\bar p}_0+\tilde{\bar q}_0)G^{\alpha-1}k^\beta(\theta, R_1),\,\,&
			v(\theta,R_1)=(\tilde{\bar q}_0-\tilde{\bar p}_0)G^{\alpha-1}k^\beta(\theta, R_1),\\
			u(\theta_2(\rho),\rho)=(\tilde{\bar p}_2+\tilde{\bar q}_2)G^{\alpha-1}k^{-\beta}(\theta_2(\rho), \rho),\,\,&
			v(\theta_2(\rho),\rho)=(\tilde{\bar q}_2-\tilde{\bar p}_2)G^{\alpha-1}k^{-\beta}(\theta_2(\rho), \rho),
		\end{cases}
	\end{equation*}
	with $(\tilde{\bar p}_i, \tilde{\bar q}_i), (i=0, 1, 2)$ given by \eqref{eq-boundary-data} or \eqref{eq-boundary-data-1}, and the simplified notation
	\begin{align*}
		\begin{split}
			(u_0, v_0)(\theta)&=(u, v)(\theta, R),\\
			(u_i^0, v_i^0)(\rho)&=(u, v)(\theta_i(\rho),\rho), \,i=1, 2.
		\end{split}
\end{align*}

As we can see in the proof of Lemma \ref{l:local}, we require the estimates of $\dthe^j(u, v)(\theta_i(\rho), \rho)$ for $i=1, 2$ and $j=0,1, 2.$
Thus, we need to derive estimates of these functions by Lemma \ref{l:boundary-bounds} and Lemma \ref{l:boundary-bounds-1}.
\begin{lemma}\label{l:derivative-boundary}
For any $R'>R$, assume that $(u, v)\in C^1([R, R'], H^2(I(\rho)))$ is the  solution  to \eqref{eq-u-general}-\eqref{eq-v-general}, % in $C^1([R, R'], H^2(I(\rho)))$
with the initial boundary data  \eqref{eq-initial}, and satisfies \eqref{e:priori-h2-assumption-uv}. Set
\[(u_i^j, v_i^j)=(\partial_\theta^ju(\theta_i(\rho),\rho), \partial_\theta^jv(\theta_i(\rho),\rho)).\]
 Then, $v>0$ on $\partial\tilde\Omega_2$ and
	\begin{equation}\label{e:initial-omega}
		\begin{split}
			&|\dthe^j u_0|, \, |\dthe^j v_0|\leq \eps\quad \text{for } \theta_1(R)\leq \theta\leq\theta_2(R), \, j=0, 1, 2;\\
			&|u_i^j|,\,|v_i^j|\leq \eps\psi_1(\rho), \quad |\theta'_i(\rho)||u_i^j|^2,\,|\theta'_i(\rho)||v_i^j|^2\leq \eps^2\psi_2(\rho) \quad\text{for } i=1, 2; \, j=0, 1, 2,
		\end{split}
	\end{equation}
	where
	\begin{align*}
	\psi_1(\rho)=
	\begin{cases}
		\rho^{-\delta/2}&\text{ if }\overline K \text{ is increasing},\\
		\rho^{-\delta} &\text{ if }\overline K \text{ is decreasing},
	\end{cases}
\end{align*}
and
\begin{align*}
\psi_2(\rho)=
\begin{cases}
	\rho^{-1-\delta} &\text{ if }\overline K \text{ is increasing},\\
	\rho^{-1-2\delta} &\text{ if }\overline K \text{ is decreasing}.
\end{cases}
\end{align*}
\end{lemma}

\begin{proof}
By Lemma \ref{l:boundary-bounds} and Lemma \ref{l:boundary-bounds-1}, we have $v>0$ since $\tilde p<\tilde q$ on $\partial\tilde\Omega_2$. In the following, we only consider the case that $\overline K$ is increasing and  $i=1$, since other cases can be proved similarly. % and by Lemma \ref{l:boundary-bounds-1}.
Now, $\alpha=0, \beta=1$
	and we write \eqref{eq-u-general}-\eqref{eq-v-general} as
	\begin{align*}
		\drho u+\frac{1}{2}k(u\dthe u-v\dthe v)&=\mathfrak{f}_{3}(u, v,\theta,\rho),\\
		\drho v+\frac{1}{2}k(u\dthe v-v\dthe u)&=\mathfrak{f}_{4}(u, v, \theta,\rho),
	\end{align*}
where
	\begin{align*}
		\mathfrak{f}_{3}(u, v,\theta,\rho)&=-u\partial_\rho\log (kG^2)-\frac{1}{2}ku^2\dthe \log (kG)\\
		&\qquad+\frac{1}{2}kv^2\dthe \log (k^2G)-\frac{1}{4}G\drho Gk^2u(u^2-v^2),\\
		\mathfrak{f}_{4}(u, v, \theta,\rho)&=\frac{1}{2}kuv\dthe \log k-\frac{1}{4}G\drho Gk^2v(u^2-v^2).
	\end{align*}
First, by Lemma \ref{l:boundary-bounds} we have the first line of $\eqref{e:initial-omega}$ and also the second line of $\eqref{e:initial-omega}$ with $j=0$. To calculate $\dthe U(\theta_1(\rho), \rho)$, we differentiate $u_1^0=u(\theta_1(\rho), \rho)$ to get
	\begin{align*}
		(u^0_1)'(\rho)=\frac{d}{d\rho}u(\theta_1(\rho),\rho)=\dthe u(\theta_1(\rho),\rho)\theta'_1(\rho)+\drho u(\theta_1(\rho),\rho).
	\end{align*}
By solving $\drho u(\theta_1(\rho),\rho)$ from \eqref{eq-u-general}, we have
	\begin{equation}\label{eq-u-1}
		(u^0_1)'(\rho)=(\theta'_1(\rho)-\frac{1}{2}k u)\dthe u+\frac {1}{2}kv\partial_\theta v+\mathfrak{f}_3.
	\end{equation}
Similarly,
	\begin{equation}\label{eq-v-1}
		(v^0_1)'(\rho)=(\theta'_1(\rho)-\frac{1}{2}k u)\dthe v
		+\frac {1}{2}k v\partial_\theta u+\mathfrak{f}_4.
	\end{equation}
All functions on the right-hand sides of \eqref{eq-u-1} and \eqref{eq-v-1} take values on $(\theta_1(\rho),\rho).$ By \eqref{eq-u-1} and \eqref{eq-v-1}, it is  straightforward to check
	\begin{equation}\label{eq-utheta}
		\begin{split}
		\dthe u&=\frac{(u^0_1)'(\rho)-\mathfrak{f}_3|_{\theta=\theta_1(\rho)}+(v^0_1)'(\rho)-\mathfrak{f}_4|_{\theta=\theta_1(\rho)}}{2\big(\theta'_1(\rho)-\frac{1}{2}k(u-v)|_{\theta=\theta_1(\rho)}\big)}\\
&\qquad		+\frac{(u^0_1)'(\rho)-\mathfrak{f}_3|_{\theta=\theta_1(\rho)}-(v^0_1)'(\rho)+\mathfrak{f}_4|_{\theta=\theta_1(\rho)}}{2\big(\theta'_1(\rho)-\frac{1}{2}k( u+v)|_{\theta=\theta_1(\rho)}\big)},
		\end{split}
	\end{equation}
and
	\begin{equation}\label{eq-vtheta}
	\begin{split}
	\dthe v&=\frac{(u^0_1)'(\rho)-\mathfrak{f}_3|_{\theta=\theta_1(\rho)}+(v^0_1)'(\rho)-\mathfrak{f}_4|_{\theta=\theta_1(\rho)}}{2\big(\theta'_1(\rho)-\frac{1}{2}k(u-v)|_{\theta=\theta_1(\rho)}\big)}\\
&\qquad		-\frac{(u^0_1)'(\rho)-\mathfrak{f}_3|_{\theta=\theta_1(\rho)}-(v^0_1)'(\rho)+\mathfrak{f}_4|_{\theta=\theta_1(\rho)}}{2\big(\theta'_1(\rho)-\frac{1}{2}k( u+v)|_{\theta=\theta_1(\rho)}\big)}.
		\end{split}
\end{equation}
  Thus, by \eqref{eq-initial}, we have
  \begin{align*}
  	(u^0_1)'(\rho)+(v^0_1)'(\rho)=\left(\frac{2\tilde{\bar q}}{kG(\theta_1(\rho),\rho)}\right)'=\frac{d}{d\rho}\left(\frac{2\tilde q(\theta_1(\rho), \rho)}{kG(\theta_1(\rho),\rho)}\right),\\
  	(u^0_1)'(\rho)-(v^0_1)'(\rho)=\left(\frac{2\tilde{\bar p}}{kG(\theta_1(\rho),\rho)}\right)'=\frac{d}{d\rho}\left(\frac{2\tilde p(\theta_1(\rho), \rho)}{kG(\theta_1(\rho),\rho)}\right),
  \end{align*}
By a direct calculation based on \eqref{e:tilde-p}-\eqref{e:tilde-q}, we conclude that
\[\hat p:=\frac{2\tilde p}{kG},\quad \hat q:=\frac{2\tilde q}{kG}\]
satisfy
\begin{equation}\label{eq-pqhat}
	\begin{split}
		\drho\hat p+\frac12k\hat q\dthe\hat p=(\mathfrak{f}_3+\mathfrak{f}_4)(\tfrac12(\hat p+\hat q),\, \tfrac12(\hat q-\hat p), \theta, \rho ),\\
		\drho\hat q+\frac12k\hat p\dthe\hat q=(\mathfrak{f}_3-\mathfrak{f}_4)(\tfrac12(\hat p+\hat q),\, \tfrac12(\hat q-\hat p), \theta, \rho ).
	\end{split}
\end{equation}
Thus,
\begin{align*}
	(u^0_1)'(\rho)+(v^0_1)'(\rho)&=\drho\hat q_*(\theta_1(\rho), \rho)+\dthe\hat q_*(\theta_1(\rho), \rho)\theta_1'(\rho)\\
	&=(\mathfrak{f}_3+\mathfrak{f}_4)(\tfrac12(\hat p+\hat q),\, \tfrac12(\hat q-\hat p), \theta, \rho )|_{\theta=\theta_1(\rho)}
	+(\theta'(\rho)-\frac12k\hat p)\dthe\hat q|_{\theta=\theta_1(\rho)}.
\end{align*}
For $\theta=\theta_1(\rho),$ we have
\[u=\frac{\tilde p+\tilde q}{kG}=\frac12(\hat p+\hat q), \quad v=\frac{\tilde q-\tilde p}{kG}=\frac12(\hat q-\hat p). \]
Hence,
\begin{equation}\label{eq-u01+v01}
	(u^0_1)'(\rho)+(v^0_1)'(\rho)=\mathfrak{f}_3|_{\theta=\theta_1(\rho)}+\mathfrak{f}_4|_{\theta=\theta_1(\rho)}+(\theta_1'(\rho)-\frac12k(u-v))\dthe\hat q|_{\theta=\theta_1(\rho)}.
	\end{equation}
Similarly, we have
\begin{equation}\label{eq-u01-v01}
	(u^0_1)'(\rho)-(v^0_1)'(\rho)=\mathfrak{f}_3|_{\theta=\theta_1(\rho)}-\mathfrak{f}_4|_{\theta=\theta_1(\rho)}+(\theta_1'(\rho)-\frac12k(u+v))\dthe\hat p|_{\theta=\theta_1(\rho)}.
	\end{equation}
By substituting \eqref{eq-u01+v01} and \eqref{eq-u01-v01} in \eqref{eq-utheta} and \eqref{eq-vtheta}, we conclude
\begin{align*}
	&\dthe u(\theta_1(\rho),\rho)=\frac12(\dthe\hat p|_{\theta=\theta_1(\rho)}+\dthe\hat q|_{\theta=\theta_1(\rho)})=\dthe\Big(\frac{\tilde p+\tilde q}{kG}\Big)\Big|_{\theta=\theta_1(\rho)},\\
	&	\dthe v(\theta_1(\rho),\rho)=\frac12(-\dthe\hat p|_{\theta=\theta_1(\rho)}+\dthe\hat q|_{\theta=\theta_1(\rho)})=\dthe\Big(\frac{-\tilde p+\tilde q}{kG}\Big)\Big|_{\theta=\theta_1(\rho)}.
\end{align*}
By Lemma \ref{l:boundary-bounds}, we get the second line of
$\eqref{e:initial-omega}$ with $j=1.$ Differentiating $\dthe u(\theta_1(\rho), \rho)$ of \eqref{eq-utheta} and $\dthe v(\theta_1(\rho), \rho)$ of \eqref{eq-vtheta} one more time with respect to $\rho$, respectively, and preceding similarly, we get the second line in $\eqref{e:initial-omega}$ with $j=2$.
\end{proof}

The proof of Proposition \ref{p-priori} is divided into two cases
depending on the monotonicity of  $\overline K$.
Note that $\overline K$ and $\rho^{1+\delta}k_*$, for some constant $\delta\in (0,1/2)$,
have the same type of monotonicity.
Recall that $\delta$ and $k_*$ are introduced in \eqref{delta} and \eqref{k-star}, respectively.

%\subsubsection{Case 1: \texorpdfstring{$\rho^{1+\delta}k_*$ is increasing}{}} %\label{s-increase}

We first consider the case that $\rho^{1+\delta}k_*$ is increasing.
We choose $(\alpha, \beta)=(0, 1)$ and write
the system \eqref{eq-u-general}-\eqref{eq-v-general} as
\begin{align}
\drho u+\frac{1}{2}k(u\dthe u-v\dthe v)&=-u\partial_\rho\log (kG^2)-\frac{1}{2}ku^2\dthe \log (kG)\notag\\
&\qquad+\frac{1}{2}kv^2\dthe \log (k^2G)-\frac{1}{4}G\drho Gk^2u(u^2-v^2),\label{e:u}\\
\drho v+\frac{1}{2}k(u\dthe v-v\dthe u)&=\frac{1}{2}kuv\dthe \log k-\frac{1}{4}G\drho Gk^2v(u^2-v^2).\label{e:v}
\end{align}
Since $\rho^{1+\delta}k_*$ is increasing, we have
\begin{equation}\label{e:damping}
	\begin{split}-\drho\log(kG^2)&=-\drho\log (\rho^2k_*)-\drho\log a-2\drho\log (\rho^{-1}G)\\
&=-\drho\log(\rho^{1+\delta}k_*)-\drho\log \rho^{1-\delta}-\drho\log a-2\drho\log (\rho^{-1}G)\\
&\leq-\frac{1-\delta}{\rho}-\drho\log a-2\drho\log (\rho^{-1}G).
\end{split}
\end{equation}
This indicates that the equation \eqref{e:u} has a damping effect and then $u$ should have better decay properties than $v$ for large $\rho$. We will
establish the following decay of $\|u(\rho)\|_1$ and uniform estimates for $\|v(\rho)\|_1$.

\begin{lemma}\label{l:u-ux-estimate}
 Assume that $\rho^{1+\delta}k_*$ is increasing for some $\delta\in (0,1/2)$. For any $R'>R$, assume that $(u, v)\in C^1([R, R'], H^2(I(\rho)))$ is the  solution  to \eqref{e:u}-\eqref{e:v},
 	with the initial boundary data  \eqref{eq-initial}, and satisfies \eqref{e:priori-h2-assumption-uv}.  Then, for any $\rho\in [R,R']$,
\begin{align}
\|u(\rho)\|\leq\Theta\eps(\rho k_*+\rho^{-\delta/2})+ \Theta\eps R^{\delta}\rho^{-\delta}%\frac{R^{1+2\delta}k_*(R)}{\rho^{1+2\delta}k_*(\rho)}
\leq 3\Theta\eps, \label{e:u-esitmate}
\end{align}
\begin{align}\|\dthe u(\rho)\|\leq 3\Theta^3\eps(\rho k_*+\rho^{-\delta/2})+2\Theta^3\eps R^{\delta}\rho^{-\delta}%\frac{ R^{1+2\delta}k_*(R)}{\rho^{1+2\delta}k_*(\rho)}
\leq 8\Theta^3\eps, \label{e:ux-esitmate}
\end{align}
and
\begin{align}
	\|v(\rho)\|\leq 2\Theta\eps, \quad \|\dthe v(\rho)\|\leq 2\Theta\eps, \label{e:v-vx-esitmate}\end{align}
where $\Theta$ is give by \eqref{A}.
\end{lemma}

%{(Comment: in the proof  we use finite total curvature, increasing property of $\overline K$ and \eqref{eq-a-bul}. )}

\begin{proof}
	We rewrite \eqref{e:u} as
\begin{align*}
\drho u&=-u\partial_\rho\log (kG^2)-\frac{1}{2}ku^2\dthe \log (kG)
+\frac{1}{2}kv^2\dthe \log (k^2G)\\
&\qquad -\frac{1}{2}k(u\dthe u-v\dthe v)-\frac{1}{4}G\drho Gk^2u(u^2-v^2).\end{align*}
By multiplying the above equation by $u$ and integrating the resulting equation over $I(\rho)$, we have
\begin{align*}
\frac12\frac{d}{d\rho}\int_{I(\rho)}u^2d\theta=-\int_{I(\rho)}u^2\drho\log(kG^2)d\theta+\Psi_1+J_1,
\end{align*}
where
\begin{align*}
\Psi_1&=\frac{1}{2}(u_2^0)^2\theta'_2(\rho)-\frac{1}{2}(u_1^0)^2\theta'_1(\rho),
\end{align*}
with the boundary data $u^0_i(\rho)$ given by \eqref{IC} for $i=1,2$, and
\begin{align*}
J_1&=-\frac{1}{2}\int_{I(\rho)}ku\big[u^2\dthe \log (kG)
-v^2\dthe \log (k^2G)\big]d\theta\\
&\qquad -\frac{1}{2}\int_{I(\rho)}ku(u\dthe u-v\dthe v)d\theta-\frac{1}{4}\int_{I(\rho)}G\drho Gk^2u^2(u^2-v^2)d\theta.
\end{align*}
By \eqref{e:initial-omega}, it is straightforward to check
\begin{align*}
|\Psi_1|\leq \eps^2\rho^{-1-\delta}.
\end{align*}
Recall that $\varphi$ is defined in \eqref{varphi}. Then,
$$\frac{\va'}{\va}=\max_{\theta}\bigg|\frac{\partial_\rho a}{a}\bigg|+2\max_{\theta}\bigg|\frac{\partial_\rho G}{G}-\frac1\rho\bigg|\geq0.$$
By \eqref{eq-a-bul}, \eqref{e:priori-h2-assumption-uv}, %the conditions on $a$,
and Lemma \ref{l:G-property}, we get
\begin{align*}
-\drho\log(kG^2)\le-\drho\log(\rho^2k_*)+\max_{\theta}|\drho\log(\rho^{-2}G^2a)|\le -\drho\log(\rho^2k_*)+\drho\log\va,
\end{align*}
and
\begin{align*}
|J_1|&\le Ck_*\int_{I(\rho)}|u|(u^2+v^2+|u\dthe u|+|v\dthe v|)d\theta
+C\rho k_*^2\int_{I(\rho)}u^2(u^2+v^2)d\theta\\
&\leq C\eps^2k_*\|u\|.
\end{align*}
Then, for any $\rho\in [R,R']$,
\begin{equation*}
\frac12\frac{d}{d\rho}\int_{I(\rho)}u^2d\theta\leq -\drho\log\bigg(\frac{\rho^2k_*}{\va}\bigg)\int_{I(\rho)}u^2d\theta+C\eps^2k_*\|u\|+\eps^2\rho^{-1-\delta},
\end{equation*}
and hence
\begin{align*}
\frac{d}{d\rho}\left(\|u\|^2\frac{\rho^4k_*^2}{\va^2}\right)\leq C\eps^2\frac{\rho^2k_*^2}{\va}\left(\|u\|\frac{\rho^2k_*}{\va}\right)+2\eps^2\frac{k_*^2\rho^{3-\delta}}{\va^2}.
\end{align*}
By Lemma \ref{l:gronwall}, we get
$$\|u(\rho)\|\leq \frac{\va R^2k_*(R)}{\rho^2k_*}\|u(R)\|+
\sqrt{\frac{\va^2}{\rho^4k_*^2}\int_R^\rho2\eps^2\frac{s^{3-\delta}k_*^2(s)}{\va(s)^2}ds}+
\frac{\va}{2\rho^2k_*}\int_R^\rho C\eps^2\frac{s^2k_*^2(s)}{\va(s)}ds.$$
Since $k_*(\rho)\rho^{1+\delta}$ is increasing  and $\varphi(R)=1, \varphi(\rho)\ge 1, \forall \rho\ge R$,
we obtain
\begin{equation}\label{e:int-ineq}
	\begin{split}
		\frac{1}{\rho^4k_*^2(\rho)}\int_R^\rho\frac{s^{3-\delta}k_*^2(s)}{\va(s)^2}ds
		\leq\frac{\rho^{2+2\delta}k_*^2(\rho)}{\rho^4k_*^2(\rho)}\int_{R}^\rho s^{1-3\delta}ds
		\leq\frac{\rho^{-\delta}}{2-3\delta}, \\
\frac{1}{\rho^2k_*(\rho)}\int_R^\rho\frac{s^2k_*^2(s)}{\va(s)}ds
\leq\frac{\rho^{2+2\delta}k_*^2(\rho)}{\rho^2k_*(\rho)}\int_{R}^\rho s^{-2\delta}ds
\leq\frac{\rho k_*(\rho)}{1-2\delta}.
\end{split}
\end{equation}
Hence,
 \begin{align*}
 \|u\|&\leq\frac{\eps\Lambda R^{2}k_*(R)}{\rho^2k_*(\rho)}+\frac{\Lambda \eps\rho^{-\delta/2}}{\sqrt{2-3\delta}}+\frac{C\Lambda\eps^2\rho k_*}{1-2\delta}\\
& \leq\eps\Theta R^{\delta}\rho^{-\delta}+\Theta\eps(\rho k_*+\rho^{-\delta/2})\le 3\Theta\eps,
 \end{align*}
by choosing $\eps$  sufficiently small and using $2-\delta>1+\delta$. Thus, we conclude \eqref{e:u-esitmate}.

Next, differentiating $\eqref{e:u}$ with respect to $\theta$ yields
\begin{equation}\label{e:ux}
\begin{split}
\drho(\dthe u)
&=-\dthe u\partial_\rho\log (kG^2)-u\dthe\drho\log(kG^2)\\
&\qquad-\frac{1}{2}\dthe(ku^2\dthe \log (kG))
+\frac{1}{2}\dthe(kv^2\dthe \log (k^2G))\\
&\qquad-\frac{1}{2} \dthe(k(u\dthe u-v\dthe v))-\frac{1}{4}\dthe(G\drho Gk^2u(u^2-v^2)).
\end{split}
\end{equation}
By multiplying \eqref{e:ux} by $\dthe u$ and integrating the resulting equation over $I(\rho)$, we have
\begin{align*}
\frac12\frac{d}{d\rho}\int_{I(\rho)}|\dthe u|^2d\theta
=-\int_{I(\rho)}|\dthe u|^2\drho\log(kG^2)d\theta+\Psi_2+J_2+J_3,
\end{align*}
where
\begin{align*}
\Psi_2=&\frac{1}{2}(u_2^1)^2\theta'_2(\rho)-\frac{1}{2}(u_1^1)^2\theta'_1(\rho),\\
J_2=&\int_{I(\rho)}u\dthe u\dthe\drho\log(kG^2)d\theta,
\end{align*}
and
\begin{align*}J_3&=
-\frac{1}{2}\int_{I(\rho)}\dthe u\big[\dthe(ku^2\dthe \log (kG))
-\dthe(kv^2\dthe \log (kG))\big]d\theta\\
&\qquad-\frac{1}{2}\int_{I(\rho)}\dthe u\dthe(k(u\dthe u-v\dthe v))d\theta-\frac{1}{4}\int_{I(\rho)}\dthe u\dthe(G\drho Gk^2u(u^2-v^2))d\theta.
\end{align*}
By \eqref{e:initial-omega}, it is straightforward to check
$$|\Psi_2|\leq \eps^2\rho^{-1-\delta}, ~~|J_3|\le C\eps^2k_*\|\dthe u\|.$$
Moreover, it follows from \eqref{e:u-esitmate} that
\begin{align*}
|J_2|&\le \frac{\Lambda}{\rho}\int_{I(\rho)}|u\dthe u|d\theta
\le \frac{\Lambda}{\rho}\|u\|\|\dthe u\| \\
&\le \eps\Lambda\Theta R^{\delta}{\rho^{-1-\delta}} \|\dthe u\|+\Lambda\Theta\eps(k_*+\rho^{-1-\delta/2})\|\dthe u\|.
\end{align*}
Then,
\begin{align*}
\frac12\frac{d}{d\rho}\int_{I(\rho)}|\dthe u|^2d\theta&\leq-\drho\log\bigg(\frac{\rho^2k_*}{\va}\bigg)\int_{I(\rho)}|\dthe u|^2d\theta+\eps^2\rho^{-1-\delta}+C\eps^2k_*\|\dthe u\|\\
&\qquad+\eps\Lambda\Theta R^{\delta}{\rho^{-1-\delta}}\|\dthe u\|+\Lambda\Theta\eps(k_*+\rho^{-1-\delta/2})\|\dthe u\|,
\end{align*}
and hence
\begin{align*}
\frac{d}{d\rho}\left(\|\dthe u\|^2\frac{\rho^4k_*^2}{\va^2}\right)\leq & \left((C\eps^2+2\Lambda\Theta\eps)\frac{\rho^2k_*^2}{\va}+2\Lambda\Theta\eps
\frac{k_*\rho^{1-\delta/2}}{\va}+\frac{2\Lambda\Theta\eps R^{\delta}\rho^{1-\delta} k_*}{\va}\right)\left(\|\dthe u\|\frac{\rho^2k_*}{\va}\right)\\
&\quad +2\eps^2\frac{k_*^2\rho^{3-\delta}}{\va^2}.
\end{align*}
By Lemma \ref{l:gronwall}, we get
\begin{align*}
\|\dthe u(\rho)\|\frac{\rho^2k_*}{\va}&\leq \|\dthe u(R)\|R^2k_*(R)+
\sqrt{\int_R^\rho2\eps^2\frac{s^{3-\delta}k_*^2(s)}{\va(s)^2}ds}\\
&\qquad +
(C\eps^2+\Lambda\Theta\eps)\int_R^\rho \frac{s^2k_*^2(s)}{\va(s)}ds
+\Lambda\Theta\eps\int_R^\rho \frac{k_*(s)s^{1-\delta/2}}{\va(s)}ds
\\
&\qquad +\eps\Lambda\Theta R^{\delta}\int_R^\rho\frac{k_*(s)s^{1-\delta}}{\va(s)}ds.
\end{align*}
Similarly as for \eqref{e:int-ineq}, we have
\begin{align*}
\|\dthe u(\rho)\|
&\leq \frac{\eps\Lambda R^{2}k_*(R)}{\rho^{2}k_*}+\frac{\Lambda\eps\rho^{-\delta/2}}{\sqrt{2-3\delta}}
+\frac{C\Lambda\eps^2+2\Lambda\Theta\eps}{(1-2\delta)}\rho k_*\\
&\qquad +\frac{2\Lambda^2\Theta\eps\rho^{-\delta/2}}{2-3\delta}+
\frac{\eps\Lambda^2\Theta R^{\delta}\rho^{-\delta}}{1-2\delta}.
\end{align*}
 Thus, we obtain \eqref{e:ux-esitmate} by choosing $\eps$  sufficiently small and using $2-\delta>1+\delta$.
 
Last, we apply the estimates of $\|u(\rho)\|_1$ in \eqref{e:u-esitmate} and \eqref{e:ux-esitmate} to control $\|v(\rho)\|_1.$
By multiplying \eqref{e:v} by $v$, we have
$$\frac12\frac{d}{d\rho}\int_{I(\rho)}v^2d\theta=\Psi_3+J_4,$$
where
$$\Psi_3=\frac{1}{2}(v_2^0)^2\theta'_2(\rho)-\frac{1}{2}(v_1^0)^2\theta'_1(\rho),$$
and
$$J_4=\int_{I(\rho)}\big[-\frac{1}{2}k(uv\dthe v-v^2\dthe u)+\frac{1}{2}kuv^2\dthe \log k-\frac{1}{4}G\drho Gk^2v^2(u^2-v^2)\big]d\theta.$$
By the estimates of $\|u\|_1$ in \eqref{e:u-esitmate}-\eqref{e:ux-esitmate}, we have
	\begin{align*}
		|J_4|&\leq
		\int_{I(\rho)}[Ck_*|v|(|u\dthe v|+|v\dthe u|+|uv|)+C\rho k_*^2v^2(u^2+v^2)]d\theta\\
		&\leq C\eps^2k_*(\|u\|+\|\dthe u\|)+C\eps^3\rho k_*^2\\
		&\leq C\eps^3R^{\delta}\rho^{-1-\delta}+C\eps^3\rho k_*^2+C\eps^3\rho^{-1-\delta/2}.
	\end{align*}
Similarly, we obtain
	\begin{align*}
	\int_{I(\rho)}v^2d\theta\leq 4\Theta^2\eps^2.
\end{align*}
This implies the first inequality in \eqref{e:v-vx-esitmate}.

Similarly, by differentiating \eqref{e:v} with respect to $\theta$, multiplying the resulting equation by $\partial_\theta v$,
and integrating over $I(\rho)$, %In the same way as in \eqref{e:ux},
we have	
	\begin{align*}
		\frac12\frac{d}{d\rho}\int_{I(\rho)}|\dthe v|^2d\theta=\Psi_4+J_5,\end{align*}
	where
	$$\Psi_4=\frac{1}{2}(v_2^1)^2\theta'_2(\rho)-\frac{1}{2}(v_1^1)^2\theta'_1(\rho)-\frac{1}{2}\dthe v(k(u\dthe v-v\dthe u))\bigg|_{\theta_1}^{\theta_2},
		$$
and
	\begin{align*}
		J_{5}&=\frac{1}{2}\int_{I(\rho)}\dthe^2 v(k(u\dthe v-v\dthe u))d\theta
		+\frac{1}{2}\int_{I(\rho)}\dthe v\dthe(kuv\dthe \log k)d\theta\\
		&\qquad-\frac{1}{4}\int_{I(\rho)}\dthe v\dthe(G\drho Gk^2v(u^2-v^2))d\theta.
	\end{align*}
By \eqref{e:u-esitmate} and \eqref{e:ux-esitmate} again, we get
	\begin{align*}
		|J_{5}| \leq C\eps^2k_*(\|u\|+\|\dthe u\|)+C\eps^3\rho k_*^2
		\leq C\eps^3R^\delta\rho^{-1-\delta}+C\eps^3\rho k_*^2+C\eps^3\rho^{-1-\delta/2}.
	\end{align*}
Thus, we obtain
the second inequality in \eqref{e:v-vx-esitmate}.
\end{proof}

\begin{lemma}\label{l:uv-h2-estimate}
Assume that $\rho^{1+\delta}k_*$ is increasing for some $\delta\in (0,1/2)$. For any $R'>R$, assume that $(u, v)\in C^1([R, R'], H^2(I(\rho)))$ is the  solution  to \eqref{e:u}-\eqref{e:v},
	with the initial boundary data  \eqref{eq-initial}, and satisfies \eqref{e:priori-h2-assumption-uv}.  Then, for any $\rho\in [R,R']$,
\begin{align}
%\|\dthe^2u(\rho)\|,\,\|\dthe^2v(\rho)\|
\|\dthe^2 u(\rho)\|\leq 100\Theta^6\eps,\quad\| \dthe^2v(\rho)\|\leq 100\Theta^6\eps.\label{e:uv-h2-estimate}
\end{align}
\end{lemma}

\begin{proof}
By differentiating \eqref{e:u} and \eqref{e:v} with respect to $\theta$ twice, we get
\begin{equation}\label{e:u-xx}
\begin{split}
\drho(\dthe^2 u)+\frac{1}{2}k\dthe^2(u\dthe u-v\dthe v)
&=-\dthe^2 u\partial_\rho\log (kG^2)
-u\dthe^2\drho\log(kG^2)\\
&\qquad-\frac{1}{2}\dthe^2k(u\dthe u-v\dthe v)-\dthe k\dthe(u\dthe u-v\dthe v)\\
&\qquad-\frac{1}{2}\dthe^2(ku^2\dthe \log (kG))+\frac{1}{2}\dthe^2(kv^2\dthe \log (k^2G))\\
&\qquad-\frac{1}{4}\dthe^2(G\drho Gk^2u(u^2-v^2))-2\dthe u\dthe\drho\log (kG^2),
\end{split}
\end{equation}
and
\begin{equation}\label{e:v-xx}
\begin{split}
\drho(\dthe^2 v)+\frac{1}{2}k\dthe^2(u\dthe v-v\dthe u)&=-\frac{1}{2}\dthe^2k(u\dthe v-v\dthe u)-\dthe k\dthe(u\dthe v-v\dthe u)\\
&\qquad+\frac{1}{2}\dthe^2(kuv\dthe \log k)+\frac{1}{4}\dthe^2(G\drho Gk^2v(u^2-v^2)).
\end{split}
\end{equation}
Note that the equations $\eqref{e:u-xx}$ and $\eqref{e:v-xx}$ are symmetrically hyperbolic for $\dthe^2 u$ and $\dthe^2 v$. By multiplying these equations by $\dthe^2 u$ and $\dthe^2 v$, respectively, we get
\begin{align}\label{e:2derivatives}
\frac12\frac{d}{d\rho}\int_{I(\rho)}(|\dthe^2u|^2+|\dthe^2v|^2)d\theta= -\int_{I(\rho)}|\dthe^2u|^2\drho\log(kG^2)d\theta+\Psi_5+\sum_{i=6}^8J_i,
\end{align}
where
\begin{align*}
\Psi_5&=\frac{1}{2}[(u_2^2)^2+(v_2^2)^2]\theta'_2(\rho)-\frac{1}{2}[(u_1^2)^2+(v_1^2)^2]\theta'_1(\rho)\\&\quad-\frac14\big[ku((\dthe^2u)^2+(\dthe^2v)^2)-2kv(\dthe^2u\dthe^2v)\big]\bigg|_{\theta_1}^{\theta_2},\\
J_6&=-\int_{I(\rho)}\big[u\dthe^2u\dthe^2\drho\log(kG^2)+2\dthe u\dthe^2u\dthe\drho\log (kG^2)\big]d\theta,\\
J_7&=\frac{1}{4}\int_{I(\rho)}\dthe k\big[u(\dthe^2u)^2-2v\dthe^2u\dthe^2v+u(\dthe^2v)^2\big]d\theta\\
&\quad-\frac{1}{4}\int_{I(\rho)}k\big[5\dthe u(\dthe^2u)^2-6\dthe v\dthe^2u\dthe^2v+\dthe u(\dthe^2v)^2\big]d\theta,
\end{align*}
and
\begin{align*}
J_8&= -\frac{1}{2}\int_{I(\rho)}\dthe^2k\big[\dthe^2u(u\dthe u-v\dthe v)+\dthe^2v(u\dthe v-v\dthe u)\big]d\theta\\
&\qquad -\int_{I(\rho)}\dthe k\big[\dthe^2u\dthe(u\dthe u-v\dthe v)+\dthe^2v\dthe(u\dthe v-v\dthe u)\big]d\theta\\
&\qquad-\frac{1}{2}\int_{I(\rho)}\dthe^2u\big[\dthe^2(ku^2\dthe \log (kG))-\dthe^2(kv^2\dthe \log (k^2G))\big]d\theta\\
&\qquad+\frac{1}{2}\int_{I(\rho)}\dthe^2v\dthe^2(kuv\dthe \log k)d\theta\\
&\qquad-\frac{1}{4}\int_{I(\rho)}\big[\dthe^2u\dthe^2(G\drho Gk^2u(u^2-v^2))-\dthe^2v\dthe^2(G\drho Gk^2v(u^2-v^2))\big]d\theta.
\end{align*}
By the estimate of $\|u(\rho)\|_1$ in Lemma \ref{l:u-ux-estimate}, we get
\begin{equation}\label{e:j1-estimate}
\begin{split}
|J_6|&\leq\frac{\Lambda}{\rho}\int_{I(\rho)}\big(|u\dthe^2u|+2|\dthe u\dthe^2u|\big)d\theta\leq\frac{\Lambda}{\rho}(\|u\|+2\|\dthe u\|)\|\dthe^2u\|\\
&\leq\frac{\Lambda}{\rho}\|\dthe^2u\|
\left(5\eps\Theta^3R^{\delta}\rho^{-\delta}+8\Theta^3\eps(\rho^{-\delta/2}+\rho k_*)\right)\\
&\leq\left( 5\eps\Lambda\Theta^3R^{\delta}\rho^{-1-\delta}+8\Lambda\Theta^3\eps(\rho^{-1-\delta/2}+k_*)\right)\|\dthe^2u\|.
\end{split}
\end{equation}
As for $J_{7},$ we note that the integrand in $J_{7}$ can be regarded as a homogeneous cubic polynomial in $u$, $v$, and their derivatives. In fact, each term is quadratic in $\dthe^2u$ and $\dthe^2v$, and is linear in $u$, $v$, $\dthe u$ and $\dthe v$. We treat terms involving $(\dthe^2v)^2$ differently from those involving $(\dthe^2u)^2$ and $\dthe^2u\dthe^2v$. A crucial observation here is that the terms $v\dthe^2v$ and $\dthe v\dthe^2v$ are absent from $J_{7}$.
By \eqref{eq-a-bul}  %the conditions on $a$
and \eqref{e:priori-h2-assumption-uv}, we obtain
\begin{equation*}
\begin{split}
|J_{7}|&\leq C\eps k_*\int_{I(\rho)}(|\dthe^2u|^2+|\dthe^2u\dthe^2v|)d\theta+ Ck_*(|\dthe u|_0{+|u|_0})\int_{I(\rho)}|\dthe^2v|^2d\theta\\
&\leq C\eps k_*\int_{I(\rho)}|\dthe^2u|^2d\theta+C\eps^2k_*\|\dthe^2 u\|+C\eps^2k_*(|\dthe u|_0{+|u|_0}).
\end{split}
\end{equation*}
By the Sobolev embedding and the Cauchy inequality, we have
\begin{align*}
|u|_0&\leq C\|u\|_1=C(\|\dthe u\|+\|u\|),\\
 |\dthe u|_0&\leq C\sqrt{\|\dthe u\|\|\dthe^2u\|}\leq C(\|\dthe u\|+\|\dthe^2u\|),
\end{align*}
since $u(\cdot, \rho)$ is defined on $I(\rho)\subset[0, \pi]$ for any fixed $\rho$.
Hence by Lemma \ref{l:u-ux-estimate} and $k_*=o(\frac1\rho)$, we get
\begin{equation*}%\label{e:j22-estimate}
\begin{split}
|J_{7}|\leq &C\eps k_*\int_{I(\rho)}|\dthe^2u|^2d\theta+C\eps^2k_*\|\dthe^2 u\|+C\eps^3\rho k_*^2\\
&\quad +C\eps^3\rho^{-1-\delta/2}+C\eps^3R^{\delta}\rho^{-1-\delta}.
\end{split}
\end{equation*}
 The first four integrals for $J_8$ can be treated similarly to   $J_7$. For the last integral, the terms quadratic in $\dthe^2u$ and $\dthe^2v$ have a coefficient $G\drho Gk^2$, which is bounded by $C\rho k_*^2$.
Similarly, we have
\begin{equation}\label{e:j3-estimate}
	|J_8|\le C\eps^2k_*\|\dthe^2u\|+C\eps^3\rho k_*^2+C\eps^3\rho^{-1-\delta/2}.
\end{equation}
By combining \eqref{e:j1-estimate} and \eqref{e:j3-estimate}, we obtain
\begin{align*}
|J_6+J_7+J_8|&\leq  C\eps k_*\int_{I(\rho)}|\dthe^2u|^2d\theta+\left(8\Lambda\Theta^3\eps(\rho^{-1-\delta/2}+k_*)+5\Lambda\Theta^3\eps R^{\delta}\rho^{-1-\delta}\right)\|\dthe^2u\|\\
&\qquad+C\eps^2k_*\|\dthe^2 u\|+C\eps^3\left(\rho k_*^2+\rho^{-1-\delta/2}+R^{\delta}\rho^{-1-\delta}\right).
\end{align*}
Take a small $\sigma>0$ to be fixed.
Using $k_*=o(\frac1\rho)$  and applying  the  Cauchy inequality to the middle two terms on the right-hand side, we obtain
\begin{align*}
|J_6+J_7+J_8|&\leq\frac{5\sigma}{\rho}\|\dthe^2u\|^2+\frac{16}{\sigma}\Lambda^2\Theta^6\eps^2(\rho k_*^2+\rho^{-1-\delta})
+\frac{25}{4\sigma}\Lambda^2\Theta^6\eps^2R^{2\delta}\rho^{-1-2\delta}\\
&\qquad+C\eps^3\left(\rho k_*^2+\rho^{-1-\delta/2}+R^{\delta}\rho^{-1-\delta}\right).
\end{align*}
Note that $\Psi_5$ contains only the boundary terms. A simple substitution in \eqref{e:2derivatives} yields
\begin{equation*}
\begin{split}
\frac12\frac{d}{d\rho}\int_{I(\rho)}(|\dthe^2u|^2+|\dthe^2v|^2)d\theta
&\leq\bigg(-\frac{k_*'}{k_*}-\frac{2}{\rho}+\frac{5\sigma}{\rho}+\frac{\va'}{\va}\bigg)\int_{I(\rho)}|\dthe^2 u|^2d\theta+2\eps^2\rho^{-1-\delta}\\
&\qquad+\frac{16}{\sigma}\Lambda^2\Theta^6\eps^2(\rho k_*^2+\rho^{-1-\delta})
+\frac{25}{4\sigma}\Lambda^2\Theta^6\eps^2R^{2\delta}\rho^{-1-2\delta}\\
&\qquad+C\eps^3\left(\rho k_*^2+\rho^{-1-\delta/2}+R^{\delta}\rho^{-1-\delta}\right).
\end{split}
\end{equation*}
In view of \eqref{e:damping}, we take $$\sigma=\frac{1-\delta}{5}.$$
Then,
$$-\frac{k_*'}{k_*}-\frac{2}{\rho}+\frac{5\sigma}{\rho}\leq-\frac{1-\delta}{\rho}+\frac{5\sigma}{\rho}=0,$$
and hence
\begin{equation*}%\label{e:uv-xx-l2}
\begin{split}
\frac{d}{d\rho}\int_{I(\rho)}(|\dthe^2u|^2+|\dthe^2v|^2)d\theta
&\leq\frac{2\va'}{\va}\int_{I(\rho)}(|\dthe^2u|^2+|\dthe^2v|^2)d\theta+4\eps^2\rho^{-1-\delta}\\
&\quad+\frac{32}{\sigma}\Lambda^2\Theta^6\eps^2(\rho k_*^2+\rho^{-1-\delta})
+\frac{25}{2\sigma}\Lambda^2\Theta^6\eps^2R^{2\delta}\rho^{-1-2\delta}\\
&\quad+C\eps^3\left(\rho k_*^2+\rho^{-1-\delta/2}+R^{\delta}\rho^{-1-\delta}\right).
\end{split}
\end{equation*}
A simple integration yields
%By Gronwall's inequality, we obtain
\begin{align*}
\int_{I(\rho)}(|\dthe^2u|^2+|\dthe^2v|^2)d\theta&\leq\frac5\delta\eps^2\va^2+\frac{32}{\sigma}\Lambda^{2}\Theta^6\va^2\eps^2(\Lambda+\frac1\delta)+\frac{25}{4\sigma\delta}\Lambda^2\Theta^6\va^2\eps^2+C\eps^3\va^2\Theta\\
&\leq10^4\Theta^{12}\eps^2\,,
\end{align*}
if $\eps$ is small sufficiently.
This implies \eqref{e:uv-h2-estimate}.
\end{proof}

By combining Lemmas  \ref{l:u-ux-estimate}-\ref{l:uv-h2-estimate}, we obtain an $H^2$-estimate of $(u, v)$ with respect to $\theta$ for the case that $\rho^{1+\delta}k_*$ is increasing.

\smallskip

Next, we study the case that $\rho^{1+\delta}k_*$ is decreasing.
We take $(\alpha, \beta)=(1, 0)$ and write \eqref{eq-u-general}-\eqref{eq-v-general} as
\begin{align}
\drho u+\frac{1}{2G}(u\dthe u-v\dthe v)&=-u\drho\log G+\frac{1}{2G}v^2\dthe\log k-\frac{1}{4}u(u^2-v^2)\drho\log G, \label{e:u1}\\
\drho v+\frac{1}{2G}(u\dthe v-v\dthe u)&=v\drho\log(kG)+\frac{1}{2G}uv\dthe\log k-\frac{1}{4}v(u^2-v^2)\drho\log G.\label{e:v1}
\end{align}
Since $\rho^{1+\delta}k_*$ is decreasing, we have
\begin{align}\label{e:liner-estimate}
\begin{split}
	-\drho\log G&\leq-\drho\log \rho+\max_{\theta} |\drho\log(\rho^{-1}G)|\leq-\frac1\rho+\drho\log\va,\\
	\drho\log(kG)&=\drho\log(\rho k_*)+\max_{\theta}|\drho\log a+\drho\log(\rho^{-1}G)|\leq-\frac{\delta}{\rho}+\drho\log\va.
\end{split}
\end{align}

This indicates that both linear terms on the right-hand sides of \eqref{e:u1} and \eqref{e:v1} have damping effects and then both $u$ and $v$ have decay properties.

\begin{lemma}\label{l:rs-h0-estimate}
Assume that $\rho^{1+\delta}k_*$ is decreasing for some $\delta\in (0,1/2)$. For any $R'>R$, assume that $(u, v)\in C^1([R, R'], H^2(I(\rho)))$ is the  solution  to \eqref{e:u1}-\eqref{e:v1},
	with the initial boundary data  \eqref{eq-initial}, and satisfies \eqref{e:priori-h2-assumption-uv}. Then, for any $\rho\in [R, R']$,
\begin{align}
\|(u(\rho)\|\leq2\Theta\eps\frac{R^{\delta/2}}{\rho^{\delta/2}}\leq 2\Theta\eps,\quad\| v(\rho)\|\leq 2\Theta\eps \frac{R^{\delta/2}}{\rho^{\delta/2}}\leq 2\Theta\eps.\label{e:rs-h0-estimate}
\end{align}
\end{lemma}

\begin{proof}
By multiplying \eqref{e:u1} and \eqref{e:v1} by $u$ and $ v$, respectively, we have
\begin{align*}
\frac{1}{2}\frac{d}{d\rho}\int_{I(\rho)}(u^2+v^2)d\theta=-\int_{I(\rho)}u^2\drho\log Gd\theta+\int_{I(\rho)}v^2\drho\log(kG)d\theta+\Psi_6+J_9,
\end{align*}
where
\begin{align*}
\Psi_6&=\frac{1}{2}[(u_2^0)^2+(v_2^0)^2]\theta'_2(\rho)-\frac{1}{2}[(u_1^0)^2+(v_1^0)^2]\theta'_1(\rho),
\end{align*}
with the boundary data $u^0_i(\rho)$ given by \eqref{IC} for $i=1,2$, and
\begin{align*}
J_9&=-\frac{1}{2}\int_{I(\rho)}\frac{1}{G}\big[u(u\dthe u-v\dthe v)+v(u\dthe v-v\dthe u)\big]d\theta\\
&\qquad +\int_{I(\rho)}\frac1G\big[uv^2\dthe\log k-\frac14{\drho G}(u^2+v^2)(u^2-v^2)\big]d\theta.
\end{align*}
By \eqref{e:initial-omega}, we have
$$|\Psi_6|\leq 2\eps^2\rho^{-1-2\delta}.$$
By \eqref{e:priori-h2-assumption-uv}, \eqref{eq-a-bul}, and  Lemma \ref{l:G-property}, we also have
\begin{align*}
|J_9|\leq\frac{C}{\rho}\int_{I(\rho)}(u^2+v^2)(|\dthe u|+|\dthe v|+|u|+|v|+u^2+v^2)d\theta
\leq\frac{C\eps}{\rho}\int_{I(\rho)}(u^2+v^2)d\theta.
\end{align*}
Then, for any $\rho\in[R, R'],$
\begin{align*}
\frac{1}{2}\frac{d}{d\rho}\int_{I(\rho)}(u^2+v^2)d\theta
\leq\left(-\frac{\delta}{\rho}+\frac{\va'}{\va}+\frac{C\eps}{\rho}\right)\int_{I(\rho)}(u^2+v^2)d\theta+2\eps^2\rho^{-1-2\delta}.
\end{align*}
By taking $\eps$ small, we have
\begin{align*}
\frac{d}{d\rho}\int_{I(\rho)}(u^2+v^2)d\theta\leq -\drho\log\left(\frac{\rho^{\delta}}{\va^2}\right)\int_{I(\rho)}(u^2+v^2)d\theta+4\eps^2\rho^{-1-2\delta}.
\end{align*}
A simple integration yields
$$\int_{I(\rho)}(u^2+v^2)d\theta\leq \eps^2\va^2(\rho)\frac{R^{\delta}}{\rho^{\delta}}+\frac{4\va^2\eps^2}{\delta\rho^{\delta}R^{\delta}}\leq 4\Theta^2\eps^2\frac{R^{\delta}}{\rho^{\delta}},$$
since $\delta<1/2.$ This implies \eqref{e:rs-h0-estimate}.
\end{proof}

\begin{lemma}\label{l:rs-h1-estimate}
Assume that $\rho^{1+\delta}k_*$ is decreasing for some $\delta\in (0,1/2)$. For any $R'>R$, assume that $(u, v)\in C^1([R, R'], H^2(I(\rho)))$ is the  solution  to \eqref{e:u1}-\eqref{e:v1},
	with the initial boundary data  \eqref{eq-initial}, and satisfies \eqref{e:priori-h2-assumption-uv}. Then, for any $\rho\in [R, R']$,
\begin{align}
\|\dthe u(\rho)\|\leq80\Theta^3\eps\frac{R^{\delta/2}}{\rho^{\delta/2}} \leq 80\Theta^3\eps,\quad\| \dthe v(\rho)\|\leq 80\Theta^3\eps\frac{R^{\delta/2}}{\rho^{\delta/2}} \leq 80\Theta^3\eps\label{e:rs-h1-estimate}
\end{align}
and
\begin{align}
	\|\dthe^2u(\rho)\|\leq10^4\Theta^5\eps\frac{R^{\delta/2}}{\rho^{\delta/2}}\leq 10^4\Theta^5\eps,\quad \|\dthe^2v(\rho)\|\leq 10^4\Theta^5\eps\frac{R^{\delta/2}}{\rho^{\delta/2}}\leq10^4\Theta^5\eps.\label{e:rs-h2-estimate}
\end{align}
\end{lemma}

\begin{proof}
Differentiating \eqref{e:u1} and \eqref{e:v1} with respect to $\theta$, we get
\begin{equation}\label{e:rx}
\begin{split}
\drho(\dthe u)+\frac{1}{2G}\dthe(u\dthe u-v\dthe v)&=-\dthe u\drho\log G-u\dthe\drho\log G \\
&\qquad+\frac{\dthe G}{2G^2}(u\dthe u-v\dthe v)+\frac12\dthe\left(\frac{1}{G}v^2\dthe\log k\right)\\
&\qquad-\frac14\dthe\left(u(u^2-v^2)\drho \log G\right),
\end{split}
\end{equation}
and
\begin{equation}\label{e:sx}
\begin{split}
\drho(\dthe v)+\frac{1}{2G}\dthe(u\dthe v-v\dthe u)&=\dthe v\drho\log(kG)+v\dthe\drho\log(kG) \\
&\qquad+\frac{\dthe G}{2G^2}(u\dthe v-v\dthe u) +\frac12\dthe\left(\frac{1}{G}uv\dthe\log k\right)\\
&\qquad-\frac14\dthe\left(v(u^2-v^2)\drho \log G\right).
\end{split}
\end{equation}
Note that the equations  \eqref{e:rx} and \eqref{e:sx} are symmetrically hyperbolic for $\dthe u$ and $\dthe v$. By multiplying these equations by $\dthe u$ and $\dthe v$, respectively, we get
%Multiplying by $\dthe u$ and $\dthe v$, respectively, summing up the resulting equations, and then integrating over $I(\rho)$ with respect to $\theta$, we  have
\begin{equation}\label{e:ux-vx-eq}
\begin{split}
\frac12\frac{d}{d\rho}\int_{I(\rho)}(|\dthe u|^2+|\dthe v|^2)d\theta&= -\int_{I(\rho)}|\dthe u|^2\drho\log Gd\theta+\int_{I(\rho)}|\dthe v|^2\drho\log(kG)d\theta\\
&\qquad +\Psi_7+J_{10}+J_{11}+J_{12},
\end{split}
\end{equation}
where
\begin{align*}
\Psi_7&=\frac{1}{2}[(u_2^1)^2+(v_2^1)^2]\theta'_2(\rho)-\frac{1}{2}[(u_1^1)^2+(v_1^1)^2]\theta'_1(\rho)\\
&\quad -\frac{u}{4G}((\dthe u)^2+(\dthe v)^2)\bigg|_{\theta_1}^{\theta_2}+\frac{v}{2G}\dthe u\dthe v\bigg|_{\theta_1}^{\theta_2},\\
J_{10}&=-\int_{I(\rho)}u\dthe u\dthe\drho\log Gd\theta+\int_{I(\rho)}v\dthe v\dthe\drho\log(kG)d\theta,\\
J_{11}&=-\frac14\int_{I(\rho)}\frac{1}{G}\dthe\log G\big[u(\dthe u)^2+u(\dthe v)^2-2v\dthe u\dthe v\big]d\theta\\
&\qquad-\frac14\int_{I(\rho)}\frac{1}{G}\big[(\dthe u)^3-\dthe u(\dthe v)^2\big]d\theta,
\end{align*}
and
\begin{align*}
J_{12}&=\int_{I(\rho)}\frac{1}{G}\dthe\log G[(u\dthe u-v\dthe v)\dthe u+(u\dthe v-v\dthe u)\dthe v]d\theta\\
&\qquad+\frac12\int_{I(\rho)}\left[\dthe u\dthe\left(\frac{1}{G}v^2\dthe\log k\right)+\dthe v\dthe\left(\frac{1}{G}uv\dthe\log k\right)\right]d\theta\\
&\qquad-\frac14\int_{I(\rho)}\left[\dthe u\dthe\left(u(u^2-v^2)\drho \log G\right)+\dthe v\dthe\left(v(u^2-v^2)\drho \log G\right)\right]d\theta.
\end{align*}
By \eqref{e:initial-omega}, we have
\begin{equation*}
%	\label{e:psi1-estimate-uv}
	|\Psi_7|\leq 2\eps^2\rho^{-1-2\delta}.
\end{equation*}
By Lemma \ref{l:G-property} and the estimates of $\|u(\rho)\|$ and $\|v(\rho)\|$ in  Lemma \ref{l:rs-h0-estimate}, we have
\begin{equation}\label{e:j1-estimate-uv}
\begin{split}
|J_{10}|\leq\int_{I(\rho)}\frac{\Lambda}{\rho}(|u\dthe u|+|v\dthe v|)d\theta
\leq \frac{4\Lambda\Theta\eps R^{\delta/2}}{\rho^{1+\delta/2}}\sqrt{\|\dthe u\|^2+\|\dthe v\|^2}\,.
\end{split}
\end{equation}
For $J_{11}$, by \eqref{e:priori-c1-assumption-uv} and Lemma \ref{l:G-property}, we get
\begin{equation*}%\label{e:j2-estimate-uv}
|J_{11}|\leq\frac{C\eps}{\rho}\int_{I(\rho)}(|\dthe u|^2+|\dthe v|^2)d\theta.
\end{equation*}
As for $J_{12}$, we write
$$J_{12}=J_{12,1}+J_{12,2},$$
where
\begin{align*}
J_{12,1}&=\frac12\int_{I(\rho)}\dthe\left(G^{-1}\dthe\log k\right)\left[v^2\dthe u+uv\dthe v\right]d\theta\\
&\qquad-\frac14\int_{I(\rho)}\dthe\drho\log G\left[u(u^2-v^2)\dthe u+v(u^2-v^2)\dthe v\right]d\theta,
\end{align*}
and
\begin{align*}
J_{12,2}&=\int_{I(\rho)}\frac{1}{G}\dthe\log G\left[u(\dthe u)^2-2v\dthe u\dthe v+u(\dthe v)^2\right]d\theta\\
&\qquad+\frac12\int_{I(\rho)}\frac1G\dthe\log k\left[3v\dthe u\dthe v+u(\dthe v)^2\right]d\theta\\
&\qquad-\frac14\int_{I(\rho)}\drho\log G\left[(3u^2-v^2)(\dthe u)^2+(u^2-3v^2)(\dthe v)^2\right]d\theta.
\end{align*}
We note that $J_{12,1}$ consists of terms linear in $\dthe u$ and  $\dthe v$, and $J_{12,2}$ quadratic in $\dthe u$ and $\dthe v$.
By \eqref{e:priori-c1-assumption-uv} and Lemma \ref{l:rs-h0-estimate}, we have
$$|J_{12,1}|\le \frac{C\eps}{\rho}(\|u\|+\|v\|)(\|\dthe u\|+\|\dthe v\|)\leq \frac{C\eps^2R^{\delta/2}}{\rho^{1+\delta/2}}\sqrt{\|\dthe u\|^2+\|\dthe v\|^2},$$
and
$$|J_{12,2}|\le \frac{C\eps}{\rho}\int_{I(\rho)}(|\dthe u|^2+|\dthe v|^2)d\theta.$$
Thus,
\begin{equation}
\label{e:j3-estimate-uv}
|J_{12}|\leq \frac{C\eps}{\rho}\int_{I(\rho)}(|\dthe u|^2+|\dthe v|^2)d\theta+\frac{C\eps^2R^{\delta/2}}{\rho^{1+\delta/2}}\sqrt{\|\dthe u\|^2+\|\dthe v\|^2}.
\end{equation}

Set
$$E(\rho)=\sqrt{\|\dthe u(\rho)\|^2+\|\dthe v(\rho)\|^2}.$$
By substituting \eqref{e:liner-estimate} and \eqref{e:j1-estimate-uv}-\eqref{e:j3-estimate-uv} in \eqref{e:ux-vx-eq}, we obtain
\begin{align*}
\frac12\frac{d}{d\rho}E^2(\rho)
\leq&\left(-\frac{\delta}{\rho}+\frac{\va'}{\va}+\frac{C\eps}{\rho}\right)E^2(\rho)
+\frac{4\Lambda\Theta R^{\delta/2}\eps+C\eps^2R^{\delta/2}}{\rho^{1+\delta/2}}E(\rho)\\
&\quad +2\eps^2\rho^{-1-2\delta}+C\eps^3\rho^{-1-3\delta},
\end{align*}
and hence
\begin{align*}
\frac{d}{d\rho}\left(E^2(\rho)\frac{\rho^{4\delta/3}}{\va^2}\right)
\leq\frac{10\Lambda\Theta\eps R^{\delta/2}}{\rho^{1-\delta/6}\va}\frac{E(\rho)\rho^{2\delta/3}}{\va}+6\eps^2\rho^{-1-2\delta/3}\va^{-2},
\end{align*}
by taking $\eps$ sufficiently small. By applying Lemma \ref{l:gronwall}, we have
\begin{align*}
E(\rho)&\leq\frac{2R^{2\delta/3}\eps\va}{\rho^{2\delta/3}}+\frac{5\Lambda\Theta\eps R^{\delta/2}\va}{\delta \rho^{2\delta/3}}\int_R^\rho\frac{ds}{s^{1-\delta/6}\va}
+\frac{\va}{\rho^{2\delta/3}}\sqrt{\int_R^\rho\frac{6\eps^2ds}{s^{1+2\delta/3}\va^2}}\\
&\leq\frac{2\Lambda\eps R^{2\delta/3}}{\rho^{2\delta/3}}+\frac{30\Lambda^2\Theta\eps R^{\delta/2}}{\delta\rho^{\delta/2}}+\frac{3\Lambda\eps}{\delta\rho^{2\delta/3}}
\leq\frac{80\Theta^3\eps R^{\delta/2}}{\rho^{\delta/2}}.
\end{align*}
This implies \eqref{e:rs-h1-estimate}.

Similarly, we can prove \eqref{e:rs-h2-estimate} due to damping effects for both $\partial^2_\theta u$ and $\partial^2_\theta  v$.
\end{proof}

By combining Lemmas \ref{l:rs-h0-estimate}-\ref{l:rs-h1-estimate}, we obtain an $H^2$-estimate of $(u, v)$ with respect to $\theta$  for the case that $\rho^{1+\delta}k_*$ is decreasing.

\begin{proof}[Proof of Proposition \ref{p-priori}]
	Proposition \ref{p-priori} is a consequence of Lemmas \ref{l:u-ux-estimate}-\ref{l:uv-h2-estimate} if $\rho^{1+\delta}k_*$ is increasing and of Lemmas \ref{l:rs-h0-estimate}-\ref{l:rs-h1-estimate} if $\rho^{1+\delta}k_*$ is decreasing.
	\end{proof}

Now, we prove that $v$ has a positive lower bound.

\begin{lemma}\label{l:v-lower-bound}
For any $R'>R$, assume that $(u, v)\in C^1([R, R'], H^2(I(\rho)))$ is the  solution  to \eqref{eq-u-general}-\eqref{eq-v-general},
with the initial boundary data  \eqref{eq-initial}, and satisfies \eqref{e:priori-h2-assumption-uv}. Then, for any $(\theta, \rho)\in \Sigma_{R, R'}=\{(\theta, \rho):\rho\in[R, R'],\,\theta\in I(\rho)\}$, $ v(\theta, \rho)>0. $
\end{lemma}

\begin{proof}
	By \eqref{e:initial-omega} and integrating \eqref{e:v-rewrite} along the characteristic line $(X(\tau; \theta, \rho), \tau)$ over $[R, \rho]$,
	we obtain
	\begin{align*}
		v(X(\rho; \theta, \rho), \rho)&=v(X(R; \theta, \rho), R)\exp\left\{\int^\rho_R\mathcal{F}(X(\tau; \theta, \rho), \tau)d\tau\right\}\\
		&\geq v(X(R; \theta, \rho), R)e^{-C(\rho-R)M_{R'}}>0,
	\end{align*}
	where we used \eqref{e:priori-c1-assumption-uv} and \eqref{eq-a-bul} to bound
	$|\mathcal{F}|_0\leq CM_{R'}. $
\end{proof}

Now we can  conclude the following proposition on the global existence of solutions in $\tilde\Omega_2$.

\begin{proposition} \label{p:uv}
	Assume that all the conditions in Theorem \ref{thrm-main1} are fulfilled.
	Then, there exists a sufficiently large constant ${R}_*$ such that the Gauss-Codazzi system \eqref{eq-u-general}-\eqref{eq-v-general} with the data given in \eqref{eq-initial} admits a unique global smooth solution $(u, v)$ satisfying $v>0$ in the region $\tilde\Omega_2$, the image of $\Omega_2$ given in  \eqref{omega} with $R={R}_*$ under the coordinate transformation $F$  defined  in \eqref{F}.
\end{proposition}

\begin{proof}[Proof of Proposition \ref{p:uv}]
 From the local existence in Lemma \ref{l:local} and the {\it a priori} estimates in Proposition \ref{p-priori},  choosing  $R_*$ by  $\eps_1=\eps(R_*)$  (see \eqref{r_1}), we obtain the global existence of \eqref{eq-u-general} and \eqref{eq-v-general} with \eqref{eq-initial} for the solutions $(u,v)$ belonging to $C^1([R, \infty), H^2(I(\rho)))$. The Sobolev embedding implies  $u,\,v\in C^1(\Sigma_\infty)$. Moreover, by Lemma \ref{l:v-lower-bound}, we conclude $v>0$.
\end{proof}

With Proposition \ref{p:uv} established, we are able to obtain the solution of \eqref{e:tilde-p}-\eqref{e:tilde-q} in $\tilde\Omega_2$ with $\tilde q>\tilde p$ by setting
$$(\tilde p,\tilde q)=
\begin{cases}
	\frac12kG(u-v, v+u) &\text{ if }\rho^{1+\delta}k_* \text{ is increasing},\\
	\frac12(u-v, v+u) &\text{ if }\rho^{1+\delta}k_* \text{ is decreasing}.
\end{cases}$$
We will transform $(\tilde{p},\, \tilde{q})$ back to $(p,\, q)$, the solution of \eqref{e:p}-\eqref{e:q} in $\Omega_2$, by the transformation
\[\frac{p}{B}=\tilde F^{-1}_*\Big(\frac{\tilde p}{G}\Big)=\frac{-\theta_t+\frac{\tilde p}{G}\rho_t}{\theta_x-\frac{\tilde p}{ G}\rho_x},\quad
\frac{q}{B}=\tilde F^{-1}_*\Big(\frac{\tilde q}{G}\Big)=\frac{-\theta_t+\frac{\tilde q}{G}\rho_t}{\theta_x-\frac{\tilde q}{G}\rho_x}.
\]
The following lemma ensures that the inverse transformation is applicable.

\begin{lemma}\label{l:transform-2}
	Assume that all the conditions in Theorem \ref{thrm-main1} are fulfilled. Let $(\tilde{p},\, \tilde{q})$ be the solution obtained in Proposition \ref{p:uv} in the domain $\tilde\Omega_2$. Then,
	\begin{equation}\label{e:boundary-tilde-pq}
		\theta_x-\frac{\tilde{p}}{G}\rho_x\geq\frac{B}{2G}, \quad \theta_x-\frac{\tilde{q}}{G}\rho_x\geq\frac{B}{2G}.
	\end{equation}
\end{lemma}

\begin{proof}
	By \eqref{thetat}, we have
	\begin{align*}
		\theta_x-\frac{\tilde{p}}{G}\rho_x=\frac{B}{G}\tanh\Phi-\frac{\tilde{p}}{G}\frac{\xi B}{\cosh\Phi}=\frac{B}{G}\left(\tanh\Phi-\frac{\xi\tilde{p}}{\cosh\Phi}\right).
	\end{align*}
	By Proposition \ref{p:uv}, we get $|\tilde{p}|, |\tilde{q}|\leq C$.  Note that
	$\rho\geq R$ and $e^\Phi\geq R+1$ in $\tilde\Omega_2.$ Hence,
	\begin{align*}
		\tanh\Phi-\frac{\xi\tilde{p}}{\cosh\Phi}\geq 1-|1-\tanh\Phi|-\frac{|\tilde{p}|}{\cosh\Phi}\geq1-\frac{C}{R}\geq\frac{1}{2}.
	\end{align*}
	Thus, we have the first inequality in \eqref{e:boundary-tilde-pq}.
	We can verify the second one similarly.
\end{proof}

Finally, we are able to finish the proof of Theorem \ref{thrm-main1} as follows.

\begin{proof}[Proof of Theorem \ref{thrm-main1}]
Let $\Omega_1$ and $\Omega_2$ be given by \eqref{omega}, for $R$ to be determined.
By Lemma \ref{l:omega-1}, there exists a smooth solution $(p,\, q)$ to \eqref{e:p}-\eqref{e:q}
with $q>p$ in $\Omega_1$, for any $R>0$.
By Lemma \ref{l:transform-2}, for $R$ sufficiently large, we obtain the smooth solution
$(p,\, q)=\big(B\tilde F^{-1}_*(\frac{\tilde p}{G}),\, B\tilde{F}^{-1}_*(\frac{\tilde{q}}G)\big)$ of \eqref{e:p}-\eqref{e:q} in the closure of $\Omega_2$ with $q>p$ (since $\tilde q>\tilde p$).
Hence, we have a smooth solution $(p,\, q)$ to \eqref{e:p}-\eqref{e:q} with $q>p$
in $\Omega_1\cup\overline{\Omega}_2=\R_+^2$.
Similarly, we can establish a smooth solution $(p,\, q)$ to \eqref{e:p}-\eqref{e:q} with $q>p$ in $\mathbb R_-^2$ and form a smooth solution in $\mathbb R^2$.  Equivalently, we have a smooth solution to the Gauss-Codazzi system, which yields a smooth isometric immersion of  $(\mathcal{M}, g)$ into $\R^3$ by the fundamental theorem of the surface theory. This completes the proof of Theorem \ref{thrm-main1}.
\end{proof}

\section*{Acknowledgments}
W. Cao was supported in part by the National Natural Science Foundation of China (No. 12471224). 
Q. Han was supported in part by   the National Science Foundation under  grant DMS-2305038.
F. Huang was supported in part by the National Key R\&D Program of China 2021YFA1000800 and the National Natural Science Foundation of China (No. 12288201). 
D. Wang was supported in part by the National Science Foundation under grants DMS-1907519 and DMS-2219384.


\begin{thebibliography}{999}
	\bibitem{ac2005}
	A. Ascanellia, M. Cicognani, Energy estimate and fundamental solution for
	degenerate hyperbolic Cauchy problems,
	{\it J. Differential Equations}, {\bf217}(2005), 305-340.
		
	\bibitem{bgy}
   R. Bryant, P. Griffiths, and D. Yang,
   Characteristics and existence of isometric embeddings, {\it Duke Math. J.} {\bf 50} (1983), 893-994.
	
	
	\bibitem{BS}
	Y. D. Burago, S. Z. Shefel,  The geometry of surfaces in
	Euclidean spaces, Geometry III, 1-85, Encyclopaedia Math. Sci.,
	48, Burago and Zalggaller (Eds.), Springer-Verlag: Berlin, 1992.
	
	\bibitem{CHW}
	W. Cao, F.  Huang  and D. Wang,
	Isometric immersions of surfaces with two classes of metrics and negative Gauss curvature,
	{\it Arch. Ration. Mech. Anal.} {\bf 218} (2015), no. 3, 1431-1457.
	
	\bibitem{CHW1}
	W. Cao, F. Huang and D.  Wang,
	Isometric immersion of surface with  negative Gauss curvature and the Lax-Friedrichs scheme,
	{\it SIAM J. Math. Anal.} {\bf 48} (2016), no. 3, 2227-2249.
	
	\bibitem{CHW-smooth}
	W.  Cao, F.  Huang and D. Wang,
	Isometric immersion of complete surfaces with slowly decaying negative Gauss curvature,
	{\it ArXiv}:1605.09491, 2016.

    \bibitem{CS19}
    W. Cao, L. Sz\'ekelyhidi Jr., Global Nash-Kuiper theroem for compact manifolds, {\it J. Differential Geom.,} {\bf122} (2022), no. 1, 35-68.
	
	\bibitem{Cartan}
	E. Cartan, Sur la possibilit\'e de plonger un espace Riemannian
	donn\'e dans un espace Euclidien,
	{\it Ann. Soc. Pol. Math.} {\bf 6} (1927), 1-7.
	
	\bibitem{CCSWY1}
	G.-Q. Chen, J. Clelland, M. Slemrod, D. Wang, and D. Yang,
	{  Isometric embedding via strongly symmetric positive systems}, {\it Asian J. Math.} {\bf 22} (2018), no. 1, 1-40.
	
	\bibitem{Chen-Li2018} G.-Q. Chen, S. Li,
Global weak rigidity of the Gauss-Codazzi-Ricci equations and isometric immersions of Riemannian manifolds with lower regularity,
{\it J. Geom. Anal.}  {\bf 28} (2018), no. 3, 1957-2007.

	\bibitem{CSW}
	G.-Q. Chen, M. Slemrod and D.  Wang,
	{  Isomeric immersion and compensated compactness}, {\it Commun. Math. Phys.} {\bf 294} (2010), 411-437.
	
	\bibitem{Christoforou}
	C. Christoforou,
	BV weak solutions to Gauss-Codazzi system for isometric immersions,
	{\it J. Differential Equations} {\bf 252} (2012), 2845-2863.
	
	\bibitem{cs} C. Christoforou, M. Slemrod, Isometric immersions via compensated compactness for slowly decaying negative Gauss curvature and rough data,
     {\it Z. Angew. Math. Phys.} {\bf  66} (2015),   3109-3122.
	
	\bibitem{CDS12}
 S. Conti,  C., De Lellis, and L. Sz\'ekelyhidi Jr,  $h$-principle and rigidity for
$C^{1, \alpha}$ isometric embeddings, {\it Nonlinear partial differential equations, The Abel Symp. 2010}, Springer-Verlag, 2012, pp. 83-116

	
	%\bibitem{Codazzi}
	%D. Codazzi, Sulle coordinate curvilinee duna superficie dello
	%spazio, {\it Ann.\ Math.\ Pura Applicata,} { \bf 2} (1860),  101-119.
	
	\bibitem{Carmo}
	M. P. do Carmo, {Riemannian geometry.} Translated from
	the second Portuguese edition by Francis Flaherty, Mathematics: Theory \&
	Applications. Birkh\"{a}user Boston, Inc., Boston, MA, 1992.
	
	\bibitem{CS-acta-82}
	F. Colombini, S. Spagnolo,
	An example of a weakly hyperbolic Cauchy problem not well posed in $C^\infty$, {\it Acta Math.} {\bf148} (1982), 243-253.

   \bibitem{dd01}
   P. D'Ancona, M. Di Flaviano, On a weakly hyperbolic quasilinear mixed problem of second order,
   {\it Ann. Scuola Norm. Sup. Pisa Cl. Sci.(4)} {\bf 30} (2001), no. 2, 251-267.
	
    \bibitem{DI20}
C. De Lellis, D. Inauen, $C^{1,\alpha}$ Isometric embeddings of polar caps,  {\it Adv. Math.}
{\bf 363} (2020), 106996, 39 pp.
%{\it arXiv preprint} {1809.04161v1} (2018)


\bibitem{DIS15}
C. De Lellis, D. Inauen, and L. Sz\'ekelyhidi Jr. A Nash-Kuiper theorem for $C^{1, 1/5-\delta}$ immersions on surfaces in 3 dimensions, {\it Rev. Mat. Iberoamericana} {\bf34} (2018), 1119-1152.
	
	\bibitem{Dong}
	G.-C. Dong, The semi-global isometric imbedding in $\R^3$ of two-dimensional Riemannian manifolds with Gauss curvature changing sign cleanly,
	{\it J. Partial Differential Equations} {\bf 6} (1993), 62-79.

\bibitem{dr97}	
M. Dreher, M. Reissig, Weakly hyperbolic equations-a modern field in the theory of hyperbolic equations,  Partial differential and integral equations (Newark, DE), 1997, 303–318,  Int. Soc. Anal. Appl. Comput., 2, Kluwer Acad. Publ., Dordrecht, 1999.

	\bibitem{Efimov1}
	N. V. Efimov, Generation of singularities on surfaces of negative curvature (Russian),  {\it Mat. Sb. (N.S.)} {\bf64} (106) (1964), 286-320.
	
	\bibitem{Efimov2}
	N. V. Efimov, Surfaces with slowly varying negative curvature,
	{\it Russian Math. Survey} {\bf 21} (1966), 1-55.

    \bibitem{evans}
    L.C. Evans, Partial differential equations. Graduate Studies in Mathematics 19,  American Mathematical Society, Providence, RI, 1998. xviii+662 pp.
     
    % \bibitem{figke}
    % A. Figalli, C. Kehle, On the prescribed negative Gauss curvature problem for graphs. 
    % {\it Discrete Contin. Dyn. Syst.} {\bf 43} (2023), no. 3-4, 1420-1435.

    \bibitem{fried}
    K.O. Friedrichs, Symmetric hyperbolic linear differential equations, {\it Comm. Pure Appl. Math.} {\bf7} (1954), 345-392.
	
	\bibitem{GY}
	J. B. Goodman,  D. Yang,
	{Local solvability of nonlinear partial differential equations of real principal type},
	Preprint, 1988.
	
	\bibitem{gromov70} M. Gromov, V.A. Rokhlin,  Embeddings and immersions in Riemannian geometry,
	{\it Uspekhi Mat. Nauk.} {\bf 25} (1970), no. 5, 3-62; Russian Math. Survey, 25(1970), no. 5, 1-57.

    \bibitem{gromov86} M. Gromov, Partial Differential Relations, {Springer-Verlage, Berlin Heidelberg}, 1986.
	
	\bibitem{GuLi}
	C. Gu, T. Li, S. Chen, S. Zheng, and Y. Tan, Equations in Mathematics and Physics, Higher Education Press, 2016, In Chinese.
	
	\bibitem{Guan}
	B. Guan, P. Guan,
	Convex hypersurfaces of prescribed curvatures, {\it Ann. of Math. (2)}, {\bf 156} (2002), no. 2, 655-673.
	
	\bibitem{gl} P. Guan, Y. Li, The Weyl problem with nonnegative Gauss curvature, {\it J. Diff. Geometry}
	{\bf 39} (1994), 331-342.
	
   \bibitem{Gun89}
    M. G\"unther, On the perturbation problem associated to isometric embeddings of Riemannian manifolds, {\it Ann. Global Anal. Geom.} {\bf 7} (1989), no. 1, 69-77.

   \bibitem{Gun90}
   M. G\"unther, Isometric embeddings of Riemannian manifolds, Proceedings of the International Congress of Mathematicians (Kyoto, 1990), Mathematical Society of Japan, 1991, pp. 1137-1143. Math. Soc. Japan, Tokyo, 1991.

	\bibitem{Han}
	Q. Han,  On isometric embedding of surfaces with Gauss curvature changing
	sign cleanly, {\it Comm. Pure Appl. Math.} { \bf 58} (2005), 285-295.

     \bibitem{han10} Q. Han, Energy estimates for a class of degenerate hyperbolic equations, {\it Math. Ann.} {\bf 347} (2010), 339-364.
	
	\bibitem{HH}
	Q. Han, J.-X. Hong,   Isometric embedding of Riemannian manifolds in Euclidean
	spaces. Providence, RI: Amer. Math. Soc., 2006.
	
	\bibitem{hhl}   Q. Han, J.-X. Hong and C.-S. Lin,
	Local isometric embedding of surfaces with nonpositive Gauss curvature, {\it J. Differential Geom.} {\bf  63} (2003), 475-520.
	
	\bibitem{HK1}
	Q. Han, M.A. Khuri, On the local isometric embedding in $\R^3$ of surfaces with Gaussian curvature of mixed sign,
    {\it Comm. Anal. Geom.} {\bf 18} (2010), no. 4, 649-704.
	
	\bibitem{HK2}
	Q. Han, M.A. Khuri, The linearized system for isometric embeddings and its characteristic variety, {\it Adv. Math.} {\bf 230} (2012), no. 1, 263-293.
	
	\bibitem{hw} P. Hartman, P. Winter,  Gauss curvature and local embedding, {\it Amer. J. Math.} {\bf 73} (1951), 876-884.
	
	\bibitem{he} E. Heinz,  On Weyl's embedding problem, {\it J. Math. Mech.} {\bf  11} (1962), 421-454.
	
	\bibitem{hi} D. Hilbert,  Ueber flachen von constanter Gausscher Krummung, {\it Trans. Amer. Math. Soc.}
	{\bf 2} (1901), 87-99.
	
%	\bibitem{H-07}
%	F. Hirosawa,
%	On the asymptotic behavior of the energy for the wave equations with time depending coefficients. {\it Math. Ann.} {\bf339}(2007), no. 4, 819-838.
	
	
	\bibitem{H}
	J.-X. Hong,   Realization in $\mathbb{R}^3$ of complete Riemannian manifolds
	with negative curvature, {\it Commun. Anal. Geom.} {\bf 1} (1993), 487-514.

    \bibitem{h-private}
    J.-X. Hong, Private communication, 2011.
	
	\bibitem{hz} J.-X. Hong, C. Zuily, Isometric embedding of the 2-sphere with nonnegative curvature in $\mathbb{R}^3$, {\it Math. Z.} {\bf 219} (1995), 323-334.
	
	
	\bibitem{Janet}
	M. Janet, Sur la possibilit\'e de plonger un espace Riemannian
	donn\'e dans un espace Euclidien,
	{\it  Ann. Soc. Pol. Math.} {\bf 5} (1926),
	38-43.
	
	\bibitem{Khuri}
	M.A.  Khuri, The local isometric embedding $\R^3$ of two-dimensional Riemannian manifolds with Gaussian curvature changing sign to finite order on a curve,
	 {\it J. Differential Geom.} {\bf 76} (2007), no. 2, 249-291.

\bibitem{Kui55}
N. Kuiper. On $C^1$ isometric embeddings I, II, {\it Proc. Kon. Acad. Wet. Amsterdam A} {\bf 58} (1955), 545-556, 683-689.

	\bibitem{LiS2020}  S. Li,
On the existence of $C^{1,1}$-isometric immersions of several classes of negatively curved surfaces into $\R^3$,
{\it Arch. Ration. Mech. Anal.} {\bf 236} (2020), no. 1, 419-449.
	
	\bibitem{Lin0} C.-S. Lin,  The local isometric embedding in $\R^3$ of 2-dimensional Riemannian manifolds
	with nonnegative curvature, {\it J. Diff. Geometry} {\bf  21} (1985), 213-230.
	
	\bibitem{Lin}
	C.-S. Lin, The local isometric embedding in $\R^3$ of 2-dimensional Riemannian
	manifolds with Gaussian curvature changing sign cleanly,
	{\it Comm. Pure Appl. Math.} {\bf 39} (1986), 867-887.
	
	
	\bibitem{Mainardi}
G. Mainardi,  { Su la teoria generale delle superficie},  {\it Giornale
dell' Istituto Lombardo} {\bf9} (1856), 385--404.

\bibitem{M1}
S. Mardare, The foundamental theorem of theorey for surfaces with little regularity,
{\it J. Elasticity} {\bf 73} (2003) 251-290.

%\bibitem{M2}
%S. Mardare, On Pfaff systems with $L^p$ coefficients and their applications in
%differential geometry.
%{\it J. Math. Pure Appl.} { 84}(2005), 1659-1692.
	
	
	
	\bibitem{NM}
	G. Nakamura, Y. Maeda,  {Local isometric embedding problem of Riemannian $3$-manifold into $\R\sp 6$},
	{\it Proc. Japan Acad. Ser. A Math. Sci.} {\bf62} (1986), no. 7, 257-259.
	
	\bibitem{NM2}
	G. Nakamura, Y.  Maeda, {Local smooth isometric embeddings of low-dimensional Riemannian manifolds into Euclidean spaces}, {\it Trans. Amer. Math. Soc.}
   {\bf 313}(1989), no. 1, 1-51.
	
	\bibitem{nash1954}
	J. Nash, $C^1$ isometric imbeddings,
	{\it Ann. of Math.} (2) {\bf60} (1954), 383-396.
	
	
	\bibitem{nash1956}
	J. Nash, The imbedding problem for Riemannian manifolds, {\it Ann. of Math.}
	(2) {\bf 63} (1956), 20-63.
	
	\bibitem{n} L. Nirenberg, The Weyl and Minkowski problems in differential geometry in the large,
	{\it Comm. Pure Appl. Math.} {\bf 6} (1953), 337-394.

    \bibitem{ol70}
	O. A. Ole\v{i}nik, On the Cauchy problem for weakly hyperbolic equations, {\it Commun. Pure Appl. Math.} {\bf23} (1970), 569-586.
	%\bibitem{Peterson}
	% K. M.  Peterson,  \"Uber die Biegung der Fl\"achen,
	% {Dorpat. Kandidatenschrift} 1853.
	
	\bibitem{Poole}
	T. E. Poole, {The local isometric embedding problem for 3-dimensional Riemannian manifolds
		with cleanly vanishing curvature},
	{\it Comm. in  Partial Differential Equations} {\bf35} (2010), 1802-1826.
	
	\bibitem{PS}
	\`{E}. G. Poznyak,  E. V.  Shikin,  Small parameters in the
	theory of isometric imbeddings of two-dimensional Riemannian
	manifolds in Euclidean spaces. In:  Some Questions of
	Differential Geometry in the Large, {Amer. Math. Soc. Transl. Ser.}
	2, { \bf 176} (1996), 151-192, AMS: Providence, RI.
	
%	\bibitem{RS-05}
%	M. Reissig, J. Smith, $L^p �CL^q$ estimate for wave equation with bounded time dependent coefficient. {\it Hokkaido Math. J.} {\bf 34}(2005), 541�C586.
%	
%	
%	\bibitem{RY-os-99}
%	M. Reissig, K. Yagdjian, Weakly hyperbolic equation with fast oscillating coefficients. {\it Osaka J. Math.}  {\bf36} (1999), no. 2, 437-464.
%	
%	\bibitem{RY-00}
%	M. Reissig, K. Yagdjian,
%	About the influence of oscillations on Strichartz-type decay estimates. {\it Partial differential operators (Torino, 2000).
%		Rend. Sem. Mat. Univ. Politec. Torino} {\bf 58} (2000), no. 3, 375-388 (2002).
	
	\bibitem{Rozh}
	B. Rozhdestvensk\u\i, Quasilinear hyperbolic systems for theory of surfaces, {\it Dokl. Akad. Nauk.
	SSSR (N.S.)}, {\bf143} (1962), 50-52; {\it Soviet Math. Dokl.}, {\bf3}(1962), 351-353.
	
	\bibitem{Smoller} J. Smoller, Shock Waves and Reaction-Diffusion Equations, Springer, NewYork, 1994.
    \bibitem{sz}  M. Szopos, An existence and uniqueness result for isometric
immersions with little regularity, {\it Rev. Roumaine Math. Pures Appl.} {\bf53} (2008), no. 5-6, 555–565
	
	\bibitem{w} H. Weyl, Uber die Bestimmheit einer geschlossenen konvex Flache durch ihr Linienelement,
	{\it Vierteljahresschrift der nat.-Forsch.  Ges. Zurich} {\bf 61} (1916), 40-72.

    \bibitem{Wh44}
H. Whitney, The self-intersections of a smooth $n$-manifold in $2n$-space,
{\it Ann. of Math. (2)}  {\bf 45} (1944), 220-246.
	
	\bibitem{Yau}
	S.-T. Yau, Seminar on Differential Geometry, {Princeton University Press,} 1982.
	
%	\bibitem{Y-80}
%	K. Yagdjian, On correctness of the Cauchy problem for non-strictly hyperbolic equations, {\it Soviet J. of Contemporary Math. Anal.} {\bf15} (1980) 6, 54-65.
\end{thebibliography}
\end{document}